\documentclass[a4paper,oneside,11pt]{article}
\usepackage{hyperref}
\usepackage{amsmath,amsfonts,amscd,amssymb,amsthm}
\usepackage[margin=2cm]{geometry}
\usepackage{longtable}
\usepackage[english]{babel}
\usepackage[utf8]{inputenc} 
\usepackage[active]{srcltx}
\usepackage[T1]{fontenc}
\usepackage{graphicx}
\usepackage{pstricks}
\usepackage{bbm}
\usepackage{enumitem}
\usepackage{MnSymbol}
\usepackage{stmaryrd}
\usepackage{nicefrac}
\usepackage{calrsfs}
\usepackage{mathtools}
\usepackage{mdwlist}   
\usepackage{epstopdf} 

\newtheorem{theorem}{Theorem}[section]

\newtheorem*{claim*}{Claim}

\newtheorem{corollary}[theorem]{Corollary}
\newtheorem{lemma}[theorem]{Lemma}
\newtheorem{proposition}[theorem]{Proposition}

\newtheorem{remark}[theorem]{Remark}

\newcommand{\be}[1]{\begin{equation}\label{#1}}
\newcommand{\ee}{\end{equation}}
\numberwithin{equation}{section}

\newcommand{\ba}[1]{\begin{align}\label{#1}}
\newcommand{\ea}{\end{align}}
\numberwithin{equation}{section}

\newcommand{\ben}{\begin{equation*}}
\newcommand{\een}{\end{equation*}}
\numberwithin{equation}{section}

\renewenvironment{proof}[1][\relax]
  {\paragraph{Proof\ifx#1\relax\else~of #1\fi}}%
  {~\hfill$\square$\par\bigskip}

\newcommand{\icomp}{{\rm i}}

\newcommand{\calH}{\mathcal{H}}

\newcommand{\calS}{\mathcal{S}}

\newcommand{\bbH}{\mathbb{H}}

\newcommand{\bbL}{\mathbb{L}}

\newcommand{\bbN}{\mathbb{N}}

\newcommand{\bbR}{\mathbb{R}}

\newcommand{\bbT}{\mathbb{T}}

\newcommand{\bbZ}{\mathbb{Z}}

\newcommand{\de}{\delta}
\newcommand{\ep}{\varepsilon}
\newcommand{\eps}{\ep}
\newcommand{\om}{\omega}
\newcommand{\Om}{\Omega}
\newcommand{\si}{\sigma}
\newcommand{\la}{\lambda}
\newcommand{\La}{\Lambda}
\newcommand{\De}{\Delta}

\newcommand{\bp}{\mathbf{p}}

\newcommand{\Tr}{{\rm Tr}}

\newcommand{\rk}[1]{\bgroup\color{red}%
  \par\medskip\hrule\smallskip%
  \noindent\textbf{#1}%
  \par\smallskip\hrule\medskip\egroup}

\usepackage{wasysym}

\newcommand{\bfR}{\mathsf{R}}

\newcommand{\pd}{\partial}

\newcommand{\ind}{\mathbf{1}}

\usepackage{psfrag}
\def\mik{1}
\newcommand\cpsfrag[2]{\ifnum\mik=1\psfrag{#1}{#2}\fi}

\newcommand{\sym}{{\rm sym}}

\renewcommand{\ba}{{\mathbf a}}

\renewcommand{\icomp}{i}

\title{Discontinuity of the phase transition for the planar random-cluster and Potts models with $q>4$}
\author{
	Hugo Duminil-Copin\thanks{Institut des Hautes \'Etudes Scientifiques}\ \thanks{Universit\'e de Gen\`eve}, 
	Maxime Gagnebin\addtocounter{footnote}{-1}\footnotemark, 
	Matan Harel\addtocounter{footnote}{-2}\footnotemark, 
	Ioan Manolescu\addtocounter{footnote}{1}\thanks{University of Fribourg}, 
	Vincent Tassion\addtocounter{footnote}{-2}\footnotemark}
\date{\today}

\begin{document}

\maketitle

\begin{abstract}
We prove that the $q$-state Potts model and the random-cluster model with cluster weight $q>4$ undergo a discontinuous phase transition on the square lattice. More precisely, we show
\begin{enumerate}
\item Existence of multiple infinite-volume measures for the critical Potts and random-cluster models,
\item Ordering for the measures with monochromatic (resp.~wired) boundary conditions for the critical Potts model (resp.~random-cluster model), and
\item Exponential decay of correlations for the measure with free boundary conditions for both the critical Potts and random-cluster models.
\end{enumerate}
The proof is based on a rigorous computation of the Perron-Frobenius eigenvalues of the diagonal blocks of the transfer matrix of the six-vertex model,
whose ratios are then related to the correlation length of the random-cluster model. 

As a byproduct, we rigorously compute the correlation lengths of the critical random-cluster and Potts models, and show that they behave as~$\exp(\pi^2/\sqrt{q-4})$ as~$q$ tends to~4.
\end{abstract}

\section{Introduction}

\subsection{Motivation} 

Lattice spin models were introduced to describe specific experiments;
they were later found to be illustrative of a large variety of physical phenomena.
Depending on a parameter (most commonly temperature), they exhibit different macroscopic behaviours (also called phases), 
and phase transitions between them. 
Phase transitions may be continuous or discontinuous, and determining their type 
is one of the first steps towards a deeper understanding of the model.

In recent years, the Potts and random-cluster models have been the object of revived interest after new rigorous results were proved. In \cite{BefDum12}, the critical points of the models were determined for any $q\ge1$. In \cite{DumSidTas16}, the models were proved to undergo a continuous phase transition for $1 \le q\le 4$, thus proving half of a famous prediction by Baxter. The object of this paper is to prove the second half of his prediction - namely, that the phase transition is discontinuous when $q>4$.

\subsection{Results for the Potts model} 

The Potts model was introduced by Potts \cite{Pot52} following a suggestion of his adviser Domb. While the model received little attention early on, it became the object of great interest in the last 50 years. Since then, mathematicians and physicists have been studying it intensively, and much is known about its rich behaviour, especially in two dimensions. For a review of the physics results, see \cite{Wu82}.

In this paper, we will focus on the case of the square lattice $\bbZ^2$ composed of vertices $x=(x_1,x_2)\in\bbZ^2$, and edges between nearest neighbours. In the $q$-state ferromagnetic Potts model (where $q$ is a positive integer larger than or equal to $2$), each vertex of a graph receives a {\em spin} taking value in $\{1,\dots,q\}$. The energy of a configuration is then proportional to the number of neighbouring vertices of the graph having different spins. 
Formally, the Potts measure on a finite subgraph $G = (V,E)$ of the square lattice, at inverse temperature $\beta > 0$ and boundary conditions $i\in\{0,1,\dots,q\}$, is defined for every $\sigma\in\{1,\dots,q\}^V$ by the formula
\begin{equation}\label{eq:Gibbs}
\mu_{G,\beta}^i[\sigma]:=\frac{\displaystyle\exp[-\beta {\bf H}^i_{G}(\sigma)]}{\displaystyle\sum_{\sigma'\in\{1,\dots,q\}^V}\exp[-\beta {\bf H}_{G}^i(\sigma')]},
\end{equation}
where 
\begin{equation*}
	{\bf H}_{G}^i(\sigma):=-\sum_{\{x,y\}\in E}\ind[\sigma_x=\sigma_y]- \,\sum_{x\in \partial V}\ind[\sigma_x=i]. 
\end{equation*}
Above, $\ind[\cdot]$ denotes the indicator function and $\partial V$ is the set of vertices of $G$ with at least one neighbour (in $\bbZ^2$) outside of $G$. Note that when $i =0$, the second sum is zero for all $\sigma$.

For any boundary conditions $i$, the family of measures $\mu_{G,\beta}^i$ converges as $G$ tends to the whole square lattice. The resulting measure $\mu_\beta^i$ defined on the square lattice is called the Gibbs measure with {\em free} boundary conditions if $i=0$ (respectively, {\em monochromatic} boundary conditions equal to $i$ if $i\in\{1,\dots,q\}$). 

The Potts model undergoes an order/disorder phase transition, meaning that there exists a {\em critical inverse temperature} $\beta_c=\beta_c(q)\in(0,\infty)$) such that:
\begin{itemize}[noitemsep,nolistsep]
	\item For $\beta<\beta_c$, the measures $\mu_\beta^i$, $i=0,\dots,q$, are all equal. 
	\item For $\beta>\beta_c$, the measures $\mu_\beta^i$, $i=0,\dots,q$, are all distinct. 
\end{itemize}
Baxter \cite{Bax78} conjectured that the phase transition is continuous if $q\le 4$ and discontinuous if $q>4$, meaning that all the measures $\mu^i_{\beta_c}$ with $i=0,\dots,q$ are equal if and only if $q\le 4$. 
It was shown in  \cite{BefDum12} that $\beta_c= \log(1+\sqrt q)$;
moreover, when $q \leq 4$, it was proved in \cite{DumSidTas16} that the phase transition is indeed continuous, along with more detailed properties of the unique critical measure~$\mu_{\beta_c}$. 
The goal of this article is to complete the proof of Baxter's conjecture by proving the following theorem. Below, $x_n$ denotes the site of $\bbZ^2$ with both coordinates equal to $\lfloor n/2\rfloor$.
\begin{theorem}\label{thm:Potts}
	Consider the $q$-state Potts model on the square lattice with $q>4$. Then, 
	\begin{enumerate}[nolistsep,itemsep = 2pt]
	\item all the measures $\mu^i_{\beta_c}$ for $i=0,\dots,q$ are distinct and ergodic 
	(in particular, $\mu_{\beta_c}^0$ is not equal to the average of the $\mu_{\beta_c}^i$ with $i\in\{1,\dots,q\}$);
	\item for any $i\in\{1,\dots,q\}$, $\mu_{\beta_c}^i[\sigma_0=i]>\tfrac1q$.
	\item Let $\lambda>0$ satisfy $\cosh(\lambda)=\sqrt q/2$. Then
	\begin{align*}
		\lim_{n\rightarrow\infty} -\tfrac1n\log\big(\mu_{\beta_c}^0[\sigma_0=\sigma_{x_n}]-\tfrac1q\big)=
		\lambda + 2 \sum_{k=1 }^\infty \tfrac{(-1)^k}k \tanh(k\lambda).
	\end{align*}
	Furthermore, the quantity above is strictly positive.
	\end{enumerate}
\end{theorem}
The limit computed in the final item above is the inverse correlation length of the critical Potts model in the diagonal direction. This theorem follows directly from Theorem~\ref{thm:RCM} below via the standard coupling between the Potts and random-cluster models (see Section~\ref{sec:Potts} for details).

\begin{figure}
	\begin{center}
	\includegraphics[width=0.27\textwidth]{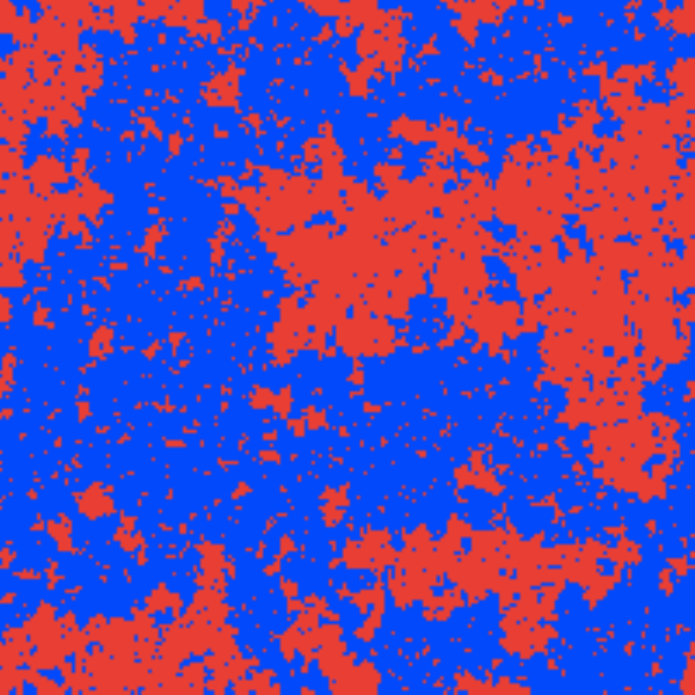}\quad
	\includegraphics[width=0.27\textwidth]{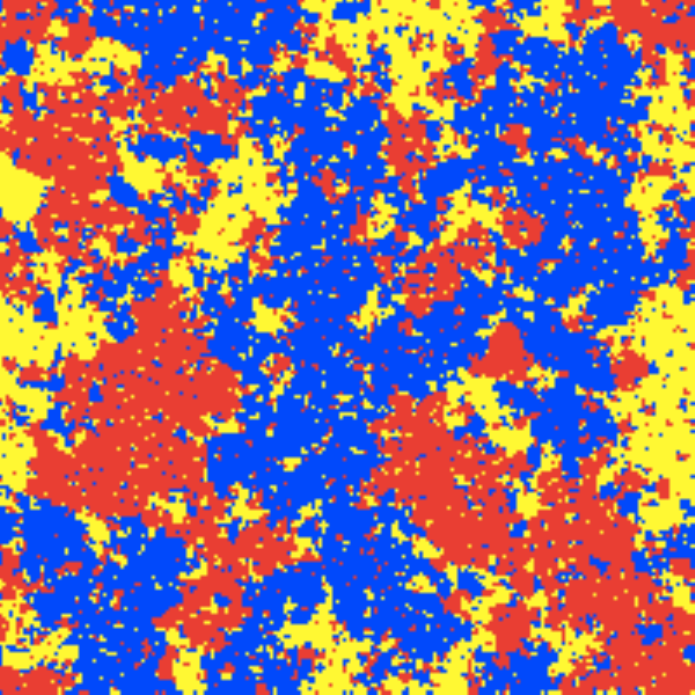}\quad
	\includegraphics[width=0.27\textwidth]{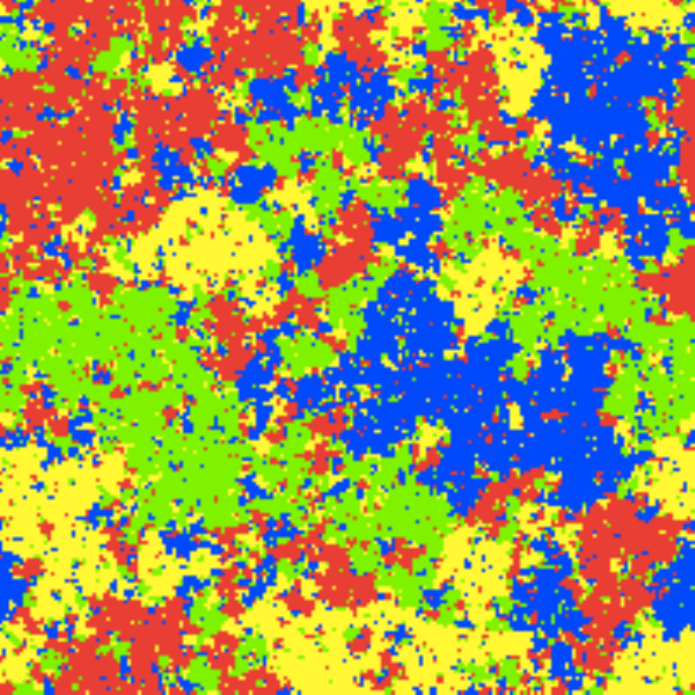}\vspace{6pt}
	
	\includegraphics[width=0.27\textwidth]{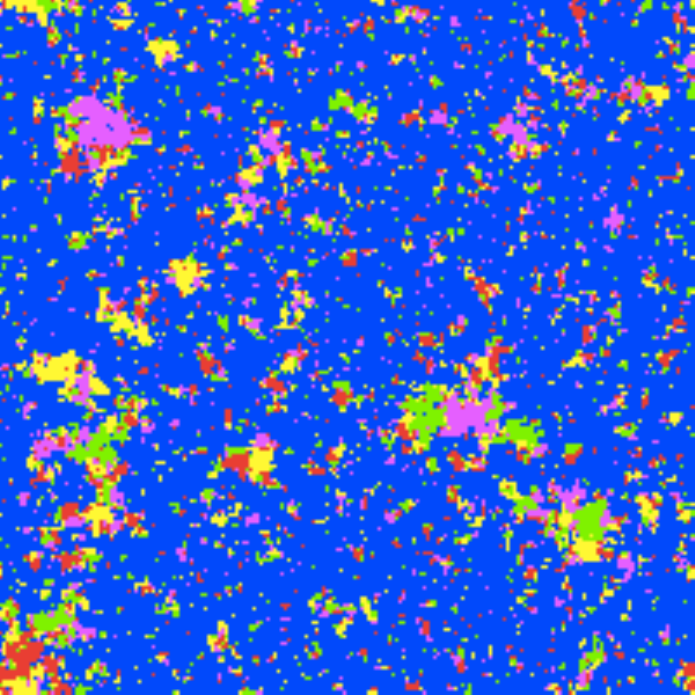}\quad
	\includegraphics[width=0.27\textwidth]{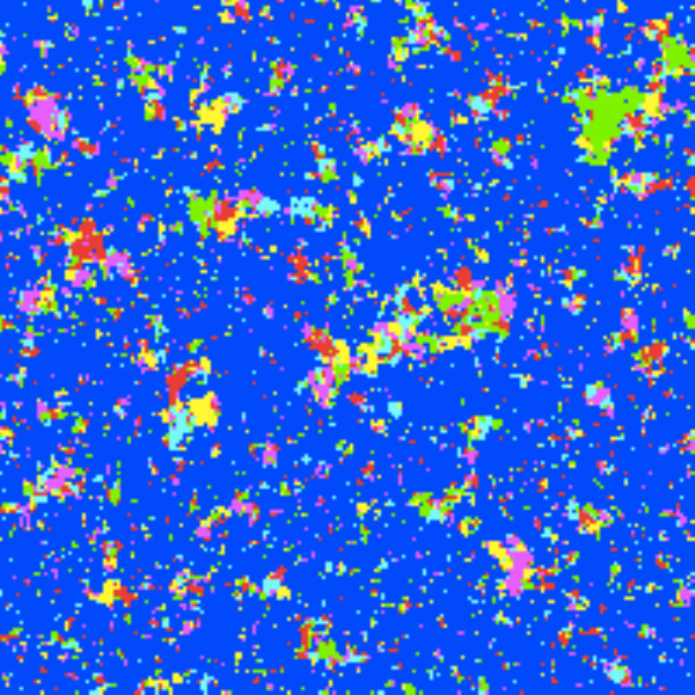}\quad
	\includegraphics[width=0.27\textwidth]{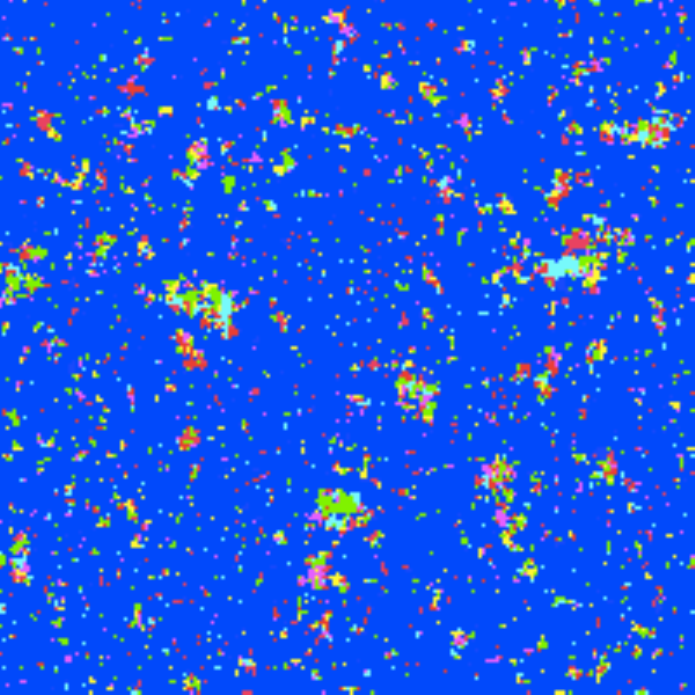}
	\caption{Simulations (courtesy of Vincent Beffara) of the critical planar Potts model $\mu_{\beta_c}^1$ (the spin 1 is depicted in blue) with $q$ equal to $2$, $3$, $4$, $5$, $6$ and $9$ respectively. The behaviour for $q\le 4$ is clearly different from the behaviour for $q>4$. In the first three pictures, each spin seems to play the same role, while in the last three, the blue spin dominates the other ones.}
	\end{center}
\end{figure}

\subsection{Results for the random-cluster model}

The random-cluster model (also called Fortuin-Kasteleyn percolation) was introduced by Fortuin and Kasteleyn around 1970 (see \cite{For70} and \cite{ForKas72}) as a class of models satisfying specific series and parallel laws. 
It is related to many other models of statistical mechanics, including the Potts model. 
For background on the random-cluster model and the results mentioned below, we direct the reader to the monographs \cite{Gri06} and \cite{Dum13}.

Consider a finite subgraph $G = (V,E)$ of the square lattice.
A percolation configuration $\omega$ is an element of $\{0,1\}^E$. An edge $e$ is said to be {\em open} (in $\omega$) if $\omega(e)=1$, otherwise it is {\em closed}. A configuration  $\omega$ can be seen as a subgraph of $G$ with vertex set $V$ and edge-set $\{e\in E:\omega(e)=1\}$. When speaking of connections in $\omega$, we view $\omega$ as a graph. 
 A {\em cluster} is a connected component of $\omega$ (it may just be an isolated vertex). 
Let $o(\omega)$ and $c(\omega)$ denote the number of open edges and closed edges in $\omega$ respectively. Let $k_0(\omega)$ denote the number of clusters of $\omega$, and $k_1(\omega)$ the number of clusters of $\om$ when all clusters intersecting $\partial V$ are counted as a single one -- as before, $\partial V$ is the set of vertices of $G$ adjacent to a vertex of $\bbZ^2$ not contained in $G$.

For $i\in\{0,1\}$, the random-cluster measure with parameters $p\in[0,1]$, $q>0$ and boundary conditions $i$  is given by
\begin{equation*}
\phi_{G,p,q}^i(\omega)=\frac{p^{o(\omega)}(1-p)^{c(\omega)}q^{k_i(\omega)}}{Z^i(G,p,q)},
\end{equation*}
where $Z^i(G,p,q)$ is a normalizing constant called the partition function. When $i=0$ and $i=1$, we speak of free and wired boundary conditions respectively.

The family of measures $\phi_{G,p,q}^i$ converges weakly as $G$ tends to the whole square lattice. The limiting measures are denoted by $\phi_{\bbZ^2,p,q}^i$ and are called {\em infinite-volume} random-cluster measures with free and wired boundary conditions (for $i$ equal to 0 and 1 respectively). 

For $q\ge1$, the random-cluster model undergoes a phase transition at the critical parameter $p_c=p_c(q)=\sqrt q/(1+\sqrt q)$ (see \cite{BefDum12} or \cite{DumMan14,DCRT16,DumRaoTas17} for alternative proofs), in the following sense:
\begin{itemize}[noitemsep,nolistsep]
\item if $p>p_c(q)$, $\phi_{\bbZ^2,p,q}^0=\phi_{\bbZ^2,p,q}^1$ and the probability of having an infinite cluster in $\omega$ is 1.
\item if $p<p_c(q)$, $\phi_{\bbZ^2,p,q}^0=\phi_{\bbZ^2,p,q}^1$ and the probability of having an infinite cluster in $\omega$ is 0.
\end{itemize}
As before, one may ask whether the phase transition is continuous or not; this  comes down to whether there exists a single critical measure or multiple ones. 
In \cite{DumSidTas16}, it was proved that for $1\le q\le 4$, $\phi_{\bbZ^2,p_c,q}^0=\phi_{\bbZ^2,p_c,q}^1$ and the probability of having an infinite cluster under this measure is 0. 
In this article, we complement this result by proving the following theorem.
Recall that, in this model, $q$ is not necessarily an integer. Also recall that $x_n$ is the site with both coordinates equal to $\lfloor n/2\rfloor$.

\begin{theorem}\label{thm:RCM}
	Consider the random-cluster model on the square lattice with $q>4$. Then
	\begin{enumerate}[nolistsep,itemsep = 2pt]
	\item $\phi_{\bbZ^2,p_c,q}^1\ne\phi_{\bbZ^2,p_c,q}^0$;
	\item $\phi_{\bbZ^2,p_c,q}^1[\text{there exists an infinite cluster}]=1$;
	\item if $\lambda>0$ satisfies $\cosh(\lambda)=\sqrt q/2$, then  
	\begin{equation}\label{eq:aaf}
	\lim_{n\rightarrow\infty} -\tfrac1n\log\phi_{\bbZ^2,p_c,q}^0[0\text{ and }x_n\text{ are in the same cluster}]=
	\lambda + 2 \sum_{k=1 }^\infty \tfrac{(-1)^k}k \tanh(k\lambda).
	\end{equation}
	Furthermore, the quantity on the right-hand side is positive and as $q\searrow 4$,
	\begin{equation}\label{eq:alternative}\lambda + 2 \sum_{k=1 }^\infty\tfrac{(-1)^k}k \tanh(k\lambda)=\sum_{k=0 }^\infty \frac{4}{(2k + 1)\sinh \left(\frac{\pi^2 (2k +1)}{2 \lambda} \right)}\sim 8\exp\left(-\frac{\pi^2}{\sqrt{q-4}}\right).\end{equation}
	\end{enumerate}
\end{theorem}
As in the Potts model, the quantity on the left-hand side of~\eqref{eq:aaf} corresponds to the inverse correlation length in the diagonal direction. Note that it directly implies exponential tails for the radius of the cluster.

The proof of this theorem relies on the connection between the random-cluster model and the six-vertex model defined below. At the level of partition functions, this connection was made explicit by Temperley and Lieb in \cite{TempLieb71}. Here, we will further explore the connection to derive the inverse correlation length; see Section~\ref{sec:RCM} for more details. 

\subsection{Results for the six-vertex model}

The six-vertex model was initially proposed by Pauling in 1931 in order to study the thermodynamic properties of ice.
While we are mainly interested in it for its connection to the previously discussed models, 
the six-vertex model is a major object of study on its own right.
We do not attempt to give an overview of the six-vertex model here; instead, we refer to \cite{Resh10} and Chapter~8 of \cite{Bax89} (and references therein) for a bibliography on the subject 
and to the companion paper \cite{BetheAnsatz1} for details specifically used below. 
  
Fix two even numbers $N$ and $M$, and consider the torus $\mathbb{T}_{N,M}:=\mathbb Z/N\mathbb Z\times\mathbb Z/M\mathbb Z$ as a graph with edge-set denoted $E(\bbT_{N,M})$. 
An {\em arrow configuration} $\vec{\omega}$ is a map attributing to each edge $e=\{x,y\}\in E(\mathbb{T}_{N,M})$ one of the two oriented edges $(x,y)$ and $(y,x)$. We say that an arrow configuration satisfies the {\em ice rule} if each vertex of $\mathbb T_{N,M}$ is incident to two edges pointing towards it (and therefore to two edges pointing outwards from it). The ice rule leaves six possible configurations at each vertex, depicted in Fig.~\ref{fig:the_six_vertices}, whence the name of the model. 
Each arrow configuration $\vec\omega$ receives a weight
\begin{align}\label{eq:w_6V}
	w(\vec{\omega}) := 
	\begin{cases}
	a^{n_1 + n_2} \cdot b^{n_3 + n_4} \cdot c^{n_5 + n_6}&\text{ if }\vec{\omega}\text{ satisfies the ice rule,}
	\\ \qquad\qquad0&\text{ otherwise},
	\end{cases}
\end{align}
where $a,b,c$ are three positive numbers, and $n_i$ denotes the number of vertices with configuration $i\in\{1,\dots,6\}$ in $\vec\omega$. In this article, we will focus on the case $a=b=1$ and $c>2$, and will therefore only consider such weights from now on. 
This choice of parameters is such that the six-vertex model is related to the critical random-cluster model with cluster weight $q>4$ on a tilted square lattice, as explained in Section~\ref{sec:RCM}.

Our choice of parameters corresponds to $\Delta := \frac{a^2+b^2 - c^2}{2ab} < -1$, called the anti-ferroelectric phase. The regime $\Delta \in [-1,1)$, also called disordered, is also of interest and is related to the random-cluster model with $q\le 4$; see \cite{Bax89}.
The regime $\Delta > 1$ (which requires $a \neq b$), called the ferroelectric phase, has also been studied under the name of stochastic six-vertex model and is related to interacting particle systems and random-matrix theory; see the recent paper \cite{BorCorGor16} and references therein.

\begin{figure}[htb]
	\begin{center}
		\includegraphics[width=0.6\textwidth, page=1]{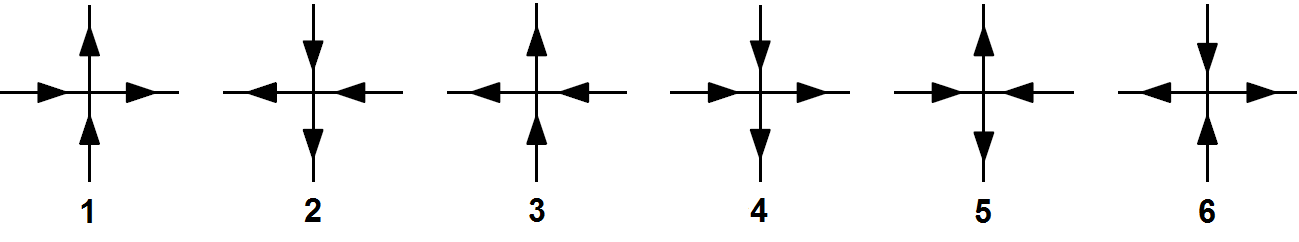}
	\end{center}
	\caption{The $6$ possibilities for vertices in the six-vertex model.
	Each possibility comes with a weight $a$, $b$ or $c$.}
	\label{fig:the_six_vertices}
\end{figure}

In the context of this paper, the utility of the six-vertex model stems from its solvability using the transfer-matrix formalism. More precisely, the partition function of a toroidal six-vertex model may be expressed as the trace of the $M$-th power of a matrix $V$ (depending on $N$) called the {\em transfer matrix}, which we define next.
For more details, see \cite{BetheAnsatz1}.

Set $\vec x=(x_1,\dots,x_n)$ to be a set of ordered integers (called {\em entries}) $1\le x_1<\dots<x_n\le N$ with $0\le n\le N$.
Let $\Omega = \{-1, 1\}^{\otimes N}$ be the $2^N$-dimensional real vector space spanned by the vectors $\Psi_{\vec x} \in \{\pm 1\}^{N}$ given by $\Psi_{\vec x}(i) = 1$ if $i\in\{x_1,\dots,x_n\}$, and $-1$ otherwise. 
The matrix $V$ is defined by the formula
\begin{equation}\label{eq:V}
	V (\Psi_{\vec{x}}, \Psi_{\vec{y}}) = 
	\begin{cases}
		2 & \text{ if }\Psi_{\vec{x}} = \Psi_{\vec{y}},\\ 
		c^{ | \{ i :\, \Psi_{\vec{x}}(i) \neq \Psi_{\vec{y}}(i)\}|} 
		& \text{ if } \Psi_{\vec{x}} \neq \Psi_{\vec{y}} \text{ and } \Psi_{\vec{x}} \text{ and }  \Psi_{\vec{y}} \, \text{ are interlaced},\\ 
		0 & \text{ otherwise},
	\end{cases}
\end{equation}
where $\vec x$ and $\vec y$ are {\em interlaced} if they have the same numbers of entries $n$ and $x_1 \leq y_1 \leq x_2\le \dots \leq x_n \leq y_n$ or $y_1 \leq x_1 \le y_2\leq \dots \leq y_n \leq x_n$.
It is immediate that $V$ is a symmetric matrix; in particular, all its eigenvalues are real. Furthermore, it is
made up of diagonal-blocks $V^{[n]}$ corresponding to its action on the vector spaces
$$\Omega_n:={\rm Vect}(\Psi_{\vec x}:\vec x\text{ has $n$ entries}) \qquad 0\leq n\leq N.$$ 
As discussed in \cite{BetheAnsatz1}, each block $V^{[n]}$ satisfies the assumption of the Perron-Frobenius theorem%
\footnote{More precisely, the entries of $V^{[n]}$ are non-negative and there exists an integer $k$ such that $\big(V^{[n]}\big)^k$ has only positive entries. We will henceforth call a matrix with these properties a Perron-Frobenius matrix.}, 
and thus has one dominant, positive, simple eigenvalue. 
For an integer $0 \leq r \leq N/2$, let $\Lambda_r(N)$ be the Perron-Frobenius eigenvalue of the block $V^{[N/2-r]}$, where we emphasize the dependence of $\Lambda_r$ on $N$ (recall that $N$ is even). The main result dealing with the six-vertex model is the following asymptotic for the aforementioned eigenvalues. 

\begin{theorem}\label{thm:6V}
	For $c>2$ and $r > 0$ integer, fix $\lambda>0$ to satisfy $\cosh(\lambda)=\frac{c^2-2}2$. Then,
	\begin{align}
	\lim_{\substack{N\rightarrow\infty\\ N \in 4\bbN}}\frac1N\log \Lambda_0(N)&=\frac{\lambda}{2} + \sum_{k=1}^\infty \frac{ e^{- k \lambda} \tanh(k\lambda)}{k}\label{eq:aggf} \\
	\lim_{\substack{N\rightarrow\infty\\ N \in 4\bbN}}\frac{\Lambda_r(N)}{\Lambda_0(N)}&=\exp\Big[-r\,\Big(\lambda + 2 \sum_{k=1 }^\infty \tfrac{(-1)^k}k \tanh (k\lambda)\Big)\Big]\label{eq:aggg}.
	\end{align}
\end{theorem}

The limit of $\frac{\Lambda_1(N)}{\Lambda_0(N)}$ is sometimes interpreted as twice the {\em surface tension} of the six-vertex model, and the second equation is effectively a computation of this quantity. The limit of $\frac{\Lambda_r(N)}{\Lambda_0(N)}$ does not have an immediate interpretation but will come in useful when transferring the result to the random-cluster model (see Remark~\ref{rmk:rr}). 
The first identity may be reformulated in terms of the free energy, 
which defines the asymptotic behaviour of the partition function, as described below. 

\begin{corollary}\label{cor:6V}
Fix $c>2$ and  $\lambda>0$ such that $\cosh(\lambda)=\frac{c^2-2}2$. Then the free-energy $f(1,1,c)$ of the six-vertex model satisfies
$$
f(1,1,c):=\lim_{\substack{N,M\rightarrow\infty}} \frac{1}{NM}\log\Big( \sum_{\vec\omega\, \, \textrm{on}\, \, \bbT_{N,M}}w(\vec\omega)\Big)=\frac{\lambda}{2} + \sum_{k=1}^\infty \frac{ e^{- k \lambda} \tanh(k\lambda)}{k} .
$$
\end{corollary}

The previous corollary follows trivially from Theorem~\ref{thm:6V} once observed that the free energy does exist, and that the leading eigenvalue of $V$ is the Perron-Frobenius eigenvalue of $V^{[N/2]}$ (see Section~\ref{sec:6V} for details).

Theorem~\ref{thm:6V} above will be obtained by applying the coordinate Bethe Ansatz to the blocks $V^{[n]}$ of the transfer matrix. 
This ansatz, aimed at finding eigenvalues of certain types of matrices, was introduced by Bethe \cite{Bethe31} in 1931 for the Hamiltonian of the XXZ model. It has since been widely studied and developed, with applications in various circumstances, such as the one at hand. Its formulation for the six-vertex model is described in detail in \cite{BetheAnsatz1}. For completeness, let us briefly discuss this technique again. 

The idea is to try to express the eigenvalues of $V_N^{[n]}$ as explicit functions (see Theorem~\ref{thm:BA} below) of an $n$-uplet $\bp=(p_1,\dots,p_n)\in(-\pi,\pi)^n$ satisfying the $n$ equations
\begin{align}\label{eq:BE}\tag{BE$_\De$}
	N p_j = 2\pi I_j - \sum_{k=1}^n	\Theta (p_j,p_k) \qquad \forall j\in \{1,\dots,n\},
\end{align}
 where the $I_j$ are integers or half-integers (depending on whether $n$ is odd or even) between $-N/2$ and $N/2$, and  $\Theta : \mathbb{R}^2 \rightarrow \mathbb{R}$ is the unique continuous function\footnote{The fact that $\Theta $ is well-defined, real-valued and analytic can be checked easily.} satisfying $\Theta (0,0) = 0$ and
\begin{equation}\label{eq:Theta}
\exp(-i \Theta (x,y)) = e^{i(x-y)} \cdot \frac{ e^{-ix} + e^{iy} - 2 \Delta}{ e^{ix} + e^{-iy} - 2 \Delta} \, ,
\end{equation}
where recall that $\De=(2-c^2)/2$. This parameterization of the six-vertex model will be used throughout the paper. We refer to~\ref{eq:BE} as the Bethe equations.
Depending on the choice of the $I_j$, the eigenvalue obtained may be different. 
It is also {\em a priori} unclear whether all eigenvalues of $V_N^{[n]}$ can be obtained via this procedure. 

The asymptotic behaviour of $\Lambda_0(N)$ was computed in \cite{YangYang66} using the coordinate Bethe Ansatz. The argument of \cite{YangYang66} assumed that $\Lambda_0(N)$ is produced by a solution $\bp(N)=(p_1,\dots, p_{N/2})$ to~\eqref{eq:BE} with $n=N/2$ and the special choice $I_j=j-(n+1)/2$. An asymptotic analysis of the distribution of $p_1,\dots, p_{N/2}$ on $[-\pi,\pi]$ was then used to derive the asymptotic behaviour of $\Lambda_0(N)$. To our best understanding, certain gaps prevent this derivation from being completely justified in this first paper. Among them are the existence of solutions to~\eqref{eq:BE}, the fact that the associated eigenvector constructed by the Bethe Ansatz is non-zero, and the justification of the weak convergence of the point measure of $\bp$ to an explicit continuous distribution.

The more refined asymptotic~\eqref{eq:aggg} requires further justification. For $r=1$ (or equivalently~$-1$), the limit was derived in \cite{Bax89} and \cite{BufWal93}. Baxter's result \cite{Bax89} is based on computations involving a more sophisticated version of the Bethe Ansatz and the eight-vertex model, which generalizes the six-vertex model. The paper \cite{BufWal93} relies on completeness of the six-vertex and Potts representations of the Bethe Ansatz. To our best understanding, both computations require assumptions which are difficult to rigorously justify. We are not aware of any computation of~\eqref{eq:aggg} for $|r|\ge2$. Similar results were obtained rigorously in \cite{Koz15,Gol05} for related models (see the discussion before Theorem~\ref{thm:1} and Remark~\ref{rmk:gold} for more details).

In light of this, we chose to write a fully rigorous, self-contained derivation of both~\eqref{eq:aggf} (which matches Baxter's computation) and~\eqref{eq:aggg}. Moreover, we only use elementary tools, so as to render it accessible to a more diverse audience, less accustomed to the mathematical physics literature. The computations of the two limits in Theorem~\ref{thm:6V} will be used in a crucial way in the proof of Theorem~\ref{thm:RCM}.

\subsection{Organization of the paper}

\paragraph{Section 2: Study of the Bethe equations.}  
This step consists in the study of~\eqref{eq:BE} with the choice 
\begin{equation}
\label{eq:choice I}
I_j:=j-\frac{n+1}2 \qquad\text{for }j\in\{1,\dots,n\}.
\end{equation} 
This section does not involve any reference to the Bethe Ansatz or the six-vertex model. It is divided in three steps:
\begin{enumerate}
\item We first study two functional equations that we call the {\em continuous Bethe Equation} and the {\em continuous Offset Equation}, respectively, via Fourier analysis.
\item We then construct solutions to~\eqref{eq:BE} with prescribed properties. This approach proves the existence of solutions to the Bethe equations and, more importantly, provides good control of the increments $p_{j+1}-p_j$ of the solution. This will be crucial when analysing the asymptotic of $\Lambda_r(N)/\Lambda_0(N)$. It also provides tools for proving that the eigenvectors built via the Bethe Ansatz are non-zero and correspond to Perron-Frobenius eigenvalues.
\item Finally, we study the asymptotic behaviour of the solutions of the discrete Bethe equations using the continuous Bethe Equation. Furthermore, we compare solutions with different values of $n$ using the continuous Offset Equation.
\end{enumerate}

\paragraph{Section 3: From the Bethe equations to the different models.}  
This part contains the proofs of the main theorems. It is divided in two steps.
\begin{enumerate}
\item We use the Bethe Ansatz to relate the Bethe equations to the eigenvalues of the transfer matrix of the six-vertex model. We then study the asymptotic behaviour of the Perron-Frobenius eigenvalues of the different blocks of the transfer matrix using the asymptotic behaviour of the solutions to the continuous Bethe Equation derived in the previous section (see the proof of Theorem~\ref{thm:6V}). 
\item We relate the six-vertex model to the random-cluster and Potts models via classical couplings. These relations, together with new results on the random-cluster model, enable us to prove Theorems~\ref{thm:RCM} and~\ref{thm:Potts}.
\end{enumerate}

\paragraph{Section 4: Fourier computations.} The study will require certain computations using Fourier decompositions. While these computations are elementary, they may be lengthy, and would break the pace of the proofs. We therefore defer all of them to Section 4.

\paragraph{Notation.} Most functions hereafter depend on the parameter $\De =\frac{2-c^2}2 < -1$. For ease of notation, we will generally drop the dependency in $\De$, and recall it only when it is relevant. We write $\pd_i$ for the partial derivative in the $i^{th}$ coordinate.

\paragraph{Acknowledgements} The authors thank Vincent Beffara for the simulations, and Alexei Borodin for useful discussions. The first and the third authors were funded by the IDEX grant of Paris-Saclay. The fifth author was funded by a grant from the Swiss NSF. All the authors are partially funded by the NCCR SwissMap.


\section{Study of the Bethe Equation} 

\subsection{The continuous Bethe and Offset Equations}  

This section studies the following continuous functional equations for $\De<-1$:
\begin{align}\tag{cBE$_\De$}\label{eq:RhoDef}
2 \pi \rho(x) &= 1 + \displaystyle \int_{-\pi}^{\pi} \pd_1 \Theta (x,y) \rho(y) dy		&\qquad \forall x\in[-\pi,\pi],\\
\tag{cOE$_\De$}\label{eq:VarsigmaDef}
2 \pi \tau(x) &= \frac{\Theta(x,-\pi) + \Theta(x,\pi)}{2} - \int_{-\pi}^\pi \partial_2 \Theta(x,y) \tau(y) dy&\qquad \forall x\in[-\pi,\pi].
\end{align} 
The first equation naturally arises as a continuous version of the Bethe equations \eqref{eq:BE}, while the second one will be useful when studying the displacement between solutions of the Bethe equations for different values of $n$.

The main object of the section is the following proposition. For $\De<-1$, let 
$k$ be the unique continuous function\footnote{The existence of $k$ follows by taking the complex logarithm and fixing $k(\pm\pi)=\pm\pi$. Furthermore, $k$ is invertible. Also notice that $z\mapsto (e^\lambda-z)/(e^\lambda z-1)$ is a Mobi\"us transformation mapping the unit circle to itself and $-1$ to $-1$.}  
from  $[-\pi,\pi]$ to itself satisfying 
$$ e^{\icomp k(\alpha)} = \frac{{e^\lambda}-e^{-\icomp \alpha}}{e^{\lambda- \icomp \alpha}-1}, $$ 
where $\lambda>0$ is such that $\cosh(\lambda)=-\De$.

\begin{proposition}\label{prop:rhoProp}
	For $\De<-1$, let $x = k(\alpha)$. The functions $\rho $ and $\tau $ defined by\footnote{In the definition of $\tau$, the series is not absolutely summable; it stands for the limit of the partial sums.}
	\begin{align}\label{eq:rho_explicit}
	\rho (x)&:=\frac1{4 \la k'(\alpha)}\sum_{j \in \mathbb{Z}} \frac{1}{\cosh[\pi (2 \pi j +\alpha)/(2 \lambda)]},\\
	\tau (x) &:= \sum_{m>0}\frac{(-1)^{m}}{\pi m}\tanh(\lambda m)\sin(m\alpha),
	\nonumber
	\end{align}
	are the only solutions in $L^2([-\pi,\pi])$ of~\eqref{eq:RhoDef} and~\eqref{eq:VarsigmaDef} respectively. 
	In particular, the function $\rho:(\De,x)\mapsto \rho (x)$ is strictly positive and analytic in $\De<-1$ and $x\in[-\pi,\pi]$.
\end{proposition}

We prove this result by making a change of variables $x=k(\alpha)$ to obtain equations involving a convolution operator, and then using Fourier analysis to compute $\rho$ and $\tau$ (and therefore deduce their uniqueness). 

\begin{proof} In this proof, we fix $\De<-1$ and drop it from the notation. 
Set $R(\alpha)=2\pi\rho(k(\alpha))k'(\alpha)$ and $T(\alpha)=2\pi\tau(k(\alpha))$.
The change of variables $x=k(\alpha)$ transforms~\eqref{eq:RhoDef} and~\eqref{eq:VarsigmaDef} into\footnote{A series of algebraic manipulations is necessary for this step. The key is to observe that 
	$k'(\alpha) = \Xi_\la(\alpha)$ and 
	$\Xi_{2 \lambda}(\alpha - \beta) = - \frac{d}{d \alpha} \Theta(k(\alpha),k(\beta)) =  \frac{d}{d \beta}\Theta(k(\alpha),k(\beta))$.
}
\begin{align}
	\tag{cBE$'_\De$}\label{eq:RhoDef'}
	R(\alpha) &=  \Xi_{\lambda}(\alpha)  - \frac{1}{2\pi} \displaystyle \int_{-\pi}^\pi\Xi_{2 \lambda}(\alpha - \beta) R(\beta) d \beta
	\qquad &\forall \alpha \in [-\pi,\pi],
	\\ 
	\tag{cOE$'_\De$}
	T(\alpha) &=  \Psi(\alpha)  - \frac{1}{2\pi} \displaystyle \int_{-\pi}^\pi\Xi_{2 \lambda}(\alpha - \beta) T(\beta) d \beta
	\qquad &\forall \alpha \in [-\pi,\pi],
	\label{eq:VarsigmaDef'}
\end{align}
where, for $\mu \in \bbR$ and $\alpha \in [-\pi,\pi]$, 
$$
	\Xi_{\mu}(\alpha):=\frac{\sinh (\mu)}{\cosh(\mu)- \cos(\alpha)}\qquad\text{ and }\qquad 
	\Psi(\alpha):=\frac{\Theta(k(\alpha),-\pi)+\Theta(k(\alpha),\pi)}2.
$$
For any function $f \in L^2([-\pi, \pi])$, denote by $(\hat{f}(m))_{m \in \bbZ}$ its Fourier coefficients 
defined as $\hat f(m) := \frac1{2\pi}\int_{-\pi}^\pi e^{-i m\alpha}f(\alpha)d\alpha$. 
Then, \eqref{eq:RhoDef'} and \eqref{eq:VarsigmaDef'} may be rewritten as 
\begin{align}\label{eq:rfourier}
	\hat{R}(m) &= \hat{\Xi}_\lambda(m) - \hat{\Xi}_{2\lambda}(m) \hat{R}(m)\qquad\text{ and }\qquad
	\hat{T}(m) = \hat{\Psi}(m) - \hat{\Xi}_{2\lambda}(m) \hat{T}(m)
	\qquad \forall m\in\bbZ.
\end{align}

The end of the proof is a simple computation which we resume next; details are given in Section~\ref{sec:Fourier}. 
The residue theorem shows that $ \hat{\Xi}_{\mu}(m) = \exp(- \mu |m|)$. 
In addition, a simple computation implies that $\hat{\Psi}(m)=\frac{(-1)^{m}}{\icomp m} \left(1 - \hat \Xi_{2\lambda}(m)\right)$ for $m \neq0$ and $\hat{\Psi}(0)=0$.
Substituting these in~\eqref{eq:rfourier}, we deduce that, for all $m \in \bbZ$ ($m\neq 0$ for the second equality), 
\begin{equation*}
	\hat{R}(m) = \frac{\hat{\Xi}_{\lambda}(m)}{1 + \hat{\Xi}_{2\lambda}(m)} = \frac{1}{2 \cosh(\lambda m)} 
	\qquad\text{and}\qquad
	\hat T (m) =  \frac{\hat{\Psi}(m)}{1 + \hat{\Xi}_{2\lambda}(m)} = \frac{(-1)^m}{\icomp m} \tanh (\lambda |m|). 
\end{equation*}
The conclusion follows by checking that functions given in~\eqref{eq:rho_explicit} have the Fourier coefficients above; details are given in Section~\ref{sec:Fourier}. 
The properties of positivity and analyticity of $\rho$ follow directly from its explicit expression (observe that the terms of the sum in \eqref{eq:rho_explicit} are positive and converge exponentially fast to $0$).
\end{proof}

Before turning to the discrete equations, let us provide an alternative proof of the uniqueness of the solution to~\eqref{eq:RhoDef} based on a fixed-point theorem. While this proof does not give an explicit formula for $\rho$ (a formula which will be useful later on), it highlights the importance of a particular norm which will play a central role in the next section. The goal is to prove that the map $\mathsf T_c$ defined below is contractive, a fact which immediately implies that \eqref{eq:RhoDef} has a unique solution.

Fix $\De<-1$. Consider the map $\mathsf T_c$ from the set $\calH$ of bounded functions $f:[-\pi,\pi]\longrightarrow\bbR$ with $\int_{-\pi}^\pi f(x)dx=1/2$ to itself\footnote{That $\mathsf T_c(\calH) \subset \calH$ follows from Fubini's theorem and the fact that $\Theta (\pi,y) - \Theta (-\pi,y) = -2\pi$ for all $y \in [-\pi,\pi]$.} defined by 
\begin{equation*}
	2 \pi {\mathsf T}_c(f)(x) = 1 + \displaystyle \int_{-\pi}^{\pi} \pd_1 \Theta (x,y) f(y) dy\qquad \forall x\in[-\pi,\pi].
\end{equation*}
We claim that this map is contractive for the norm defined by 
\begin{align}\label{eq:defnorm}
	\|f\|~:=~\sup\big\{\,\big|k'(k^{-1}(x))f(x)\big|\,,\,x\in[-\pi,\pi]\big\}
	=~\sup\big\{\,\big|k'(\alpha) f(k(\alpha))\big|\,,\,\alpha\in[-\pi,\pi]\big\}.
\end{align}
Note that $k'$ is bounded away from 0 and infinity, so that the norm above is equivalent to the supremum norm $\|\cdot\|_\infty$ 
(with constants depending on $\Delta  <-1$). 

Indeed, let $f$ and $g$ be two functions in $\calH$. Set
$F=k'\cdot(f\circ k)$, $G=k'\cdot(g\circ k)$, 
$\tilde F=k'\cdot(\mathsf T_c(f)\circ k)$ and $\tilde G=k'\cdot(\mathsf T_c(g)\circ k)$, 
and notice that all these functions integrate to $1/2$ on $[-\pi,\pi]$.
Letting $m_{\Xi} = \min\{\Xi_{2\lambda}(x):x\in[-\pi,\pi]\} > 0$, we find that, for any $\alpha \in [-\pi,\pi]$,
\begin{align}
	|\tilde F(\alpha)-\tilde G(\alpha)|
	&=\tfrac1{2\pi}\Big|\int_{-\pi}^\pi \Xi_{2\lambda}(\alpha-\beta)(F(\beta)-G(\beta))d\beta\Big|\nonumber\\
	&=\tfrac1{2\pi}\Big|\int_{-\pi}^\pi \big(\Xi_{2\lambda}(\alpha-\beta)-m_{\Xi}\big)(F(\beta)-G(\beta))d\beta\Big|\nonumber\\
	&\le \tfrac1{2\pi}\|F-G\|_\infty\int_{-\pi}^\pi \big(\Xi_{2\lambda}(\beta)-m_{\Xi}\big)d\beta \nonumber\\
	&\le (1-m_{\Xi}) \|F-G\|_\infty,\label{eq:art}
\end{align}
where we used the fact that $\Xi_{2\lambda}(\beta)$ integrates to $2\pi$ (since $\hat\Xi_{2\lambda}(0)=1$) in the final line. Observing that $\|F-G\|_\infty = \|f - g\|$ and $\|\tilde F- \tilde G\|_\infty = \|\mathsf T_c(f) - \mathsf T_c(g)\|$, we conclude that $\mathsf T_c$ is contracting.


\subsection{The discrete Bethe equations}
The main object of this section is to prove the existence and regularity of solutions to the Bethe equations recalled below:
\begin{align*}\tag{BE$_\De$}
	N p_j = 2\pi I_j - \sum_{k=1}^n	\Theta (p_j,p_k), \qquad \forall j\in \{1,\dots,n\},
\end{align*}
with the choice~\eqref{eq:choice I} for the $I_j$, namely $I_j=j-\frac{n+1}2$ for $1\le j\le n$. We will be looking for solutions $\bp$ with additional symmetry (which takes into account the symmetry of the $I_j$). More precisely, we will be looking for solutions in
$$\calS_n:=\Big\{\bp = (p_1,\dots,p_n)\,:-\pi<p_1<p_2<\dots<p_n<\pi\text{ and } p_{n+1-j}=-p_j,\forall j\Big\}.$$
For any vector $\bp =(p_1,\dots, p_n)$, we set $p_0=p_n-2\pi$ and $p_{n+1}=p_1+2\pi$.
Hereafter, $N$ will always denote an even integer. 
\bigbreak

Maybe the most natural approach to proving the existence of solutions to~\eqref{eq:BE} (for fixed $\De$, $N$ and $n$) is to apply the Brouwer Fixed-Point Theorem\footnote{For sufficiently small values of $\De$, the Brouwer Fixed-Point Theorem is not even necessary, since one may show that $\mathsf T$ is contractive for the $\ell^1$ norm; this is not true for $\Delta$ close to $-1$.} to the map $\mathsf T : \calS_n \to \calS_n$ defined by 
\begin{align*}
 	\mathsf T(p_1,\dots,p_n) = \bigg(\frac{2\pi I_j}{N}-\frac{1}{N}\sum_{k=1}^n\Theta (p_j,p_k)\bigg)_{1\leq j\leq n}.
\end{align*}
Indeed, $\bp$ being a fixed point of $\mathsf T$ is equivalent to it satisfying~\eqref{eq:BE}.
The fact that $\mathsf T$ maps $\calS_n$ to itself follows directly from the monotonicity and anti-symmetry of $\Theta$, and from the fact that 
$$ -2\pi \leq \Theta(x,y) + \Theta(x,-y) \leq 2\pi  \qquad \forall x,y \in [-\pi,\pi].$$ 

The Brouwer Fixed-Point Theorem indeed applies to $\mathsf T$, and solutions to~\eqref{eq:BE} may thus be shown to exist for any $\De < -1$. 
Having said that, it will be important that the solutions vary continuously as functions of $\De$, which does not follow from such arguments. 
Such a continuity statement was proved by Karol Kozlowksi in \cite{Koz15} and Pedro Goldbaum in \cite{Gol05} for the 1D Hubbard model. The argument used in the latter paper generalizes the earlier work of Yang and Yang \cite{YangYang66}, using an Index theorem on a well-chosen field, and thus deducing that the solutions form families of continuous curves, proving that there exists a continuous curve of solutions to~\eqref{eq:BE} in the set $[-\infty,-1)\times [-\pi,\pi]^n$, extending over the whole range of $\De$. 

However, we wish to prove a  stronger statement: we would like the solutions to have some regularity, in that they should be close to $\rho$, the solution we explicitly computed in~\eqref{eq:rho_explicit}, in some appropriately-chosen sense. This will be important when comparing solutions for different values $n$ to compute the limit of $\Lambda_r(N)/\Lambda_0(N)$. 

We therefore choose another path to prove the existence of solutions, based on the Implicit Function Theorem. 
Our approach has the further advantage of being fairly short and elementary, and of proving that the obtained solution is close to the continuous one (which renders the asymptotic analysis of $\La_0(N)$ essentially trivial). Furthermore, we will also prove that the map $\De \mapsto \bp_\De$ is not only continuous but analytic, a fact which will be useful in proving that the eigenvalue associated with $\bp_\De$ is the Perron-Frobenius one (see Section~\ref{sec:6V_PF}). The downside is that it only yields a solution on an interval $[-\infty,\De_N]$ with $\De_N<-1$, tending to $-1$ as $N$ tends to infinity (which will be sufficient for the application we have in mind).

Before stating the theorem, let us explain how we will compare a solution $\bp$ of~\eqref{eq:BE} to the continuous solution $\rho$ of~\eqref{eq:RhoDef}. For $\bp\in\calS_n$, introduce the step function $\rho_\bp:[-\pi,\pi] \to \bbR$ defined by 
\begin{equation}\label{eq:add}
	\rho_{\bp}(t)=\frac{I_{j+1}-I_j}{N(p_{j+1}-p_j)}
	\qquad \text{if $t\in[p_j,p_{j+1})$,}
\end{equation}
where $I_{n+1}$ and $I_0$ are defined by $I_{n+1}-I_1=I_n-I_0=N - n$. 
We measure the distance from $\bp$ to the continuous solution using $\|\rho_\bp-\rho\|$, where $\|\cdot\|$ is the norm introduced in~\eqref{eq:defnorm}. This norm appears naturally in this context since the map $\mathsf T_c$ -- which may be viewed as a continuous version of $\mathsf T$ -- is contractive for~$\|\cdot\|$.

\begin{remark}
We chose to write $I_{j+1}-I_j$ in the numerators, since this would be the natural quantity would the $I_j$ take arbitrary values. In our case, $I_{j+1}-I_j$ is equal to $1$ for any $1\le j<n$, and to $2r+1$ for $j=0$ and $n$ (recall that $n=N/2-r$).
\end{remark}

We are now in a position to state the main theorem of this section.

\begin{theorem}\label{thm:1}
Fix $r\ge0$ and $\De_0<-1$. There exist $K>0$ and $N_0$ such that, for any $N \geq N_0$, there exists a family of solutions $(\bp_\De)_{\De\le \De_0}$ to the Bethe equations~\eqref{eq:BE} with $n=N/2-r$ satisfying 
\begin{itemize}[nolistsep, itemsep = 2pt]
\item[(i)] $\De\mapsto \bp_{\De}$ is analytic on $[-\infty,\De_0)$ \footnote{Here, $f$ is analytic at $-\infty$ if there exists $(a_n)$ such that $f(x)=\sum_{n=0}^\infty a_n x^{-n}$ for all $x$ sufficiently small.},
\item[(ii)] $\|\rho_{\bp_\De}-\rho\|\le \frac{K}{N}$ for all $\De\le \De_0$.
\end{itemize}
\end{theorem}

Property (ii) should be understood as a regularity statement. It implies in particular that, for all $0\le j\le n$,
\begin{align}\label{eq:regularity}
p_{j+1}-p_j-\frac{I_{j+1}-I_j}{\rho (p_j)N}&~=~ O\Big(\frac{1}{N^2}\Big),
\end{align}
where $O(\cdot)$ depends on $\De_0$ only%
\footnote{The notation $O(N^\alpha)$ is used to indicate a quantity that is bounded by $CN^{\alpha}$ for all $N$, where $C$ is a constant. 
We say $O(\cdot)$ is uniform in certain parameters, to mean that $C$ may be chosen independently of those parameters.}.
As an important consequence for us, the previous expression implies that, for $N > N_0$ large enough and $0\le j\le n$,
\begin{align}\label{eq:cd}
p_{j+1}-p_j~&\le~ \frac {2 (I_{j+1}-I_j)}{m_\rho N},
\end{align}
where $m_\rho > 0$ is the infimum of $\rho$ over $x\in[-\pi,\pi]$ and $\De\le \De_0$.  
It will be crucial to us that the bound~\eqref{eq:cd} above does not depend on the quantity $K$ of Theorem~\ref{thm:1} (even though $N_0$ may depend on $K$).
Also notice that (ii) implicitly shows that $\bp_\De$ is in the interior of $\calS_n$ for all $\De\leq\De_0$, provided $N$ is large enough. 


The rest of this section is dedicated to proving Theorem~\ref{thm:1}.
\bigbreak
As we already mentioned, our strategy is based on the Implicit Function Theorem, which will be applied to $\mathbb I-\mathsf T$ (seen as a function of $\De$ and $\bp$), where $\mathbb I$ denotes the identity function:
$$ \mathbb I (\De,\bp) = \bp,\qquad \forall \bp \in \calS_n \text{ and } \De < -1.$$ 
There is no {\em a priori} reason that allows us to apply the Implicit Function Theorem at any zero of $\mathbb I-\mathsf T$, as the differential is not guaranteed to be invertible. Nonetheless, we will show that we may construct a family of such zeros that remains close to the continuous solution, and that this ensures that the differential of $\mathbb I -\mathsf T$ is invertible. 
The key of this argument is the following stability lemma. 


\begin{lemma}\label{lem:1}
	Fix $r\ge0$ and $\De_0<-1$. 
	Then, for $K > 0$ and $N_0$ large enough, for any $\De\le \De_0$ and $N\ge N_0$, there exists no solution $\bp \in \calS_n$ of~\eqref{eq:BE} with $n= N/2 -r$ and 
	\begin{align*}
		\frac{K}{2N}\le \|\rho_\bp-\rho\|\le \frac{K}N.
	\end{align*}
\end{lemma}

This lemma should not appear as a surprise. Indeed, as mentioned above, $\mathsf T$ is, in some sense, a discrete version of $\mathsf T_c$ (which is contractive and has fixed point $\rho$), at least in a vicinity of~$\rho$, and could therefore be expected to  be contractive for $N$ large enough. We did not manage to prove this fact, but the lemma above is sufficient for our use.

Let us assume this lemma for now. Write $\bbR_\sym^n$ for the $\lfloor n/2\rfloor$-dimensional subspace of $\bbR^n$ of symmetric vectors: 
$$ \bbR_\sym^n = \big\{(q_1,\dots,q_n) \in \bbR^n:\, q_j = -q_{n + 1 -j}, \, \forall j\big\}.$$
The map $\mathbb I - \mathsf T$ leaves this space stable (as can be seen by the symmetry properties of $\Theta$).
Therefore, we may apply the Implicit Function Theorem to $\mathbb I - \mathsf T$ 
as a function from $[-\infty,\De_0]\times \calS_n$ to $\bbR^n_\sym$ (recall that $\calS_n \subset \bbR^n_{\sym}$).
Write $d(\mathbb I - \mathsf T)$ for (the restriction of) the differential of $\mathbb I - \mathsf T$ in $\bp$ 
as an automorphism of $\bbR^n_\sym$. 
To apply the  Implicit Function Theorem at some point $(\De,\bp)$ one needs to ensure that $d(\mathbb I - \mathsf T)(\De,\bp)$ is invertible.
This is done via the lemma below; its proof is deferred to the end of the section.   

\begin{lemma}\label{lem:2}
	Fix $r\ge0$, $\De_0<1$ and $K>0$. Then there exists $N_0$ such that, for any $\De\le\De_0$ and $N\ge N_0$,
	$d(\mathbb I-\mathsf T)(\De,\bp)$ is invertible 
	for any solution $\bp \in \calS_n$ of~\eqref{eq:BE} with $n= N/2 -r$ such that $\|\rho_\bp-\rho\|\le K/N$.
\end{lemma}

\begin{proof}[Theorem~\ref{thm:1}]
Fix $r\ge0$ and $\De_0<-1$; $K \geq 2r$ and $N$ will be assumed large enough for Lemmas~\ref{lem:1} and~\ref{lem:2} to apply,
further conditions on $N$ will appear in the proof.   

For $\De = -\infty$ we have $\Theta(x,y) = y-x$, and the Bethe equations have a unique solution $\bp_{-\infty}$ with $p_j=2\pi I_j/(N-n)$ for $1\le j\le n$. This solution satisfies $\bp \in \calS_n$ and $\|\rho_{\bp}-\rho\|\le K/N$ 
($\rho$ is the constant function $1/(4\pi)$ when $\De = -\infty$ and we assumed $K \geq 2r$).

Due to Lemma~\ref{lem:2}, the Implicit Function Theorem may be repeatedly applied to extend the solution from $\bp_{-\infty}$ to an analytic function $\De \mapsto \bp_{\De}$, as long as $\|\rho_{\bp_\De} - \rho \| \leq K/N$ and $\bp_\De \in \calS_n$.
The latter condition is implied by the former when $N$ is large enough; we may therefore ignore it.
Lemma~\ref{lem:1} shows that $\bp_{\De}$, being continuous in $\De$, may never exit the ball of radius $K/(2N)$ around $\rho$ for the $\|\cdot\|$-norm. 
Thus, the map $\De \mapsto \bp_{\De}$ is defined for all $\De \leq \De_0$, analytic and such that $\|\rho_{\bp_\De} - \rho \| \leq K/N$ for all $\De\le \De_0$.
%
\end{proof}
%

To close the section, we prove Lemmas~\ref{lem:1} and~\ref{lem:2}.
 
\begin{proof}[Lemma~\ref{lem:1}] 
Let $r\ge0$ and $\De_0<-1$; bounds on $K$ and $N_0$ will appear throughout the proof. 
Consider $\De \leq \De_0$, $N\ge N_0$ and $\bp = (p_1,\dots, p_n)$ a solution of~\eqref{eq:BE} 
with $n = N/2 - r$ and $\|\rho_\bp-\rho\|\le K/N$.

In this proof, $O(\cdot)$ is uniform in $K$ and $j = 1,\dots, n$ (but may depend on $r$). 
In particular, by~\eqref{eq:cd}, we may write that $p_{j+1}-p_j=O(1/N)$, provided $N_0$ is large enough. 
For further reference, note that the derivatives of the functions $\rho$, $k$, $\Theta$, etc, are all bounded uniformly in $\De < \De_0$. 

Let  $f_{\bp}:\bbR \to \bbR$ be the smooth function defined by 
\begin{equation}\label{eq:function}
	f_\bp(x):=\frac1{2\pi}\Big(x+\frac1{N}\sum_{k=1}^n\Theta(x,p_k)\Big).
\end{equation}
For any $t\in[p_j,p_{j+1})$, apply the Mean Value Theorem to construct $\xi_j\in(p_j,p_{j+1})$ such that\begin{equation} \label{eq:MVT}
\rho_\bp(t)\stackrel{\eqref{eq:add}}=\frac{I_{j+1}-I_j}{N(p_{j+1} - p_j)} =\frac{ f_\bp(p_{j+1})-f_\bp(p_j)}{p_{j+1} - p_j} = f'_\bp(\xi_j).
\end{equation}
In the second identity, we used that $\bp$ is a fixed point for $\mathsf T$ and therefore satisfies $f_\bp(p_{j}) = I_j/N$ for any $j \in \{1,\dots, n-1\}$. In fact, this relation also holds for $j=0$ and $n$; to see this, we recall that $\Theta (x + 2\pi,y) - \Theta (x,y) = -2\pi$, and therefore
\[
f_{\bp}(x \pm  2\pi) = f_{\bp}(x) \pm \frac{N -n}{N}.
\] 
Thus,
\[
f(p_1) - f(p_0) = \frac{I_1 -(I_n - N + n)}{N} = \frac{I_1 - I_0}{N}. 
\]
The argument is identical for $j = n$. Since $p_{j+1}-p_j=O(1/N)$, for any $t \in [p_j,p_{j+1})$, we may approximate $\rho(t)$ by $\rho(\xi_j)$ and $k'(k^{-1}(t))$ by $k'(k^{-1}(\xi_j))$ to deduce that 
\begin{align*}
	k'(k^{-1}(t))|\rho_\bp(t)-\rho(t)|
	&\le (1+O(\tfrac1N))\,k'(k^{-1}(\xi_j))|f'_\bp(\xi_j)-\rho(\xi_j)|+O(\tfrac1N)\\
	&\le (1+O(\tfrac1N)) \|f'_\bp-\rho\|+O(\tfrac1N).
\end{align*}
Therefore, the lemma follows readily from the following inequality, which we prove below:
\begin{equation}\label{eq:ac}
	\|f'_{\bp}-\rho\|\le (1-m_{\Xi})\|\rho_{\bp}-\rho\|+ O(\tfrac1N),
\end{equation}
where $m_\Xi = \inf\{\Xi_{2\lambda}(x) :\, x \in [-\pi,\pi] \text{ and } \De\le\De_0\} > 0$.
Indeed, assuming~\eqref{eq:ac} holds, the previous computation shows that
$$\|\rho_\bp-\rho\|\le (1-m_\Xi)\|\rho_\bp-\rho\|+O(\tfrac1N),$$
which implies the result for $K$ large enough (recall that the constant in $O(1/N)$ above does not depend on $K$).

Hence, we only need to prove~\eqref{eq:ac} to finish the proof of the lemma.
Set $R_\bp(\alpha):=\rho_{\bp}(k(\alpha))k'(\alpha)$. 
Fix $x=k(\alpha)$. With this definition, the change of variable explained in the previous section implies that 
\begin{align*}
    2\pi f_\bp'(x)
    = 1+\frac1{N}\sum_{k=1}^n\partial_1\Theta(x,p_k)
    &= 1 + \int_{-\pi}^\pi \partial_1\Theta(x,y) \rho_{\bp}(y)dy ~+~O(\tfrac1N)\\
    &= 1 + \frac{1}{k'(\alpha)}\int_{-\pi}^\pi \Xi_{2\lambda}(\alpha-\beta)R_\bp(\beta)d\beta ~+~O(\tfrac1N), 
\end{align*}
where we used again that $\max\{p_{j+1}-p_j\}=O(\tfrac1N)$ and that $\partial_2\partial_1\Theta$ is bounded uniformly to approximate $\partial_1\Theta(x,p_k)$ by $\partial_1\Theta(x,k(\beta))$. Thus,
\begin{align*}
k'(k^{-1}(x))|f_\bp'(x)-\rho(x)|~&\stackrel{\phantom{\eqref{eq:art}}}=~\Big|\tfrac1{2\pi}\int_{-\pi}^\pi \Xi_{2\lambda}(\alpha-\beta)(R_\bp(\beta)-R(\beta))d\beta\Big|~+~O(\tfrac1N)\\
&\stackrel{\eqref{eq:art}}{\le}~(1-m_\Xi)\|\rho_\bp-\rho\|~+~O(\tfrac1N),
\end{align*}
where in the last inequality, we can apply~\eqref{eq:art} since 
$	\int_{-\pi}^\pi R_\bp(\alpha)d\alpha=\int_{-\pi}^\pi\rho_\bp(x)dx=\frac12.$
\end{proof}

\begin{proof}[Lemma~\ref{lem:2}]
Let $r, \De_0$ and $K$ be as in the statement of the lemma; $N_0$ will be chosen later in the proof. 
Fix $\De\le\De_0$, $N\ge N_0$ and $\bp \in \calS_n$ satisfying~\eqref{eq:BE} with $n= N/2 -r$ and such that $\|\rho_\bp-\rho\|\le K/N$.

Note that for $\De=-\infty$, $\mathsf T$ is equal to $\mathbb I/2$, and the result is trivial.
We may therefore assume $\De\in(-\infty,\De_0]$. 

Write $A$ for $d(\mathbb I-\mathsf T)(\De,\bp)$, the differential of $\mathbb I-\mathsf T$ in $\bp$ at the point $(\De,\bp)$ fixed above.
Recall that we see $A$ as an automorphism of $\bbR^n_\sym$.
We will regard it as a square matrix of size $\lfloor n/2\rfloor$, 
when written in the basis $(e_j - e_{n+1-j})_{1\leq j\leq \lfloor n/2 \rfloor}$ of $\bbR^n_\sym$, 
where $(e_j)_{1\leq j \leq n}$ is the canonical basis of $\bbR^n$. 
We may write $A$ explicitly: 
\begin{align*}
	A_{jk} 
	= \frac{\pd\big[(\mathbb I-\mathsf T)(\De,\bp)\big]_j}{\pd p_k} - \frac{\pd\big[(\mathbb I-\mathsf T)(\De,\bp)\big]_j}{\pd p_{n+1-k}}
	=
	\begin{cases}
		\displaystyle1 + \tfrac{1}{N}\sum_{\ell \neq j} \partial_1\Theta (p_j,p_\ell) - \tfrac1N\partial_2\Theta (p_j,-p_j) & \text{ if $j = k$},\vspace{3pt}\\
		\frac1N\big[\partial_2\Theta (p_j,p_k) - \partial_2\Theta (p_j, - p_k)] & \text{ if $j \neq k$},
	\end{cases}
\end{align*}
for $1 \leq j,k\leq n/2$. For the second equality, we have used $p_{n+1-k} = -p_{k}$.

Also, write $B$ for the diagonal matrix of size $\lfloor n/2 \rfloor$, with entries $N(p_{j+1} - p_j) =\rho_\bp(p_j)^{-1}$ on the diagonal. 
Rather than proving that $A$ is invertible, we will prove that $\tilde A = AB$ is invertible, by showing that it is diagonally dominated -- i.e. $\tilde{A}_{ii} > \sum_{j \neq i} \tilde{A}_{ij}$ for every $i$. 

Below, the notation $O(\cdot)$ is considered uniform in $\De < \De_0$ and $j$, but may depend on the fixed constants $K$ and $r$. 
Due to the condition  $\|\rho_\bp-\rho\|\le K/N$, we may write $p_{j+1} - p_j = O(1/N)$. 
Finally, we will use that the functions $\rho$ and $\Theta$ and their derivatives are uniformly bounded for $\De \leq \De_0$
(provided that $N$ is large enough).

The diagonal terms of $\tilde A$ are 
\begin{align}\label{eq:rrc}
	\tilde  A_{jj} 
	&=\frac{1}{\rho_\bp(p_j)} \Big( 1 + \frac{1}{N}\sum_{k\neq j} \partial_1\Theta (p_j,p_k)\Big)+ O\big(\tfrac{1}{N}\big)\\
	&= \frac{1}{\rho_\bp(p_j)}  \Big( 1 + \int_{-\pi}^{\pi} \pd_1 \Theta (x,y) \rho(y)\Big) + O\big(\tfrac{1}{N}\big)
	\stackrel{\eqref{eq:RhoDef}}= \frac{2 \pi \rho(p_j)}{\rho_\bp(p_j)} + O\big(\tfrac{1}{N}\big) \nonumber
	= 2\pi + O\big(\tfrac{1}{N}\big).
\end{align}
For the second equality\footnote{The equality is obtained by a simple computation similar to that of the proof of Theorem~\ref{thm:2} below. We omit the details here.}, 
we used $\|\rho_\bp-\rho\|\le K/N$.  We further note that the final equality follows thanks to the fact that $m_\rho > 0$.

We now compute the off-diagonal terms of $\tilde A$. 
For $x,y \in [-\pi,\pi]$, write $G(x,y) := \Theta(x,y) - \Theta(-x,y)$.
A direct computation shows that $G(x,y)$ is increasing in $y$ when both $x$ and $y$ are in $[-\pi,0]$.
For $1 \leq j \neq k \leq n/2$, since $\Theta(x,-y) = -\Theta(-x,y)$, we have
\begin{align*}
	\tilde A_{jk} 
	=(p_{k+1} - p_k)\big[\partial_2\Theta (p_j,p_k) - \partial_2\Theta (-p_j, p_k)\big] 
	=(p_{k+1} - p_k)\partial_2G (p_j,p_k) 	
	\geq 0.
\end{align*}
Therefore, for any fixed $1 \leq j \leq n/2$,    
\begin{align}
	 \sum_{k \neq j} |\tilde A_{jk} |
	 = \sum_{k \neq j} \tilde A_{jk} 
	 & = \sum_{k =1}^{\lfloor n/2 \rfloor }(p_{k+1} - p_k) \pd_2G (p_j,p_k)  +O\big(\tfrac1N\big) \nonumber\\
	 & = G (p_j,0) - G(p_j,-\pi)+ O\big(\tfrac1N\big).
	 \label{eq:afd}
\end{align}
A straightforward calculus exercise can show that, for any $\Delta < \Delta_0$, the function $G(x,0)-G(x,-\pi)$ satisfies 
\begin{align*}
	G(x,0)-G(x,-\pi)\le 4 \arctan\Big(\frac{1}{2 |\Delta_0| \sqrt{\Delta_0^2 -1}}\Big)<2\pi, \qquad \forall x \in [-\pi,0].
\end{align*}
In conclusion,~\eqref{eq:rrc} and~\eqref{eq:afd} show that for $N$ large enough (depending on $\De_0, r$ and $K$ only), $\tilde A$ is diagonal dominant and therefore invertible. 
\end{proof}

\subsection{The asymptotic behaviour of the solutions to the Bethe equations}

This section is devoted to two results that control the asymptotic behaviour of solutions to the Bethe equations when $\rho_\bp$ is close to $\rho$. 
The first deals with the ``first order'' asymptotics of solutions to~\eqref{eq:BE} with $n = N/2 - r$, for fixed $r$.

\begin{theorem}\label{thm:2}
	Fix $\De<-1$ and $r \geq 0$. 
	Consider a family of $\bp(N)\in \calS_{N/2-r}$ for $N$ even large enough 
	satisfying $\|\rho_{\bp(N)}-\rho\|\longrightarrow0$. 
	Then, $\mu_N:=\frac1{N}\sum_{i=1}^n\delta_{p_i(N)}$
	converges weakly to $\rho(x)dx$, where $dx$ is Lebesgue's measure on $[-\pi,\pi]$. 
\end{theorem}

\begin{proof}
Fix $\De<-1$, $r \geq 0$ and set $n=N/2-r$. For any continuous function $g$ on $[-\pi,\pi]$ and $N \geq2r$, define $g_{\bp(N)}:[-\pi,\pi] \to \bbR$ by 
$g_{\bp(N)}(t):=g(p_j)$ if $t\in[p_j,p_{j+1})$ for some $0\le j\le n$ (where we extend $g$ periodically whenever needed).
Then 
$$\int_{-\pi}^\pi g(x)d\mu_N(x) = \frac1N\sum_{j=1}^n g(p_j) = \int_{-\pi}^\pi g_{\bp(N)}(x)\rho_{\bp(N)}(x)dx  + \frac{g(p_n) -g(p_1)}{2N} ,$$ 
and we find
\begin{align*}
& \int_{-\pi}^\pi g(x)\rho (x)dx-\int_{-\pi}^\pi g(x)d\mu_N(x) \\
= &
\int_{-\pi}^\pi g(x)\big[\rho (x)-\rho_{\bp(N)}(x)\big]dx
+\int_{-\pi}^\pi\big[g(x)-g_{\bp(N)}(x)\big]\rho_{\bp(N)}(x)dx + \frac{ g(p_n)-  g(p_1)}{2N} ,
\end{align*}
Then,~\eqref{eq:regularity} and~\eqref{eq:cd} imply that each integral above converges to $0$, and the result follows.
\end{proof}

The second result deals with the displacement of the solutions to the Bethe equations with $N$ and $n=N/2-r$ with respect to the solution with $N$ and $n=N/2$. 
Fix $r > 0$ and write henceforth $n = N/2 - r$.
For $\bp = (p_1,\dots,p_{N/2}) \in \calS_{N/2}$ and $\tilde \bp = (\tilde p_1,\dots,\tilde p_{n})\in\calS_n$, 
introduce the {\em offset displacement} $\varepsilon=\varepsilon(\bp,\tilde\bp)\in\bbR^n$ defined for $1\le j\le n$ by 
\begin{equation}\label{eq:yy}
\varepsilon_j=\begin{cases}\displaystyle N\big(\tilde p_{j}-p_{j+r/2}\big) &\text{ if $r$ is even,}\smallskip\\
\displaystyle N\Big(\tilde p_{j}-\frac{p_{j-(r-1)/2}+p_{j-(r+1)/2}}2\Big)&\text{ if $r$ is odd} \end{cases}\end{equation}
and the {\em offset function} 
$f_{\bp,\tilde\bp}(t):=\varepsilon_j\text{ if $t\in[ \tilde p_j, \tilde p_{j+1})$ for some $0\le j\le n$}.$

\begin{remark}
The difference of index in~\eqref{eq:yy} between $\tilde\bp$ and $\bp$ is made in such a way that the indices coincide when ``starting from the middle of the interval $[-\pi,\pi]$''. 
\end{remark}

\begin{remark}
Consider $\bp$ and $\tilde\bp$ given by Theorem~\ref{thm:1} for $r$ and $r+1$. Then, the solutions may be proved to be interlaced%
\footnote{The strategy is to show that the property of being interlaced is true for $\De=-\infty$ (this is a straightforward computation) and that this property does not cease to be true when increasing $\De$ continuously. 
Namely, one can prove that for any $\De<-1$, it is not possible that $p_j\le \tilde p_j\le p_{j+1}$ for every $1\le j<n$ and $\tilde p_k$ be equal to $p_k$ or $p_{k+1}$ for some $1\le k<n$. 
This is based on the fact that $\Theta(x,0)\in(-\pi,\pi)$ for any $x\in(-\pi,\pi)$, and that $G(x,y)=\Theta(x,y)-\Theta(-x,y)$ defined on $[-\pi,0]^2$ is decreasing in the first variable and increasing in the second one. The continuity of $\De \mapsto \bp,\tilde\bp$ is then used to conclude. 
We leave the details of the computation to the reader.}%
, in the sense that $p_j<\tilde p_j<p_{j+1}$ for any $1\le j<n$. We will not use this property later, but this may be useful in subsequent works.
\end{remark}

While the asymptotic behaviour of individual solutions $\bp$ is described by the continuous Bethe Equation, 
that of the offset displacement is governed by the Offset Equation, as shown in the next theorem.
\begin{theorem}\label{thm:3}
Fix $\De<-1$ and $r\ge0$. Consider two families $\bp(N) \in \calS_{N/2}$ and $\tilde\bp(N)\in \calS_{N/2-r}$ 
of solutions to the Bethe equations with parameters $\De$ and $N$ even sufficiently large.
If $\|\rho_{\bp(N)}-\rho\|=O(\tfrac1N)$ and $\|\rho_{\tilde\bp(N)}-\rho\|=O(\tfrac1N)$, then 
\begin{enumerate}
	\item $\rho\cdot f_{\bp(N),\tilde\bp(N)}\text{ converges uniformly on $[-\pi,\pi]$ to }r\cdot \tau.$
	\item There exists $C>0$ such that $|f_{\bp(N),\tilde\bp(N)}(x)|\le C|x|+O(1/N)$ for all $N$ and $x \in [-\pi,\pi]$.
\end{enumerate}
 \end{theorem}
 
 The second property is slightly technical but will be useful when integrating functions against the empirical measure of the $\tilde{\bf p}(N)$ (see Section~\ref{sec:6V}).
 \begin{proof}
We  drop $N$ and $n=N/2-r$ from the notation in the computations, except that we set $f_N=f_{\bp(N),\tilde\bp(N)}$. We treat the case $r$ even and odd separately.
Below, all quantities $O(\cdot)$ may depend on $\De$ and $r$ but are uniform in $j =1,\dots, n$ and $x \in [-\pi,\pi]$.
 
 \paragraph{Case $r$ even.} 
 First, we bound the increments of $f_N$ and show that $f_N$ is almost equal to $0$ at the origin, so as to prove the second property. 
Equation~\eqref{eq:regularity} and the bound~\eqref{eq:cd} on the increments of $\bp$ and $\tilde\bp$ (both valid due to our assumptions)
imply that 
\begin{align}\label{eq:ll}
	|\varepsilon_{j+1}-\varepsilon_j|
	\stackrel{\eqref{eq:regularity}}\leq \Big|\frac{1}{\rho(p_{j+r/2+1})} - \frac{1}{\rho(p_{j+r/2})}\Big|
	+\Big|\frac{1}{\rho(\tilde p_{j+1})} - \frac{1}{\rho(\tilde p_{j})}\Big|+O\big(\tfrac1{N}\big)
	\stackrel{\eqref{eq:cd}}=O\big(\tfrac1N\big).
\end{align}
Now, by symmetry, $p_{N/4}=-p_{N/4+1}$ (recall that $p_{N/4}$ and $p_{N/4 +1}$ are the two elements of $\bp$ closest to the origin) and $\tilde{p}_{n/2} = - \tilde{p}_{n/2+1}$ so  
\begin{equation}\label{eq:lll}
	\varepsilon_{n/2}=-\varepsilon_{n/2+1}=O\big(\tfrac1N\big).
\end{equation}
Finally, observe that $\rho_{\tilde p}$ is bounded uniformly in $N$ 
(since it converges to $\rho$ in the norm $\| \cdot \|$, it also does in the uniform norm), 
and therefore $\tilde p_{j+1} - \tilde p_j > c/N$ for all $N$ and $j$, 
where $c > 0$ is some constant independent of $N$ and $j$. 
This implies the existence of $C>0$, independent of $N$ and $j$, such that 
\begin{align}\label{eq:f_N_control}
	|f_N(x)|\le C|x| +O\big(\tfrac1N\big) \qquad \text{ for all $x \in [-\pi,\pi]$.}
\end{align}

Let us now prove the first statement - that is, the convergence of $\rho f_N$. 
In light of~\eqref{eq:ll}, we may apply the Arzela-Ascoli theorem to the sequence $(f_N)$ to extract a sub-sequential limit $f$. 
It suffices to show that $\rho f=r \cdot \tau$ to conclude.

For $N$ and $1 \leq j \leq n$, the Bethe equations applied to $p_{j+r/2}$ and $\tilde p_j$ imply
\[
\varepsilon_j = \sum_{k=1}^{N/2} \Theta(p_{j+r/2},p_k) - \sum_{k=1}^{n} \Theta(\tilde{p}_{j},\tilde{p}_k).
\]
In the first sum, we Taylor expand the terms  $\Theta(p_{j+r/2},p_{k+r/2})$ at $(\tilde p_{j},\tilde p_{k})$ for any $1\leq  k \leq n$ (while leaving the remaining terms as they are). This gives 
\begin{align}\nonumber
\varepsilon_j
= \underbrace{\sum_{k=1}^{r/2} \Theta(p_{j+r/2},p_k)+\Theta(p_{j+r/2},p_{n+1-k})}_{(1)} 
- \underbrace{\tfrac1N \sum_{k=1}^{n} \partial_1\Theta(\tilde p_{j},\tilde p_{k})\varepsilon_j}_{(2)}
-\underbrace{\tfrac1N \sum_{k=1}^n\partial_2 \Theta(\tilde p_{j},\tilde p_{k})\varepsilon_k}_{(3)}+ O\big(\tfrac1N\big).
\end{align}
The final term is due to the second order errors in the Taylor expansion; it is indeed $O\big(\tfrac1N\big)$, since it contains $O(N)$ terms of order $O\big(\frac1{N^2}\big)$. 

Fix $x \in [-\pi,\pi]$ and for each $N$ (along the subsequence for which $f_N$ tends to $f$) 
pick $\tilde p_j$ so that $x \in [\tilde p_j, \tilde p_{j+1})$. 
Then the equation displayed above offers an expression for $f_N(x)$. 
Taking $N$ to infinity, 
we find that  (1) converges to $\tfrac r2(\Theta(x,-\pi)+\Theta(x,\pi))$, 
and (2) and (3) converge to $(1-2\pi\rho(x))f(x)$ and $\int_{-\pi}^\pi \partial_2\Theta(x,y)f(y)\rho(y)dy$, respectively, 
by the definition of $f$ and the weak convergence of $\mu_N$ (defined in statement of Theorem~\ref{thm:2}).
Thus, 
$$
2 \pi f(x)\rho(x) = \tfrac r2\big(\Theta(x,-\pi) + \Theta(x,\pi)\big) - \int_{-\pi}^{\pi} \partial_2 \Theta(x,y) f(y)\rho(y) dy.
$$ 
It follows that $\frac1r f(x) \rho(x) =\tau(x)$ by the uniqueness of the solution to the Offset Equation~\eqref{eq:VarsigmaDef}.

\paragraph{Case $r$ odd.} The reasoning is similar. 
Equation~\eqref{eq:ll} may be obtained in the same way and~\eqref{eq:lll} may be replaced by  $\varepsilon_{(n+1)/2}=0$, which results from the symmetry of $\bp$ and $\tilde\bp$. One then expands around $(\tilde p_i,\tilde p_k)$ the expression
 \[
\sum_{k=1}^{n} \Theta(\tilde{p}_{j},\tilde{p}_k)-\tfrac12\big[\Theta(p_{j+(r-1)/2},p_k)+\Theta(p_{j+(r+1)/2},p_k)\big]
\]
to obtain the same result.
 \end{proof}

%

\section{Proofs of the theorems}

\subsection{Perron-Frobenius eigenvalues of six-vertex model via Bethe Ansatz}\label{sec:6V_PF}

The goal of this section is to show that the Perron-Frobenius eigenvalue of $V^{[n]}$ is given by the Bethe Ansatz from the solution $\bp$ of~\eqref{eq:BE} given by Theorem~\ref{thm:1} (recall the choice $I_j=j-\frac{n+1}2$ for $1\le j\le n$ in the theorem). 
We start by recalling the Bethe Ansatz for the transfer matrix of the six-vertex model.
A more detailed discussion (with references) and an expository proof may be found in the companion paper \cite{BetheAnsatz1}.

Recall that $\Delta = (2 - c^2)/2$ and that the function $\Theta$ depends implicitly on $\De$. For $z\ne 1$, define
\begin{equation}\label{eq:LM}
	L(z):= 1 + \frac{c^2 z}{1-z} \qquad \text{and}\qquad
	M(z):= 1 - \frac{c^2}{1-z} \, . 
\end{equation}

\begin{theorem}[Bethe Ansatz for $V$]\label{thm:BA} 
	Fix $n \leq N/2$. 
	Let $(p_1, p_2, \dots, p_n) \in (-\pi,\pi)^n$ be distinct and satisfy the equations
	\begin{align}\tag{BE}\label{eq:BA}
		\exp\left(i N p_j\right) 
		= (-1)^{n-1} \exp \left( -i\sum_{k=1}^n \Theta(p_j,p_k) \right)  \quad \forall j \in \{1, 2, \dots, n\}. 
	\end{align}
	Then, 
	$\psi = \sum_{|\vec x|=n}\psi(\vec x)\, \Psi_{\vec x} \, ,$ where $\psi(\vec x)$ is given by
	$$\psi(\vec x) 
		:= \sum_{\sigma \in \mathfrak{S}_n} A_\sigma \prod_{k=1}^n \exp\left(i p_{\sigma(k)}x_k \right) \quad\text{where}\quad
		A_\sigma := \varepsilon(\sigma)\,  \prod_{1 \leq k < \ell \leq n} e^{ip_{\sigma(k)}}\  (e^{-\icomp p_{\sigma(k)}}+e^{\icomp p_{\sigma(\ell)}}-2\Delta),
	$$
	(for $\sigma$ an element of the symmetry group $\mathfrak{S}_n$) satisfies the equation $V\psi = \Lambda \psi$, where 
	\begin{align*}
		\Lambda=\Lambda(\bp) := 		
		\begin{cases}
			\displaystyle \prod_{j=1}^n L(e^{ip_j}) + \prod_{j=1}^n M(e^{ip_j})\quad &\text{if $p_1,\dots,p_n$ are non zero,} \vspace{3pt}\\
			\displaystyle
			\Big[2+ c^2 (N-1) + c^2 \sum_{j\neq \ell} \partial_1 \Theta (0,p_{j}) \Big] \cdot \prod_{j \neq \ell} M(e^{ip_j})
			\quad &\text{if $p_\ell = 0$ for some $\ell$.}
		\end{cases}
	\end{align*}
\end{theorem}
It is {\em a priori} unclear whether $\psi$ is non-zero, so that the previous theorem does not trivially imply that $\Lambda(\bp)$ is an eigenvalue of~$V$. 
It is also unclear whether solutions of~\eqref{eq:BA} exist. Nonetheless, any solutions of~\eqref{eq:BE} do also satisfy~\eqref{eq:BA}. In particular, Theorem~\ref{thm:1} provides us with a family of solutions to~\eqref{eq:BE}, and our goal is to prove that the corresponding value $\La$ given by the theorem above is the Perron-Frobenius eigenvalue of $V^{[n]}$. 
 
Below, we will view $V^{[n]}$ as a function of $\Delta$, hence we write it $V^{[n]}_\De$. 
We begin by computing the asymptotic of the Perron-Frobenius eigenvalue of $V^{[n]}_\De$ when $\De$ tends to $-\infty$.

\begin{lemma}\label{lem:PFAtInfinity}
	Fix $r \geq 0$ and $N > 2r$  an even integer. Set $n = N/2 -r$. 
	Then the largest eigenvalue $\la$ of the matrix 
	\[
	V^{[n]}_\infty := \lim_{\Delta \rightarrow -\infty} \frac{V^{[n]}_\Delta}{(-2\Delta)^n}
	\]
	is simple and satisfies
	\begin{align}\label{eq:PFAtInfinity}
		\la  \leq 2^r \prod_{j=0}^{r-1} \Big[1 + \cos \Big(\frac{\pi(2j+1)}{n+2r}\Big)\Big],
	\end{align}
	where the empty product is set to 1.
\end{lemma}

\begin{remark}
	The matrix $V^{[n]}_\infty$ is symmetric and thus all its eigenvalues are real; its largest eigenvalue is therefore well-defined. It is not a Perron-Frobenius matrix, and thus we cannot be sure {\em a priori} that the largest eigenvalue is simple and largest in absolute value. We further note that the largest eigenvalue of $V^{[n]}_\infty$ is actually equal to the RHS of~\eqref{eq:PFAtInfinity},	as will be shown in the proof of Corollary~\ref{cor:a} below. 
\end{remark}

\begin{proof}
Fix $N\ge2r$ and $n = N/2 -r$. 
For two distinct configurations $\Psi_{\vec x}$ and $\Psi_{\vec y}$ in $\Om_n$\footnote{Recall that $\Om_n$ is the vector space generated by the $\Psi_{\vec x}$, where $\vec x$ has $n$ entries.}, recall that $V^{[n]}_\De(\Psi_{\vec x}, \Psi_{\vec y})$ 
is non-zero only when $\Psi_{\vec x}$ and $\Psi_{\vec x}$ are interlacing, and in this case it is equal to 
$$ c^{ | \{ i :\, \Psi_{\vec{x}}(i) \neq \Psi_{\vec{y}}(i)\}|} = (2 -2 \De)^{\frac12  | \{ i :\, \Psi_{\vec{x}}(i) \neq \Psi_{\vec{y}}(i)\}|}.$$
Since $\Psi_{\vec x}, \Psi_{\vec y} \in \Om_n$, 
the number $ P(\Psi_{\vec x}, \Psi_{\vec y}) = | \{ i :\, \Psi_{\vec{x}}(i) \neq \Psi_{\vec{y}}(i)\}|$ is at most $2n$. 
The normalization $(-2 \Delta)^n$ is chosen to ensure that, for any pair of configurations $\vec x$ and $\vec y$ as above, 
\[
V^{[n]}_\infty(\Psi_{\vec x}, \Psi_{\vec y}) = 
\begin{cases} 1 \quad &  \text{if } P(\Psi_{\vec x}, \Psi_{\vec y}) = 2n, \\ 0 & \text{otherwise}.  \end{cases}
\]

If $\vec x$ and $\vec y$ are configurations as above with $V^{[n]}_\infty(\Psi_{\vec x}, \Psi_{\vec y}) =1$, 
then $\Psi_{\vec x}$ has no consecutive up-arrows (and by symmetry neither does $ \Psi_{\vec y}$).
Indeed, if we suppose that $\Psi_{\vec x}$ has at least two consecutive up-arrows, 
 then interlacement requires $\Psi_{\vec y}$ to have an up-arrow above at least one of the consecutive up-arrows of $\vec x$,
which induces $ P(\Psi_{\vec x}, \Psi_{\vec y}) < 2n$ and therefore $V^{[n]}_\infty(\Psi_{\vec x}, \Psi_{\vec y}) = 0$.
Thus, to study $V^{[n]}_\infty$, we may study its restriction to the set of configurations with no consecutive up-arrows. 

In the case $n=N/2$, there is only one pair of such configurations: the completely staggered configurations -- i.e. those with alternating up and down arrows. Hence, $V^{[N/2]}_\infty$ breaks down into a block-diagonal structure: a $2 \times 2$ block of the form $\left(\begin{array}{cc} 0 & 1 \\ 1 & 0\end{array} \right)$, and a $\left[{N \choose N/2} -2\right]$-dimensional block of 0's. The spectral structure of this matrix is very straightforward - there are simple eigenvalues at $\pm 1$, and all other eigenvalues are 0, as required.

For $n = N/2 - r$, the situation is more complicated, 
and a direct computation of the spectrum of $V_\infty^{[n]}$ is best avoided. However, we do have the tools to bound the dominant eigenvalue. 
 
The set of configurations with no consecutive up-arrows can be parameterized by the location of the $2r$ ``defects" -- i.e. coordinates $i$ with a down arrow preceded by another down arrow. By periodicity, we say that $\vec x$ has a defect at $1$ if $\vec x$ has a down arrow at $1$ and at $N$. 

It is straightforward to show that a configuration with $n$ up-arrows has no consecutive up-arrows if and only if there are exactly $2r$ defects whose parities alternate. Moreover, $\vec x$ and $\vec y$ are such that $V^{[n]}_\infty(\Psi_{\vec x}, \Psi_{\vec y}) = 1$ if and only if $\Psi_{\vec x} \neq \Psi_{\vec y}$ and the locations of the defects of $\vec y$ may be obtained from those of $\vec x$, by moving each defect by precisely one unit on the left or on the right (taken toroidally). 
Since the parity of the defects alternates in both states, no two defects can exchange positions between $\vec x$ and $\vec y$.
See Fig.~\ref{fig:defects} for an example. 

\begin{figure}
		\begin{center}
		\includegraphics[width=0.47\textwidth, page=1]{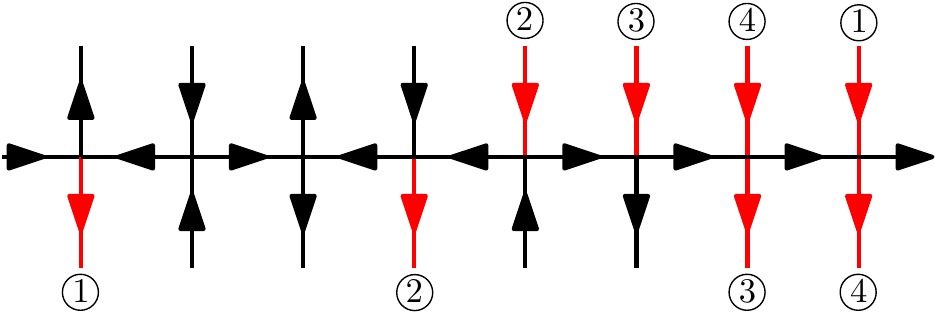}\qquad
		\includegraphics[width=0.47\textwidth, page=2]{defects.pdf}
		\caption{\emph{Left}: Two configurations $\vec x$ and $\vec y$ with $N=8$ and $r = 2$. 
		The defects are marked by red arrows and are numbered. Notice that each defect has moved by one unit when going from $\Psi_{\vec x}$ (below) to $\Psi_{\vec y}$ (above), but none have exchanged places. 
		\emph{Right}: The $2r$ paths corresponding to the evolutions of the defects.}
		\label{fig:defects}
		\end{center}
\end{figure}
Write $\tilde \Om_n$ for the subspace of $\Om_n$ generated by basis vectors $\Psi_{\vec x}$ with no two consecutive up arrows. 
Then, $V^{[n]}_\infty$ leaves this space stable, and we may consider its restriction to $\tilde \Om_n$. A straightforward computation shows that this matrix is irreducible, in the sense that, for any $\Psi_{\vec x}, \Psi_{\vec y} \in \tilde \Om_n$, there exists $K$ such that $[V^{[n]}_\infty ]^K (\Psi_{\vec x},\Psi_{\vec y}) >0$. As any symmetric irreducible matrix, it is either aperiodic or of period $2$; a more precise analysis can show that the latter occurs in the case of $V^{[n]}_\infty$.
Thus, the Perron-Frobenius theorem for irreducible (but not aperiodic) matrices guarantees that the largest eigenvalue is simple and maximizes the absolute value; the {\em smallest} eigenvalue actually  has the same absolute value as the largest, unlike for true Perron-Frobenius matrices.

To determine $\la$, the largest eigenvalue, consider the following related construction.
Let $M$ be an even integer and $(a_1, \dots a_{2r})$ be an ordered set of integers between 1 and $N$ of alternating parity. Consider families of $2r$ paths on $\bbZ/N\bbZ$ denoted $\{X_j(t): 0\leq t \leq M;\, j = 1,\dots, 2r\}$  such that, for each $j$, $X_j(0) = X_j(M) = a_j$ and $|X_j(t+1) - X_j(t)| =1$ for $1 \leq t <M$. Additionally, impose that the paths $X_1,\dots, X_n$ are {\em non-intersecting}, in the sense that no pair of adjacent paths ever exchange position. 
Let $Z(M;a_1, \dots a_{2r})$ be the number of such paths, and $Z(M)$ the sum of $Z(M;a_1, \dots a_{2r})$ over all admissible $(a_1, \dots a_{2r})$. 
The discussion above indicates that 
\[
Z(M) = \Tr\left( [V^{[n]}_\infty]^M \right), 
\]
which in turn implies that the largest eigenvalue (in absolute value) of $V^{[n]}_\infty$ is given by 
$$ \la = \lim_{M \to \infty} Z(M)^{1/M}.$$

Families of non-intersecting paths as those appearing in the definition of $Z(M)$ have been studied before, in particular in the work of Fulmek \cite{Fulmek}, which enables us to compute the asymptotic of $Z(M)$ directly. 
Fulmek enumerates the number of vertex-avoiding paths (i.e. families of paths as above, but such that no two ever hit the same vertex, rather than not intersecting). Luckily, the two are closely related: consider the transformation of a set of paths $\{X_j(t):\, t,j\}$ as above to the set of paths $\{\tilde{X}_j(t): \, t,j\}$ on $\bbZ/(N+2r)\bbZ$, where 
$$\tilde{X}_j(t) =  X_j(t) + j, \qquad \forall 1 \leq t\leq M,\, 1\leq j\leq 2r.$$ 
One may check that this transformation induces a bijection between the set of non-intersecting paths starting and ending at $(a_1, \dots a_{2r})$ on $\bbZ/N\bbZ$ and that of vertex-avoiding paths starting and ending at $(a_1 +1, \dots a_{2r}+2r)$ on $\bbZ/(N+2r)\bbZ$. Note that, while vertex-avoiding paths are generally allowed to intersect, the parity constraints of the starting positions prevents them from doing so in this case (more precisely, observe that $\tilde{X}_{j+1} (t) - \tilde{X}_{j} (t)$ is even for all $t$ and $j$). 

Since we may get from any admissible starting position (that is, with even spacing between the starting points) 
to the position $(2,4,\dots,4r)$ in at most $N$ steps, the limit of interest to us may be computed as 
$$  \lim_{M\to \infty} Z(M)^{1/M} = \lim_{M\to \infty}Z(M; 2,4,\dots, 2r)^{1/M}.$$
We now state Corollary~7 of \cite{Fulmek}, 
which provides an exact expression of $Z(M; 2,4,\dots, 2r)$ as the determinant of a matrix of size $2r$:
\begin{align*}
Z(M; 2,4,\dots, 2r) 
= (N+2r)^{-2r}\det \left( \xi^{i-j}\sum_{\ell = 0}^{N+2r-1 } \xi^{2(i-j)\ell}\left[2\cos \left(\frac{\pi(2\ell + 1)}{N+2r} \right)\right]^M \right)_{1\leq i,j \leq 2r},
\end{align*}
where we set $\xi = e^{i\frac{2\pi}{N+2r}}$.
Since we are only interested in $\lim_{M \to \infty} Z(M; 2,4,\dots, 2r)^{1/M}$, we can simply study the dominating terms (as $M \to \infty$) in the Leibniz formula for the determinant on the right-hand side. 
However, the apparently maximal terms cancel out in the computation of the determinant, and some care is needed. 

To start, observe that the entries of the matrix above may be rewritten by grouping the terms $\ell$ and $\ell + N/2 + r$ (which are equal) 
as a sum with only half the terms: 
\begin{align*}
	2 \xi^{i-j}\sum_{\ell = 0}^{n+2r-1} \xi^{2(i-j)\ell}\Big[2\cos \Big(\frac{\pi(2\ell + 1)}{N+2r}\Big)\Big]^M.
\end{align*}
Then, we write the determinant out as
\begin{align*}
	(N+2r)^{2r} 2^{-2r(M+1)}Z(M; 2,4,\dots, 2r) = 
	\hspace{-8pt}\sum_{\substack{\si \in \mathfrak S_{2r}\\ 0 \leq \ell_1,\dots, \ell_{2r} < n+2r}}\hspace{-8pt}
	\eps(\si)\prod_{i=1}^{2r} \xi^{(i - \si(i))(2\ell_i+1)}\Big[\cos \Big( \frac{\pi(2\ell_i + 1)}{N+2r} \Big)\Big]^M.
\end{align*}
In the above, note that the term taken to the power $M$ does not depend on $\si$. 
We conclude that
\begin{align}\label{eq:term_to_maximise}
	\lim_{M\to \infty}Z(M)^{1/M} = 2^{2r}\prod_{i=1}^{2r}  \Big|\cos \Big( \frac{\pi(2\ell_i + 1)}{N+2r}\Big)\Big|,
\end{align}
where  $\ell_1,\dots, \ell_{2r} \in \{0,\dots,n+2r-1\}$ maximise the product above and are such that 
\begin{align}\label{eq:term_non_zero}
	\sum_{\si \in \mathfrak S_{2r}} \eps(\si)\prod_{i=1}^{2r} \xi^{(i - \si(i))(2\ell_i+1)} \neq 0.
\end{align} 
Consider $\ell_1,\dots, \ell_{2r}$ as above with $\ell_j = \ell_{j'}$ for some $j\neq j'$ and a permutation $\si$. 
Write $\tau_{j,j'}$ for the transposition of $j$ and $j'$. 
The sum of the terms corresponding to $\ell_1,\dots, \ell_{2r}$ with $\si$ and $\si \circ \tau_{j,j'}$ sum up to $0$, 
and we find that the term in~\eqref{eq:term_non_zero} is zero.

Thus, we may limit ourselves to terms with $\ell_1,\dots, \ell_{2r}$ all distinct. 
Among such sets, one maximising the product in~\eqref{eq:term_to_maximise} is given by $\ell_i = i -1$ for $i \leq r$ and $\ell_i = n+2r -i$ for $i > r$. 
For this set, we find that the term in~\eqref{eq:term_to_maximise} is equal to 
\begin{align*}
	2^{2r}\prod_{i=1}^{r}  \left[\cos \left(\frac{\pi(2\ell_i + 1)}{N+2r} \right)\right]^2
	= 2^{r}\prod_{j=0}^{r-1} \left[1 +\cos\left( \frac{\pi(2j+1)}{n + 2r} \right) \right].
\end{align*}
This does not prove that $\lim_{M \to \infty} Z(M)^{1/M}$ is equal to the above, since \eqref{eq:term_non_zero} may not be satisfied. 
It does, however, show the claimed inequality.
\end{proof}

\begin{corollary}\label{cor:a}
Fix $\De_0<-1$ and $r\ge0$. Then, for $N$ large enough, 
the Perron-Frobenius eigenvalue of $V_N^{[N/2-r]}$ for $\De_0$ is given by 
$\Lambda(\bp_{\De_0})$, where $(\bp_\De)_{\De\le\De_0}$ is the family given by Theorem~\ref{thm:1} applied to $\De_0$ and $r$. 
\end{corollary}

\begin{proof}
Fix $\De_0<-1$, $r\ge0$ and let $N$ be large enough for Theorem~\ref{thm:1} to apply. 
Write $n = N/2 - r$. Since $N$ is fixed, we drop it from the notation. 

The dependency of the Perron-Frobenius eigenvalue of $V_\Delta^{[n]}$ on $\De$ will be important, and we therefore denote it by $\Lambda_r(\De)$. 
Also, write $\psi(\bp_\De)$ for the vector given by Theorem~\ref{thm:BA} for the solution $\bp_\De$ to~\eqref{eq:BE}.
We wish to prove that $\Lambda_r(\De) = \Lambda(\bp_\De)$ for $\De = \De_0$. We will prove more generally that this is true for all $\De \leq \De_0$.

First, observe that the Perron-Frobenius eigenvalue of a family of irreducible symmetric matrices varying analytically in a parameter (here $\De$) varies analytically in this parameter as well (since it is a simple zero of the characteristic polynomial). Therefore, $\Lambda_r(\De)$ is an analytic function. 
Since $\De\mapsto\bp_\De$ is analytic, we deduce that $\De\mapsto \Lambda(\bp_\De)$ also is,  so that it is sufficient to show that $\Lambda(\bp_\De)=\Lambda_r(\De)$ for $\De$ small enough in order 
to conclude that the two are equal for all $\De \leq \De_0$. 
To do this, we shall prove two facts, namely that 
\begin{itemize}[noitemsep, nolistsep]
\item $\psi(\bp_\De)$ is non-zero for $\De$ small enough (which implies that $\Lambda(\bp_\De)$ is an eigenvalue of $V^{[n]}_\De$ for the corresponding values of $\De)$, 
\item $\lim_{\De \rightarrow -\infty}\frac1{(-2\Delta)^n} \Lambda(\bp_\De)$ 
is the largest eigenvalue of $V^{[n]}_\infty$ (defined in Lemma~\ref{lem:PFAtInfinity}).
\end{itemize}
These two facts indeed prove the result:
since the largest eigenvalue of $V^{[n]}_\infty$ is simple, 
by continuity of $\De \mapsto \Lambda(\bp_\De)$ and $\De \mapsto V^{[n]}_\De$, we deduce that $\Lambda(\bp_\De)$ is the largest eigenvalue of $V^{[n]}_\De$ for $\De$ small enough. However, for finite $\De$, $V^{[n]}_\De$ is a Perron Frobenius matrix, and $\Lambda(\bp_\De)$ is then its Perron Frobenius eigenvalue. The observation of the previous paragraph is then sufficient to conclude.


The rest of the proof is dedicated to the two facts listed above. 
Recall that, at $\Delta = - \infty$, we have a simple formula for $\bp$, namely 
\[
p_j = \frac{2 \pi I_j}{N -n} \qquad \text{ for all $1\leq j \leq n$}.
\]
For the rest of the proof, write $\zeta = e^{2 \pi i/(N-n)}$. 

We start with the study of $\psi(\bp_\De)$.
Set $\psi_\infty:= \lim_{\De \to -\infty}(-2\De)^{- \frac{n(n-1)}2}\psi(\bp_\De)$. 
It suffices then to prove that $\psi_{\infty}$ has at least one non-zero coordinate, and we shall do so for the coordinate $\psi_\infty(2,4, \dots, 2n)$. First, we need to study the asymptotics of the coefficients $A_\si$ appearing in the definition of $\psi$. For $\si \in \mathfrak{S}_n$, as $\De \to -\infty$, 
$$ A_{\si} = \eps(\si) \prod_{1\le j<k\le n}\left[-2\De \zeta^{\si(j) - \frac{n+1}2}\right] + o(\De^{\frac{n(n-1)}2}).$$
By injecting this into the definition of $\psi$,  we find that, 
\begin{align*}
	\psi_{\infty}(2,\dots,2n) 
	&= \sum_{\sigma \in \mathfrak{S}_n} A_\sigma \prod_{k=1}^n \exp(i p_{\sigma(k)} \cdot 2k ) \, \\
	&= \sum_{\sigma \in \mathfrak{S}_n} \eps(\si)   \Big(\prod_{1 \leq j<k \leq n}\zeta^{\si(j) - \frac{n+1}2}\Big)\times\prod_{k=1}^n 
	 \zeta^{2(\si(k)-\frac{n+1}2)k}
	 \\
	 &= \zeta^{ - \frac{1}4 (n+1)^2n}
	 	\sum_{\si\in\mathfrak{S}_n} \eps(\si)\zeta^{\sum_{j=1}^n\si(j)j}.
\end{align*}
In the sum above, we recognise the determinant of the matrix $\big( \zeta^{j\cdot k}\big)_{1 \leq j,k\leq n}$. This is the Vandermonde matrix corresponding to the values $\zeta, \zeta^2, \dots, \zeta^n$, which are all distinct (since $2n \leq N$). Thus, 
$$ \psi_{\infty}(2,\dots,2n) = \lim_{\De \to -\infty}(-2\De)^{- \frac{n(n-1)}2} \psi_\De(2,\dots,2n)  \neq 0.$$

We now turn to the study of $\lim_{n\to \infty}(-2\Delta)^{-n} \Lambda(\bp_\De)$ (we show below that this limit exists). 
Before starting, we mention that, since $\psi_\infty \neq 0$, the above limit is an eigenvalue of $V^{[n]}_\infty$. 
With this and Lemma~\ref{lem:PFAtInfinity} in mind, it suffices to prove that it is equal to the RHS of~\eqref{eq:PFAtInfinity} to deduce that it is the largest eigenvalue of $V^{[n]}_\infty$. We do this below. 

The functions $L$ and $M$ defined in~\eqref{eq:LM} depend on $\De$ and degenerate when $\De \to -\infty$. 
However, we have 
\begin{align*}
		\frac{1}{-2\De}L(z)\xrightarrow[\De \to -\infty]{} \frac{z}{1-z} 
		\quad \text{ and } \quad 
		\frac{1}{-2\De}M(z)\xrightarrow[\De \to -\infty]{} \frac{-1}{1-z} 
		\qquad \forall z\in [-\pi,\pi]\setminus\{0\}.
\end{align*}
Therefore, we find that
\begin{equation}\label{eq:LimitEigenvalue}
\frac{\Lambda(\bp_{\Delta})}{(-2\Delta)^n} 
\xrightarrow[\De \to -\infty]{} \begin{cases} 
\frac{2}{\prod_{j=1}^n (1 - \zeta^{j - (n+1)/2})} & \text{if $n$ is even}, \\
(N - n) \times  \prod_{j =1, j \neq (n+1)/2}^n \left( \frac{1}{1 - \zeta^{j - (n+1)/2}} \right) & \text{if $n$ is odd}.
\end{cases}
\end{equation}
(Recall that $\Theta_{-\infty}(x,y)=y-x$, $c^2$ behaves like $-2\Delta$, and $\zeta^{N-n}=1$.) When $n$ is an even number, the decomposition of the polynomial $x^{N-n} -1$ reads
\[
x^{N-n} -1 = \prod_{j=1}^{N-n} \left(x - \zeta^{j - n/2} \right). 
\]
Thus, if we multiply the numerator and denominator in~\eqref{eq:LimitEigenvalue} by the terms corresponding to $j=n+1$ to $N-n$ and apply the above to $x = \zeta^{1/2}$, we find that 
\begin{align*}
\frac{2}{\prod_{j=1}^n (1 - \zeta^{j - (n+1)/2})} 
& = \frac{2 \times \prod_{j=n+1}^{N-n} \left(1 - \zeta^{j - (n+1)/2}\right)}{\zeta^{-(N-n)/2} \left( \zeta^{(N-n)/2} - 1\right)} \\
& = \prod_{j=n+1}^{N-n} \left(1 - \zeta^{j - (n+1)/2}\right) \\ 
& = \prod_{j=0}^{2r-1} \left(1 - \zeta^{j+ (n+1)/2}\right) \\ 
& = \prod_{j=0}^{r-1} \left[2 - 2\cos\left(\frac{\pi(2j+n+1)}{N-n} \right) \right] \\ 
& = 2^r\prod_{j=0}^{r-1} \left[1 +\cos\left( \frac{\pi(2j+1)}{N-n} \right) \right],
\end{align*}
where in the second equality we used that $\zeta^{(N-n)/2}=-1$, in the third we changed $j$ to $N-n-j$ and used that $N-n = n+2r$, in the fourth we grouped the $j=k$ and $j = 2r -1 - k$ terms together. The last equality follows again from the fact that $N-n=n+2r$ and changing  $j$ to $r-1-j$.
This matches the expression in~\eqref{eq:PFAtInfinity}, as required. 

We use a similar strategy when $n$ is odd. Noting that
\[
\prod_{\substack{1 \leq j \leq N-n \\ j \neq (n+1)/2}} \left({1 - \zeta^{j - (n+1)/2}} \right)
= \lim_{x \to 1}\frac{x^{N-n}-1}{x-1} = N-n,
\]
we may perform a similar computation to obtain again
$$
(N - n) \times  \prod_{j =1, j \neq (n+1)/2}^n \left( \frac{1}{1 - \zeta^{j - (n+1)/2}} \right) 
= 2^r\prod_{j=0}^{r-1} \left[1 +\cos\left( \frac{\pi(2j+1)}{N-n} \right) \right].
$$
\end{proof}

\begin{remark}\label{rmk:gold}
The analyticity of $\De \mapsto \bp_\De$ allows us to avoid using a highly non-trivial fact (which would be necessary would we have continuity only), namely that for each $\De$, $N$ and $n$, the vector obtained by the Bethe Ansatz from the solution $\bp_\De$ to~\eqref{eq:BE} is non-zero. 
This is necessary to deduce that the associated value $\Lambda(\bp_\De)$ is indeed an eigenvalue of the transfer-matrix.
Let us mention that Goldbaum proves that the vector obtained by the Bethe Ansatz for the 1D Hubbard model is indeed non-zero for every $\Delta$. The proof relies on a symmetry of the model which is not satisfied by the six-vertex model. Kozlowski claims a similar result for the XXZ chain in \cite{Koz15}.  \end{remark}

\subsection{From the Bethe Equation to the six-vertex model: proof of Theorem~\ref{thm:6V}}\label{sec:6V}

The goal of this section is the proofs of Theorem~\ref{thm:6V} and Corollary~\ref{cor:6V}. 

\begin{proof}[Theorem~\ref{thm:6V}] 
We divide the proof in three steps. We first treat relation \eqref{eq:aggf}. 
We then focus on \eqref{eq:aggg} with $r>0$, even, and finally treat the case of \eqref{eq:aggg} with $r>0$, odd. Note that \eqref{eq:aggg} with $r<0$ follows directly from $r>0$ since the transfer matrix $V$ is invariant under global arrow flip, and therefore, the spectrums of $V$ on $\Omega_n$ and $\Omega_{N - n}$ are identical. 

Fix $c > 2$ and recall that $\De = \frac{2-c^2}2 < -1$. 
Generically, in this proof $\bp=\bp_\De(N)$ and $\tilde\bp=\tilde\bp_\De(N)$ are given by Theorem~\ref{thm:1} applied to $\De_0=\De$ and $n=N/2$ and $N/2 - r$ respectively. We will always assume $N$ to be a multiple of $4$ (in particular $N/2$ is even). 
For clarity, we will drop $N$ and $\De$ from the notation and write $n = N/2 - r$. 

\paragraph{Proof of \eqref{eq:aggf}.}
The Bethe Ansatz and Corollary~\ref{cor:a} imply that 
$$\Lambda_0(N):=2 \prod_{j=1}^n |M(e^{ip_j})|,$$
where we used above that $\bp$ is symmetric with respect to the origin and that $L(z)=M(\overline z)$ for $|z|=1$
to deduce both products in the expression in Theorem~\ref{thm:BA} are equal to the product of the $|M|$.
By Theorem~\ref{thm:2}, we deduce\footnote{One should be wary of the log singularity at 0 of $\log |M|$. However, since $\log |M|$ is in $L^1[(-\pi,\pi)]$, standard truncation techniques are sufficient to show the convergence of the sum to the integral above. In particular one uses that the $p_j$'s are well-separated -- that is that $p_{j+1} - p_j \geq {\pi}/N$ for all sufficiently large values of $N$, which follows from \eqref{eq:BE} and the monotonicity of $\Theta$ -- to ensure that there are not too many $p_j$'s near the origin.} that
\begin{equation}\label{eq:ag}
\lim_{N\rightarrow\infty}\frac1N\log \Lambda_0(N)= \int_{-\pi}^\pi \log \left|M\left(e^{ix}\right)\right| \rho(x) dx.
\end{equation}
The explicit form of $\rho$ enables us to compute this integral explicitly via Fourier analysis (see Section~\ref{sec:Fourier} for details) to obtain the result.

\paragraph{Proof of \eqref{eq:aggg}, case $r>0$ even.}

In this case, both $N/2$ and $n$ are even, so that the Bethe Ansatz together with Corollary~\ref{cor:a} imply that
\begin{equation}\label{eq:EigenvalueRatio}
\frac{\Lambda_r(N)}{\Lambda_0(N)}  = \underbrace{\prod_{j=1}^{n} \frac{|M\left(e^{\icomp \tilde{p}_{j}}\right)|}{|M\left(e^{\icomp p_{j+r/2}}\right)|}}_{(1)} \cdot \underbrace{\Big(\prod_{j=1}^{r/2}|M\left(e^{\icomp p_j}\right)|\Big)^{-2}}_{(2)},
\end{equation} 
where again we used that $\bp$ is symmetric with respect to the origin to group the two products into a single one. 

We study the two terms separately.
The term (2) converges to $|\Delta|^{-r}$, since $M$ is continuous, $M(-1)=\Delta$ and the first $r/2$ coordinates of $\bp$ converge to $-\pi$ as $N \to \infty$. 
As for the first term, by taking the logarithm and using that $\mu_N$ converges weakly (by Theorem~\ref{thm:2}) and $f_{\bp(N),\tilde\bp(N)}$ converges uniformly (by Theorem~\ref{thm:3}), we deduce that (1) converges to 
\begin{equation}\label{eq:agg}
	\exp\Big(r\int_{-\pi}^{\pi}\ell'(x)\tau(x) dx\Big),
\end{equation}
where $\ell(x):=\log |M(e^{\icomp x})|$. Note that $\ell'(x)$ behaves like $1/|x|$ near the origin. Nonetheless, this does not raise any issue here since by Theorem~\ref{thm:3}, $f_{\bp(N),\tilde\bp(N)}(x)\le C|x|$ uniformly in $N$; thus, $\ell'(x) \tau(x)$ is uniformly bounded, and the weak convergence applies. 

The explicit forms of $\tau$ and $\ell$ lead to the expression in the statement of Theorem~\ref{thm:6V}, thus concluding the proof. The relevant computation is based on Fourier analysis and is deferred to Section~\ref{sec:Fourier}.

\paragraph{Proof of \eqref{eq:aggg}, case $r>0$ odd.}

In this case $N/2$ is even and $n$ is odd. The Bethe Ansatz and Corollary~\ref{cor:a} imply that
\begin{align*}
\frac{\Lambda_r(N)}{\Lambda_0(N)}  &=\underbrace{\frac{\displaystyle 2+ c^2 (N-1) + c^2 \sum_{j\ne (n+1)/2} \partial_1 \Theta (0,\tilde p_{j})}{2}}_{(A)}\underbrace{\frac{\displaystyle\prod_{\substack{j=1:j\ne (n+1)/2}}^n|M\left(e^{\icomp \tilde{p}_{j}}\right)|}{\displaystyle\prod_{j=1}^{N/2}|M\left(e^{\icomp p_j}\right)|}}_{(B)}.
\end{align*}
(The 2 in the denominator of the first fraction comes from the fact that $\Lambda_0(N)$ involves two products, whereas $\La_r(N)$ only contains one.) First, observe that the weak convergence of $\tilde p_j$ and~\eqref{eq:RhoDef} imply that 
\begin{equation*}
(A)= \frac{Nc^2}2\Big(1 + \int_{-\pi}^\pi\pd_1\Theta(0,x)\rho(x)dx + o(1)\Big) =
\frac{c^2}22\pi\rho(0) N + o(N).
\end{equation*} 
We now focus on (B) and divide it into four terms
\begin{align*}
(B)&=\underbrace{\prod_{\substack{j=1\\ j\ne (n+1)/2}}^{n}\frac{|M\left(e^{\icomp \tilde{p}_{j}}\right)|}{|M\left(e^{\icomp \hat p_{j}}\right)|}}_{(1)}\cdot  \underbrace{\Big(|M(e^{\icomp p_{(r+1)/2}})| \prod_{j=1}^{(r-1)/2}|M\left(e^{\icomp p_j}\right)|^2\Big)^{-1}}_{(2)}\cdot\underbrace{|M(e^{\icomp p_{N/2}})|^{-1}}_{(3)}\\
&
\qquad\cdot \underbrace{\prod_{j=1}^n\frac{|M\left(e^{\icomp \hat p_{j}}\right)|}{\sqrt{|M\left(e^{\icomp p_{j+(r-1)/2}}\right)||M\left(e^{\icomp p_{j+(r+1)/2}}\right)|}}}_{(4)}
\end{align*} 
where $\hat p_j=\tfrac12(p_{j+(r-1)/2}+p_{j+(r+1)/2})$.
To obtain the terms (2) and (3), we have used that $p_{N/2+1-j}=-p_j$. 
The same arguments as in the previous case imply that (1) converges to 
$\exp(r\int_{-\pi}^{\pi}\ell'(x)\tau(x)dx)$
and (2) to $|\Delta|^{-r}$. 
Furthermore, symmetry and~\eqref{eq:regularity} imply\footnote{We used that $\displaystyle p_{N/2+1}=\tfrac12(p_{N/2+1}-p_{N/2})= \frac{1}{2\rho(0)N}+O\Big(\frac1{N^2}\Big)$ and $\displaystyle p_{N/2+{k+1}}-p_{N/2+k}= \frac{1}{\rho(0)N}+O\Big(\frac{k}{N^2}\Big)$.} that for each fixed~$k$,
$$p_{N/2+k}=\frac{k-1/2}{\rho(0)N}+O\Big(\frac{k^2}{N^2}\Big),$$
where $O(\cdot)$ is uniform in $k$ and $N$.
Since $|M(e^{ip})|= c^2/|p| +o(1)$ for $p$ close to 0, we deduce that  
\begin{align*}
(3)& = \frac{1}{2\rho(0)c^2 N} + o\Big(\tfrac1{N}\Big),\\
(4)& = \prod_{k=N/2+1}^\infty \frac{4p_{k}p_{k+1}}{(p_{k}+p_{k+1})^2} +o(1) = \prod_{k=1}^\infty\Big(1-\frac{1}{4k^2}\Big) + o(1)=\frac2{\pi} + o(1).
\end{align*}
(In approximating (4), we used~\eqref{eq:regularity} to control $p_{N/2+k+1}-p_{N/2+k}$ for $k\ge N^{1/2}$.) 
Combining the estimates above and appealing to the computation of $\int_{-\pi}^{\pi}\ell'(x)\tau(x)dx$ in Section~\ref{sec:Fourier}, we obtain the expected result.
\end{proof}

We now prove Corollary~\ref{cor:6V}. The proof consists in two steps. We first prove that the free energy exists and that it is related to the sum of weighted configurations that are ``balanced''. We then relate the latter to the rate of growth of the Perron-Frobenius eigenvalue of $V_N^{[N/2]}$. 

The proof of the existence of the free energy is slightly tedious due to the fact that the six-vertex model does not enjoy the finite-energy property. Nevertheless, it is close in spirit to corresponding proofs for other models.   
The next section enables us to deduce the existence of the limit along $N$ and $M$ even using the connection to the random-cluster model.
Nonetheless, we believe that a direct proof is of value.

The proof below is not connected to other arguments in this paper except through its result. We encourage the reader mostly interested in Theorems~\ref{thm:Potts} and~\ref{thm:RCM} to skip this proof.

\begin{proof}[Corollary~\ref{cor:6V}]
{\bf Step 1: Existence of free energy.}
For $N,M \in \bbN$, let $\bfR_{N,M}$ be the subgraph of $\bbZ^2$ with vertex set $V(\bfR_{N,M}) =\{1,\dots,N\} \times \{1,\dots, M\}$ and edge-set $E(\bfR_{N,M})$ formed of all edges of $\bbZ^2$ with both endpoints in  $V(\bfR_{N,M})$. 
Define the {\em edge-boundary} of $\bfR_{N,M}$ as the set $\pd_e \bfR_{N,M}$ of edges of $\bbZ^2$ with exactly one endpoint in $V(\bfR_{N,M})$.

A six-vertex configuration on $\bfR_{N,M}$ is an assignment of directions to each edge of $E(\bfR_{N,M}) \cup \partial_e \bfR_{N,M}$. For such a configuration $\vec{\om}$, the weight is computed as on the torus:
$$
	w(\vec{\om}) = a^{n_1 + n_2} \cdot b^{n_3 + n_4} \cdot c^{n_5 + n_6}, 
$$
where $n_1, \dots, n_6$ are the numbers of vertices of $\bfR_{N,M}$ of types $1, \dots, 6$ respectively
(as on the torus, we implicitly assign weight $0$ to configurations not obeying the ice rule).
As in the rest of the paper, we fix $a=b=1$ and $c > 0$. 

A {\em boundary condition} $\xi$ for $\bfR_{N,M}$ is an assignment of directions to each edge of $\partial_eR_{N,M}$.
Let 
$$ Z_{N,M}^{\xi} = \sum_{\vec{\om}} w(\vec{\om})\ind_{\{\vec{\om}(e) = \xi(e)\, \forall e \in \pd_e \bfR_{N,M}\}}. $$
Here, we are effectively summing only over configurations which agree with $\xi$ on the edge-boundary. 
Observe that configurations obeying the ice-rule and consistent with $\xi$ 
exist only when $\bfR_{N,M}$ has as many incoming as outgoing edges in $\xi$. 

Some boundary conditions $\xi$ is called {\em toroidal} if $\xi(e) = \xi(f)$ for any boundary edges $e$ and $f$ of $\bfR_{N,M}$
such that $f$ is a translate of $e$ by $(0,M)$ or $(N,0)$. 
When $N$ is even, some toroidal boundary conditions $\xi$ is called {\em balanced} if it
contains exactly $N/2$ up arrows on the lower row of $\pd_e \bfR_{N,M}$.
Using this notation, the partition function of the six-vertex model on $\bbT_{N,M}$ may be expressed as 
\begin{align*}
	Z_{N,M} = \sum_{\vec{\om} \in \bbT_{N,M}} w(\vec{\om}) = \sum_{\xi:\, \text{toroidal}} Z^{\xi}_{N,M},
\end{align*}
where the second sum is over all toroidal boundary conditions $\xi$ on $\bfR_{N,M}$.
Moreover, set 
\begin{align*}
	Z_{N,M}^{\rm (bal)} = \sum_{\xi:\, \text{balanced}} Z^{\xi}_{N,M},
\end{align*}
the sum now being only over balanced toroidal boundary conditions. 
%
Our goal is to prove that the following limits exist:
	\begin{align}\label{eq:agr}
		\lim_{N,M\rightarrow\infty} \frac1{MN} \log Z_{N,M} 
		=	\lim_{\substack{N,M\rightarrow\infty\\ N\text{ even}}} \frac1{MN} \log Z_{N,M}^{\rm (bal)}.
	\end{align}
Above, the limits can be taken in whichever order we desire. We leave it as a simple exercise to the reader to check that~\eqref{eq:agr} can be easily deduced from the following lemma.

\begin{lemma}\label{lem:bc}
	(i)
	The following inequality holds: $$Z_{2N,2M}\ge   Z_{2N,2M}^{\rm (bal)} \geq (\tfrac1{16})^{M+N}\big(Z_{N,M}\big)^4 .$$
	(ii) 
	There exists $C>0$ such that for all integers $n>N$ and $m>M$ with $n$ and $N$ even,  
	$$ 
		\frac1{nm}\log Z^{\rm (bal)}_{n,m} 
		\geq \frac1{MN} \log Z^{\rm (bal)}_{N,M} 
		-C\Big(\frac Nn+\frac Mm\Big).
	$$  
\end{lemma}

\begin{remark}
This lemma may also be used to show that the free energy of the six-vertex model with ``free boundary conditions'' (i.e.~with partition function $\sum_{\xi} Z^{\xi}_{N,M}$ with sum over all boundary conditions) is equal to $f(1,1,c)$. 
\end{remark}

\begin{proof}[Lemma~\ref{lem:bc}]
	\noindent (i)
	Before proceeding to the proof, observe that the weight of a configuration is invariant under 
	horizontal and vertical reflections and rotations by $\pi$ of the configuration, as well as under the inversion of all arrows. 
	It follows that if $\xi$ is some boundary conditions on some rectangle $\bfR_{N,M}$
	and $\xi'$ is the boundary conditions obtained from $\xi$ via one of the operations mentioned above, 
	then 
	$$ Z^{\xi'}_{N,M} = Z^{\xi}_{N,M}.$$
	
	Let $N,M$ be integers and $\xi$ be boundary conditions on $\bfR_{N,M}$.
	The construction below is described in Fig.~\ref{fig:bcT}.
	\begin{figure}
		\begin{center}
		\includegraphics[width=0.7\textwidth]{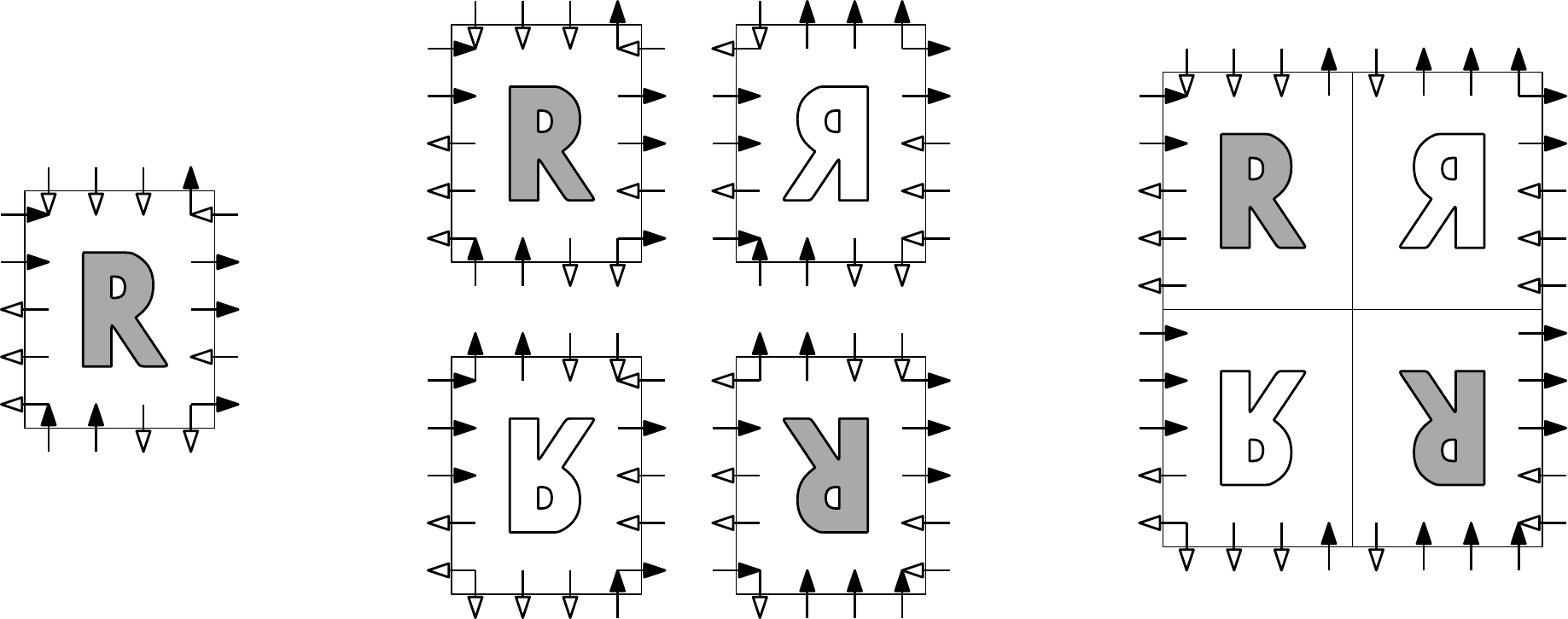}
		\caption{The passage from boundary conditions $\xi$ to balanced toroidal boundary conditions $\zeta$. 
		The letter ${R}$ inside the rectangles is only to indicate the performed transformations 
		(rotations, reflections and reversal of all arrows). }
		\label{fig:bcT}
		\end{center}
	\end{figure}
	Let $\xi_1$ be the boundary conditions on $\bfR_{N,M}$ 
	obtained from $\xi$ by horizontal reflection and arrow reversal, 
	$\xi_2$ be obtained from $\xi$ by vertical reflection and arrow reversal
	and $\xi_3$  be obtained from $\xi$ by rotation by $\pi$. 
	Let $\zeta$ be the toroidal boundary conditions on $\bfR_{2N,2M}$ composed as follows: 
	\begin{itemize*}
	\item the top half of the left side agrees with $\xi$,
	\item the bottom half of the left side agrees with $\xi_2$,
	\item the left half of the top side agrees with $\xi$ and
	\item the right half of the top side agrees with $\xi_1$.
	\end{itemize*}
	Note that $\zeta$ is balanced so that $$Z_{2N,2M}\ge   Z_{2N,2M}^{\rm (bal)} \ge Z^{\zeta}_{2N,2M}.$$

	Upon inspection of Fig.~\ref{fig:bcT}, one may easily deduce that for any boundary conditions $\xi$ on $\bfR_{N,M}$,
	$$ 
		 Z_{2N,2M}^{\rm (bal)}\ge Z^{\zeta}_{2N,2M} 
		\geq Z^{\xi}_{N,M} Z^{\xi_1}_{N,M}Z^{\xi_2}_{N,M}Z^{\xi_3}_{N,M}
		= \big(Z^{\xi}_{N,M}\big)^4.
	$$
	By summing over the $2^{N+M}$ toroidal boundary conditions $\xi$, the result follows.
	\smallskip

	\noindent (ii) Write $n=aN+r$ and $m=bM+q$ with $0\le r<N$ and $0\le q<M$. Fix a balanced toroidal boundary conditions $\xi$ on $\bfR_{N,M}$. 
	The construction that follows is described in Fig.~\ref{fig:subadd}. 
	
	Let $\xi_1$ be the toroidal boundary conditions on $\bfR_{aN, bM}$ obtained
	by repeating $a$ times each horizontal side
	and $b$ times each vertical side of $\xi$ on the corresponding sides of $\xi_1$. 

	Let $\xi_2$ be the toroidal boundary conditions on $\bfR_{r, bM}$, equal to $\xi_1$ on the vertical sides, 
	with $r/2$ down arrows amassed to the left of the bottom (and top) side,
	completed by $r/2$ up arrows at the right of the bottom (and top) side. 
	
	Finally, define $\xi_3$ to be the toroidal boundary conditions on $\bfR_{n, q}$
	with only left-pointing arrows on the vertical sides, 
	equal to the top of $\xi_1$ for the left-most $aN$ arrows of both the top and bottom sides
	and equal to top of $\xi_2$ for the remaining $r$ right-most arrows of the top and bottom sides.
	
	Set $\zeta$ to be the boundary conditions obtained from the gluing of $\xi_1,\xi_2$ and $\xi_3$, 
	that is: 
	\begin{itemize*}
	\item the top and bottom sides of $\zeta$ are equal to those of $\xi_3$,
	\item the bottom $bM$ arrows of the left and right sides of $\zeta$ are equal to those of $\xi_1$,
	\item the top $q$ arrows of the left and right sides of $\zeta$ are pointing leftwards.
	\end{itemize*} 
We thus easily deduce that
	$$ 
		Z^\zeta_{n,m} 
		\geq Z^{\xi_1}_{aN, bM}  Z^{\xi_2}_{r, bM}  Z^{\xi_3}_{n, q}
		\geq \big(Z^{\xi}_{N,M}\big)^{ab}  Z^{\xi_2}_{r, bM}  Z^{\xi_3}_{n, q}.
	$$
	It remains to prove a lower bound on the last two terms. 
	Observe that there exists at least one configuration $\vec{\om}_3$ on $\bfR_{aN + r, q}$, agreeing with the boundary conditions $\xi_3$ 
	and having non-zero weight. 
	It is obtained by setting all horizontal edges pointing left and all rows of vertical edges being identical to the top of $\xi_3$. 
	This proves that 
	$$ Z^{\xi_3}_{aN + r, q} \geq w(\vec{\om}_3) \geq \min\{1,c\}^{(aN+r)q}.$$
	A slightly more involved construction is necessary to exhibit a configuration $\vec{\om}_2$ on $\bfR_{r, bM}$, consistent with $\xi_2$ 
	and with non-zero weight.
	We represent it in Fig.~\ref{fig:subadd} and leave it to the meticulous reader to check the details of its construction. 
	It follows that 
	$$ Z^{\xi_2}_{r, bM} \geq w(\vec{\om}_2) \geq \min\{1,c\}^{bMr}.$$
	We conclude that 
	$$ 
		Z^{\rm (bal)}_{n,m} \ge Z^\zeta_{n,m} 
		\geq \big(Z^{\xi}_{N,M}\big)^{ab} \min\{1,c\}^{bMr + aNq + rq}.
	$$
	By choosing $\xi$ maximizing $Z^\xi_{N,M}$, we deduce that 
	$$Z^{\rm (bal)}_{n,m}\ge  \big(Z^{\rm (bal)}_{N,M}\big)^{ab} \min\{1,c\}^{bMr + aNq + rq}\big(\tfrac12\big)^{ab(M+N)}.$$
	The result follows by taking the logarithm.
	\begin{figure}
		\begin{center}
		\includegraphics[width=0.6\textwidth]{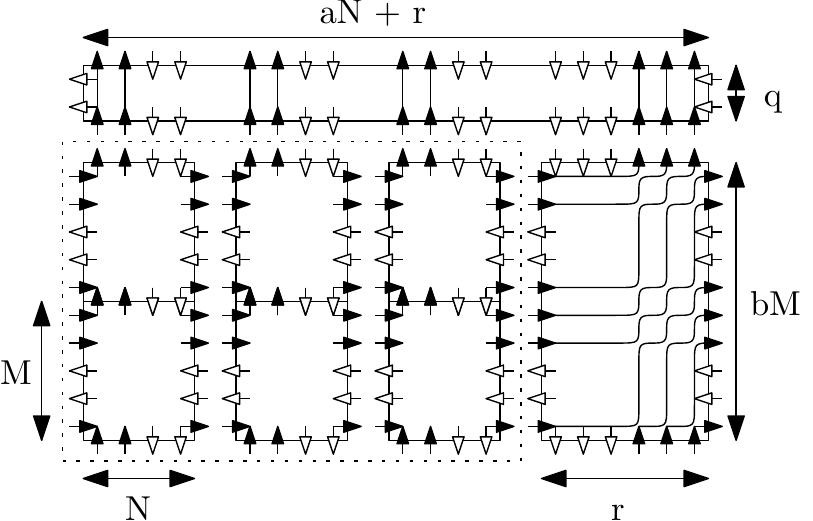}
		\caption{The block of size $\bfR_{aN, bM}$ with balanced toroidal boundary conditions $\xi_1$ is encircled; 
		on its right is the block $\bfR_{r, bM}$ with boundary conditions $\xi_2$
		and above is the block  $\bfR_{aN + r, q}$ with boundary conditions $\xi_3$.
		In the two latter blocks, examples of configurations with positive weight are given
		(only the up and right-pointing edges are drawn in the interior of the blocks).}
		\end{center}
		\label{fig:subadd}
	\end{figure}
\end{proof}
\medskip

\noindent{\bf Calculation of the free energy.}
Recall Proposition 2.1 of \cite{BetheAnsatz1} and more specifically equation (3.1), that expresses $Z_{N,M}$ as the trace of $V^M$. A straightforward adaptation shows that, for all $N$ multiple of $4$ and even, 
\begin{align*}
	Z_{N,M}^{\rm (bal)} = \Tr\Big[\big(V^{[N/2]}\big)^M\Big]=\la_0^M+\la_1^M+\dots,\end{align*}
where $\la_0,\la_1,\dots$ are the ${\binom{N}{N/2}}$ eigenvalues of the diagonalizable matrix $V^{[N/2]}$, listed  with multiplicity and indexed such that $|\la_0| \geq |\la_1|\geq \dots$. 
Since $V^{[N/2]}$ is a Perron-Frobenius matrix, $|\la_0| > |\la_1|$ and $\la_0 = \La_0(N)$ (the eigenvalue computed in Theorem~\ref{thm:6V}), so that 
\begin{align*}
	\lim_{M \to \infty} \frac{1}{M} \log \Tr\ \big(V^{[N/2]}\big)^M = \log \La_0(N).
\end{align*}
In light of~\eqref{eq:agr}, the limit defining $f(1,1,c)$ may be taken with $M \to \infty$ first, 
then $N\to \infty$ along multiples of $4$. Thus we find, 
\begin{align*}
	f(1,1,c) = \lim_{\substack{N \to \infty\\ N \in 4\bbN}} \frac{1}{N} \log \La_0(N) \stackrel{\eqref{eq:aggf}}=\frac{\lambda}{2} + \sum_{m=1}^\infty \frac{ e^{- m \lambda} \tanh(m\lambda)}{m} .
\end{align*}
\end{proof}

\begin{remark}
We have shown that the free energy for the torus is the same as that for ``free" boundary conditions. However, it is possible to construct boundary conditions on rectangles that lead to nonzero, but strictly smaller free energy. One prominent example of this is the Domain Wall boundary conditions, which have been studied extensively due to their relations to combinatorial objects, such as Young diagrams. Under these boundary conditions, the six-vertex model partition functions satisfy recursion relations that make it possible to exactly compute them for finite lattices (see \cite{Zinn00} for more detail). This technique gives a formula for the free energy of this model (see\cite{KorZinn00}), which is different from the one we showed above.
\end{remark}

\subsection{From the six-vertex to the random-cluster model: proof of Theorem~\ref{thm:RCM}}\label{sec:RCM}

\newcommand{\TFK}{\bbT^{\diamond}}
\newcommand{\TV}{\bbT
}

The proof is split into two main steps. 
First, we present a classical correspondence between the six-vertex and random-cluster models using a series of intermediate representations (this correspondence may be found in \cite{Bax89}). 
Then, certain estimates on the random-cluster model are provided,  
that are used to relate its correlation length to quantities obtained via the six-vertex model.

\subsubsection{Correspondence between the random-cluster and six-vertex models}

Fix two integers $M,N$, both even and $q > 4$. 
Notice that the torus $\TV_{N,M}$ is then a bipartite graph. 
Let $V_{\circ}(\TV_{N,M})$ and $V_{\bullet}(\TV_{N,M})$ be a partition 
of the vertices of the graph $\TV_{N,M} + (\frac12,\frac12)$ (that is, $\TV_{N,M}$ translated by $(\frac12,\frac12)$), each containing no adjacent vertices. 
Define the graphs $\TFK_{N,M}$ and $(\TFK_{N,M})^*$ as having vertex sets  $V_{\bullet}(\TV_{N,M})$ and $V_{\circ}(\TV_{N,M})$, respectively, 
and having an edge between vertices $u$ and $v$ if $u$ is a translation of $v$ by $(1,1)$ or $(-1,1)$ (see Fig.~\ref{fig:lattices2}). 
By construction, $(\TFK_{N,M})^*$ is the dual graph of $\TFK_{N,M}$. 
\begin{figure}[htb]
	\begin{center}
		\includegraphics[width=0.25\textwidth, page=2]{TFK.pdf} \qquad
		\includegraphics[width=0.25\textwidth, page=1]{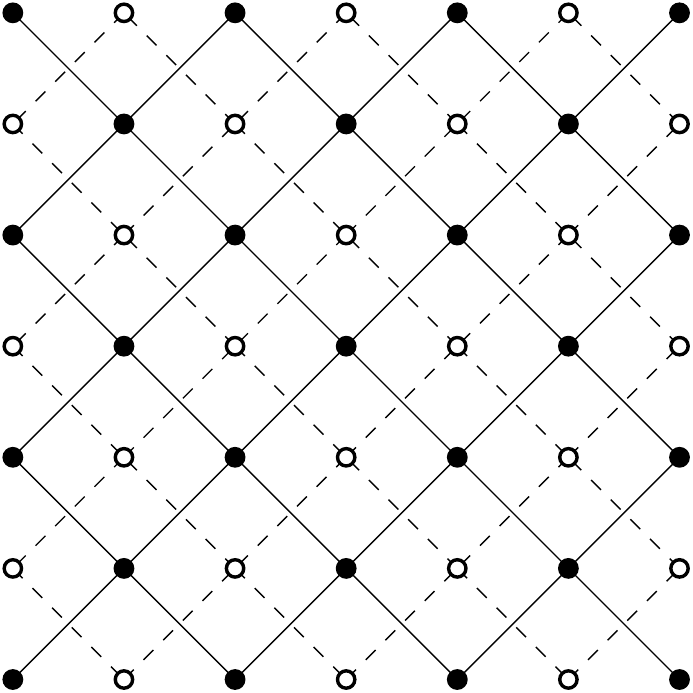}
	\caption{{\em Left:}  the lattice $\bbT_{N,M}$ used for the six-vertex model.
		{\em Right:} the corresponding lattice for the random-cluster model, $\bbT^{\diamond}_{N,M}$ (in solid lines), 
		and its dual (with dotted lines).}
		\label{fig:lattices2}
		\end{center}
\end{figure}

Let $\Om_{\rm RC}$ be the set of random-cluster configurations on $\TFK_{N,M}$
and $\Omega_{\rm 6V}$ be the set of six-vertex configurations on $\TV_{N,M}$.
We will exhibit a correspondence between $\Om_{\rm RC}$ and  $\Omega_{\rm 6V}$ that will allow us to 
relate the free energy and correlation length of the two models. 
The correspondence consists of several intermediate steps embodied by Lemmas~\ref{lem:correspondence1} --~\ref{lem:correspondence4};
the whole process is depicted in Fig.~\ref{fig:correspondence}.
The ultimate goal of this part is Corollary~\ref{cor:ZFK_vs_Z6v}, which will be the only result used in the proof of Theorem~\ref{thm:RCM}.

In linking the random-cluster and six-vertex models, we will use another type of configurations, called {\em loop configurations}. 
An oriented loop on $\TV_{N,M}$ is a cycle on $\TV_{N,M}$ which is edge-disjoint and non-self-intersecting. 
We may view oriented loops as ordered collections of edges of $E(\TV_{N,M})$, quotiented by cyclic permutations of the indices.  
Un-oriented loops (or simply loops) are oriented loops considered up to reversal of the indices.
A (oriented) loop configuration on $\TV_{N,M}$ is a partition of $E(\TV_{N,M})$ into (oriented) loops.

To each $\om \in \Om_{\rm RC}$ we associate a loop configuration $\om^{(\ell)}$ as in Fig.~\ref{fig:the_six_vertices}. In order to do so, we first construct the dual configuration $\om^*$ on $(\TFK_{N,M})^*$ by setting $\om^*(e^*)=1-\om(e)$, where $e^*$ is the edge of $(\TFK_{N,M})^*$ intersecting the edge $e$ of $\TFK_{N,M}$ in its middle (in words, a dual edge is in $\om^*$ if the corresponding edge of $\TFK_{N,M}$ is not in $\om$, and vice versa).
Then, consider the loop configuration $\om^{(\ell)}$ on $\TV_{N,M}$ created by loops that do not cross the edges of $\om$ or $\om^*$.
It is easy to see that $\om \mapsto \om^{(\ell)}$ is a bijection between $\Om_{\rm RC}$ and the set of all loop configurations.  

Call $\ell(\om)$ the number of different loops of $\om^{(\ell)}$,
and $\ell_0(\om)$ the number of such loops that are not retractable (on the torus) to a point.
Call $\ell_c(\om) := \ell(\om) - \ell_0(\om)$, the number of retractable loops.
We say that $\om$ has a \emph{net} if it has a cluster that winds around $\TFK_{N,M}$ in both directions. 
Set
$$
	s(\om)  =
	\begin{cases}
		0 & \text{ if $\om$ has no net},\\
		1 & \text{ if $\om$ has a net}.
	\end{cases}
$$

Fix $q > 4$. For $\om \in \Om_{\rm RC}$, define the weight of $\om$ in the critical random-cluster model as 
$$w_{\rm RC}(\om) = p_c^{o(\om)}(1 - p_c)^{c(\om)}q^{k(\om)},$$
where we recall that $p_c = \frac{\sqrt{q}}{1+\sqrt{q}}$ (see \cite{BefDum12}).

\begin{figure}[htb]
	\begin{center}
		\includegraphics[width=0.23\textwidth, page=1]{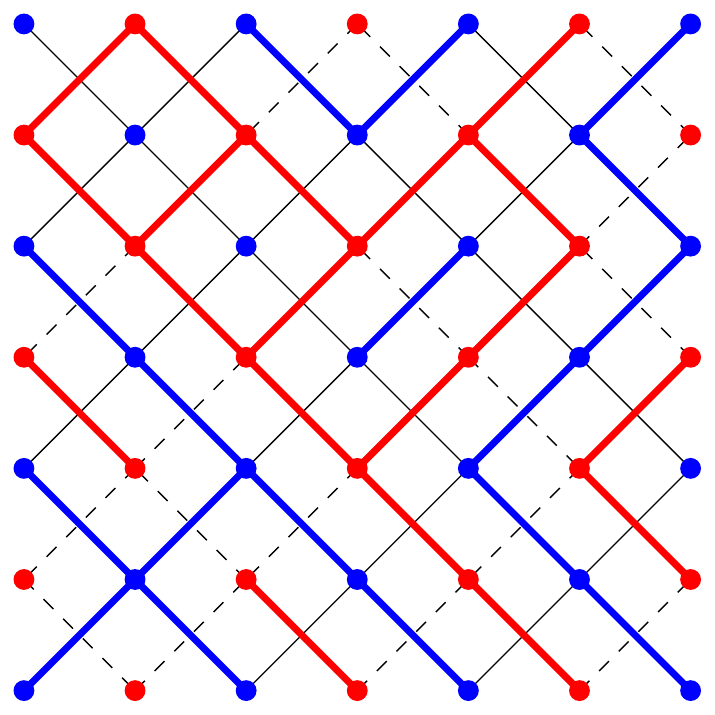}\,
		\includegraphics[width=0.23\textwidth, page=2]{maps.pdf}\,
		\includegraphics[width=0.23\textwidth, page=3]{maps.pdf}\,
		\includegraphics[width=0.23\textwidth, page=4]{maps.pdf}
	\end{center}
	\caption{The different steps in the correspondence between the random-cluster and six-vertex models on a torus. 
	From left to right: A random-cluster configuration and its dual, the corresponding loop configuration,
	an orientation of the loop configuration, the resulting six-vertex configuration. 
	Note that in the first picture, there exist both a primal and dual cluster winding vertically around the torus; 
	this leads to two loops that wind vertically (see second picture); 
	if these loops are oriented in the same direction (as in the third picture), 
	then the number of up arrows on every row of the six-vertex configuration is equal to $\frac{N}2 \pm 1$. 
	}
	\label{fig:correspondence}
\end{figure}

\begin{lemma}\label{lem:correspondence1}
	For all $\om\in \Om_{\rm RC}$, 
	$$
		w_{\rm RC}(\om) = C \sqrt{q}^{ \ell(\om)  + 2s(\om)}, 
	$$
	where $C = q^{\frac{MN}4}(1+\sqrt q)^{-MN}$ is a constant not depending on $\om$. 
\end{lemma}

\begin{proof}Set $V_\bullet=V_\bullet(\bbT_{N,M})$ and $E_\bullet$ to be the set of edges of $\bbT_{N,M}$.	Fix $\om \in \Om_{\rm RC}$. 
	Observe that, due to the Euler formula, 
	\begin{align*}
		2 k(\om) = \ell(\om) - o(\om) + 2s(\om) + |V_{\bullet}|.
	\end{align*}
	This relation offers us an alternative way of writing the random-cluster weight of a configuration:
	$$
		w_{\rm RC}(\om)
		= (1-p_c)^{|E_\bullet|}\left(\frac{p_c}{1-p_c}\right)^{o(\om)}\sqrt{q}^{ \ell(\om) - o(\om) + 2s(\om) + |V_{\bullet}|}.
	$$
	Since $p_c = \frac{\sqrt{q}}{1+\sqrt{q}}$, the above becomes
	$$
		w_{\rm RC}(\om)
		=\Big(\frac{1}{1+\sqrt q}\Big)^{|E_\bullet|}\sqrt{q}^{|V_{\bullet}|} \sqrt{q}^{ \ell(\om)  + 2s(\om)}
		= C \sqrt{q}^{ \ell(\om)  + 2s(\om)},
	$$
	where we have used that $|E_\bullet| = MN$ and $|V_{\bullet}| = MN/2$.
\end{proof}

\newcommand{\ol}{{\includegraphics[scale=0.2]{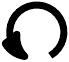}}}
	
Write $\om^{\ol}$ for oriented loop configurations, 
$\ell_0(\om^\ol)$ for the number of non-retractable loops of $\om^\ol$ and 
$\ell_-(\om^\ol)$ and $\ell_+(\om^\ol)$ for the number of retractable loops of $\om^\ol$ 
which are oriented clockwise and counterclockwise, respectively.
We introduce $\lambda>0$ defined by \begin{align}\label{eq:v}
	e^\lambda + e^{-\lambda} = \sqrt{q}.
\end{align} 

For an oriented loop configuration $\om^\ol$, write 
$$ w_{\ol} (\om^{\ol}) = e^{\lambda\ell_+(\om^\ol)} \, e^{-\lambda\ell_-(\om^\ol)}.$$ 

\begin{lemma}\label{lem:correspondence2}
	For any $\om \in \Om_{\rm RC}$, 
	\begin{align*}
		w_{\rm RC}(\om) = C\bigg(\frac{\sqrt{q}}2 \bigg)^{ \ell_0(\om)}q^{s(\om)} \sum_{\om^\ol} w_{\ol} (\om^{\ol}), 
	\end{align*}
	where the sum is over the $2^{\ell(\om)}$ oriented loop configurations $\om^\ol$ obtained by orienting each loop of $\om^{(\ell)}$ in one of two possible ways. 
\end{lemma}

\begin{proof}
	Fix $\om \in \Om_{\rm RC}$ and consider its associated loop configuration $\om^{(\ell)}$. 
	In summing the $2^{\ell(\om)}$ oriented loop configurations $\om^\ol$ associated with $\om^{(\ell)}$, 
	each loop appears with both orientations. 
	Thus, 
	\begin{align*}
		\sum_{\om^\ol} w_{\ol} (\om^{\ol})  
		= \Big(1+1\Big)^{\ell_0(\om)}\big(e^\lambda + e^{-\lambda}\big)^{\ell_c(\om)}
		= 2^{\ell_0(\om)} \sqrt{q}^{ \ell_c(\om)}
		= \frac{1}{C}\bigg(\frac2{\sqrt{q}}\bigg)^{ \ell_0(\om)}q^{-s(\om)} w_{\rm RC}(\om).
	\end{align*}
\end{proof}

Notice now that an oriented loop configuration gives rise to $8$ different configurations at each vertex. 
These are depicted in Fig.~\ref{fig:oriented_loop_vertices}. 
For an oriented loop configuration $\om^\ol$, write $n_{i}(\om^\ol)$ for the number of vertices of type $i$ in $\om^\ol$, 
with $i = 1,\, 2,\, 3,\, 4,\,  5A,\, 5B,\, 6A,\, 6B$.

\begin{figure}[htb]
	\begin{center}
		\includegraphics[width=0.8\textwidth]{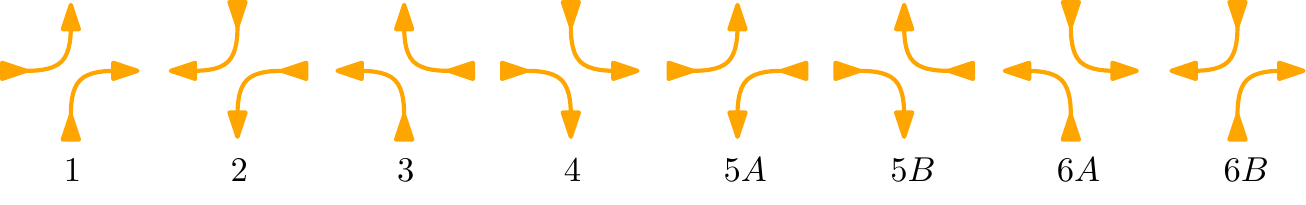}
	\end{center}
	\caption{The $8$ different types of vertices encountered in an oriented loop configuration.}
	\label{fig:oriented_loop_vertices}
\end{figure}

\begin{lemma}\label{lem:correspondence3}
	For any oriented loop configuration $\om^\ol$, 
	\begin{align*}
		w_{\ol}(\om^\ol) = e^{\frac\lambda2[n_{5A}(\om^\ol) + n_{6A}(\om^\ol)]}\ e^{-\frac\lambda2[n_{5B}(\om^\ol) + n_{6B}(\om^\ol)]}.
	\end{align*}
\end{lemma}

\begin{proof}
	Fix an  oriented loop configuration $\om^\ol$. 
	Notice that the retractable loops of $\om^\ol$ which are oriented clockwise have total winding $-2\pi$, 
	while those oriented counterclockwise have winding $2\pi$. 
	Loops which are not retractable have total winding $0$. 
	Write $W(\ell)$ for the winding of a loop $\ell \in \om^\ol$. Then
	\begin{align}\label{eq:winding}
		w_{\ol}(\om^\ol) =  \exp\Big(\frac{\la}{2\pi}\sum_{\ell \in \om^\ol}W(\ell)\Big),
	\end{align}
	where the sum is over all loops $\ell$ of $\om^\ol$.
	The winding of each loop may be computed by summing the winding of every turn along the loop. 
	The compounded winding of the two pieces of paths appearing in the different diagrams of Fig.~\ref{fig:oriented_loop_vertices} are 
	\begin{itemize*}
		\item vertices of type $1,\dots,4$: total winding $0$;
		\item vertices of type $5A$ and $6A$: total winding $\pi$;
		\item vertices of type $5B$ and $6B$: total winding $-\pi$.
	\end{itemize*} 
	The total winding of all loops may therefore be expressed as 
	\begin{align*}
		\sum_{\ell \in \om^\ol}W(\ell) =\pi\ \big[ n_{5A}(\om^\ol) + n_{6A}(\om^\ol) - n_{5B}(\om^\ol) - n_{6B}(\om^\ol)\big].
	\end{align*}
	The lemma follows from the above and \eqref{eq:winding}.
\end{proof}

For the final step of the correspondence, notice that each diagram in Fig.~\ref{fig:oriented_loop_vertices} 
corresponds to a six-vertex local configuration (as those depicted in Fig.~\ref{fig:the_six_vertices}). 
Indeed, configurations $5A$ and $5B$ correspond to configuration $5$ in Fig.~\ref{fig:the_six_vertices} 
and configurations $6A$ and $6B$ correspond to configuration $6$ in Fig.~\ref{fig:the_six_vertices}.
The first four configurations of Fig.~\ref{fig:oriented_loop_vertices} 
correspond to the first four in Fig.~\ref{fig:the_six_vertices}, respectively.

Thus, to each oriented loop configuration $\om^\ol$ is associated a six-vertex configuration $\vec{\om}$. 
Note that the map associating $\vec{\om}$ to $\om^\ol$ is not injective since
there are $2^{n_5(\vec\om) + n_6(\vec\om)}$ oriented loop configurations corresponding to each $\vec{\om}$.

Define the parameter $c$ of the six-vertex model by 
\begin{align}\label{eq:c}
	c = e^{\frac\la2} + e^{-\frac\la2} = \sqrt{2 + \sqrt{q}}.
\end{align}
(The latter equality is obtained from~\eqref{eq:v} by straightforward computation.) 
As in the rest of the paper, $a = b =1$ are fixed. Write $w_{6V}(\vec\om)$ instead of simply $w(\vec\om)$ for the weight of a six-vertex configuration $\vec\om$ as defined in~\eqref{eq:w_6V}.

\begin{lemma}\label{lem:correspondence4}
	For all six-vertex configurations $\vec{\om}$ (that is configurations obeying the ice rule), 
	$$ w_{6V}(\vec{\om}) = \sum_{\om^\ol} w_{\ol}(\om^\ol),$$
	where the sum is over all oriented loop configurations $\om^\ol$  corresponding to $\vec{\om}$.
\end{lemma}

\begin{proof}
	Fix a six-vertex configuration $\vec{\om}$. 
	Let $N_{5,6}(\vec{\om})$ be the set of vertices of type $5$ and $6$ in $\vec\om$.
	Then, due to the choice of $c$, 
	\begin{align*}
		w_{6V}(\vec{\om}) 
		=\prod_{u \in N_{5,6}(\vec\om)} \big(e^{\frac\la2} + e^{-\frac\la2}\big)
		=\sum_{\eps \in \{\pm 1\}^{N_{5,6}(\vec\om)}}\, \prod_{u \in N_{5,6}(\vec\om)} e^{\frac\la2\eps(u)}
		= \sum_{\om^\ol} w_{\ol}(\om^\ol).
	\end{align*}
	For the last equality above, 
	notice that each choice of $\eps \in \{\pm 1\}^{N_{5,6}(\vec\om)}$ corresponds to a choice of type $A$ or $B$ 
	for every vertex of $N_{5,6}(\vec\om)$, 
	and hence to one of the $2^{n_5(\vec\om) + n_6(\vec\om)}$ oriented loop configurations corresponding to $\vec\om$. 
\end{proof}

\newcommand{\Up}{U}

For a six-vertex configuration $\vec\om$ on $\TV_{N,M}$, write $|\vec\om|$ for the number of up arrows on each row 
(recall that this number is the same on all rows). 
The notation obviously extends to oriented loop configurations. 
Moreover, for $r \geq 0$, set 
\begin{align*}
	Z_{6V}^{(r)}(N,M) = \sum_{\vec\om:\, |\vec\om| = \frac{N}2-r} w_{6V}(\vec\om).
\end{align*}
For $\om \in \Om_{\rm RC}$, let $2 \Up(\om)$ be the total number of times loops of $\om^{(\ell)}$ wind vertically around $\TFK_{N,M}$
(due to periodicity, this number is necessarily even).

\begin{corollary}\label{cor:ZFK_vs_Z6v}
	Let $q > 4$ and set $c = \sqrt{2 + \sqrt q}$.  Fix $r\ge1$. For $N,M$ even, set $C = q^{\frac{MN}4}(1+\sqrt q)^{-MN}$.
	Then 
	\begin{align*}
		\sum_{\om \in \Om_{\rm RC}} w_{\rm RC}(\om) \bigg(\frac2{\sqrt{q}} \bigg)^{\ell_0(\om)}q^{ - s(\om)} &= C\ Z_{6V}(N,M);
		\label{eq:cori}\tag{i} \\
		\sum_{\om \in \Om_{\rm RC}:\, \Up(\om) = 1} w_{\rm RC}(\om) \bigg(\frac2{\sqrt{q}} \bigg)^{\ell_0(\om)}q^{ - s(\om)} 
		&\leq 4 C\ Z^{(1)}_{6V}(N,M);
		\label{eq:corii}\tag{ii}\\
		\sum_{\om \in \Om_{\rm RC}:\, \Up(\om) \geq r} w_{\rm RC}(\om) \bigg(\frac2{\sqrt{q}} \bigg)^{\ell_0(\om)}q^{ - s(\om)} 
		&\geq  C\ Z^{(r)}_{6V}(N,M).
		\label{eq:coriii}\tag{iii}
	\end{align*}
\end{corollary}
Note that Items~\eqref{eq:corii} and \eqref{eq:coriii} imply that for $r=1$, the left and right sides of \eqref{eq:corii} are of the same order. Item~\eqref{eq:coriii} may appear technical for $r>1$, but will be used later to bound the correlation length of the random-cluster model from below.

\begin{proof}
	Let us start by proving~\eqref{eq:cori}. Due to Lemmas~\ref{lem:correspondence2} and~\ref{lem:correspondence4}, we have
	\begin{align*}
		\sum_{\om \in \Om_{\rm RC}} w_{\rm RC}(\om) \bigg(\frac2{\sqrt{q}} \bigg)^{\ell_0(\om)}q^{ - s(\om)}
		= C \, \sum_{\om^\ol} w_{\ol}(\om^\ol)
		= C \, \sum_{\vec\om} w_{6V}(\vec\om) = Z_{6V}(N,M),
	\end{align*}
	where the sums in the second and third terms run over all oriented loop configurations and six-vertex configurations, respectively.
	\smallskip 
	
	Let us now prove~\eqref{eq:corii}. 
	We restrict ourselves to random-cluster configurations with $\Up(\om) = 1$. 
	For such configuration $\om$, $\om^{(\ell)}$ has two loops winding vertically around $\TV$.
	Moreover, for any oriented loop configuration $\om^\ol$ which is compatible with $\om^{(\ell)}$, 
	we may consider the oriented loop configuration $\tilde\om^\ol$, obtained from $\om^\ol$ 
	by orienting the two vertically-winding loops downwards. 
	Then, $w_{\ol}(\om^\ol) = w_{\ol}(\tilde\om^\ol)$ and there are four oriented loop configurations corresponding to any $\tilde\om^\ol$. 
	Thus,
	$$w_{\rm RC}(\om)  \bigg(\frac2{\sqrt{q}} \bigg)^{\ell_0(\om)}q^{ - s(\om)} = 4C \, \sum_{\om^\ol} w_{\ol}(\om^\ol),$$
	where the sum in the right-hand side is over oriented loop configurations corresponding to $\om$
	in which the two vertically-winding loops are oriented downwards. 
	Since all other loops do not wind vertically around $\TV$, 
	the total number of up arrows on any given row of such an oriented loop configuration is $N/2 -1$. 
	Thus
	\begin{align*}
		\sum_{\om \in \Om_{\rm RC}:\, \Up(\om) = 1} w_{\rm RC}(\om) \bigg(\frac2{\sqrt{q}} \bigg)^{\ell_0(\om)}q^{ - s(\om)}
		\leq 4C \sum_{\om^\ol:\, |\om^\ol| =N/2 - 1}w_{\ol} (\om^\ol) 
		= 4 C \ Z^{(1)}_{6V}(N,M).
	\end{align*}
	\smallskip 
	
	Finally we show~\eqref{eq:coriii}. 
	If $\om^\ol$ is an oriented loop configuration with $|\om^\ol| = N/2-r$, then, by the same up-arrow counting argument as above, 
	the corresponding random-cluster configuration $\om$ has $\Up(\om) \geq r$. Thus, 
	\begin{align*}
		C Z^{(r)}_{6V}(N,M) 
		= C \sum_{\om^\ol:\, |\om^\ol| = N/2 - r} w_{\ol} (\om^\ol)
		\leq \sum_{\om \in \Om_{\rm RC}:\, \Up(\om) \geq r} w_{\rm RC}(\om) \bigg(\frac2{\sqrt{q}} \bigg)^{\ell_0(\om)}q^{ - s(\om)}.
	\end{align*}
\end{proof}

\subsubsection{Random-cluster computations}

In this section, we relate the correlation length of the random-cluster model to the rates of growth of the quantities $Z^{(r)}_{6V}(N,M)$ defined in the previous section. We will need some notation. 

Let $a,b$ be two vertices and $C$ be a subset of vertices. Let $\{a \xleftrightarrow{C} b\}$ be the event that there exists a path of vertices in $C$, starting at $a$ and finishing at $b$ composed of edges in $\omega$ only. In this case, we say that $a$ is {\em connected} to $b$ in $C$. We also set $\{A\xleftrightarrow{C}B\}$ for the union on $a\in A$ and $b\in B$ of $\{a \xleftrightarrow{C} b\}$. When $C$ is the whole graph, we omit it from the notation.

Consider the sub-lattice $\mathbb L$ of $\bbZ^2$ made of vertices with sum of coordinates even, and edges between two vertices if one is the translate of the other by $(1,1)$ or $(1,-1)$ \footnote{This lattice is the local limit of the graphs $\TFK_{N,M}$ as $M$ and $N$ tend to infinity. It is a version of $\sqrt 2\bbZ^2$ rotated by an angle of $\pi/4$.}. This is not the same as in the introduction, but we believe that since this change is restricted to this section, it should not lead to any confusion. We will view $\TFK_{N,M}$ as having vertices $(i,j)$ with $i,j$ integers of even sum, taken modulo $N$ and $M$ respectively. Also, we write $[a,b]\times[c,d]$ for the subgraph of $\bbL$ composed of vertices $(i,j)$ with $a\le i\le b$ and $c\le j\le d$. Let $\phi^0_{\bbL,p_c,q}$ be the infinite-volume random-cluster measure on $\mathbb L$ with free boundary conditions.

Write  $\xi(q)$ for the correlation length of the critical random-cluster model on this rotated lattice defined by 
\begin{align}\label{eq:cor_len_def}
	\xi (q)^{-1} = \lim_{n\rightarrow\infty} -\tfrac1{2n}\log\phi_{\bbL,p_c,q}^0[0\longleftrightarrow(0,2n)].
\end{align}
By the definition of the lattice $\mathbb L$ on which $\phi^0_{\bbZ^2,p_c,q}$ is defined, the right-hand side corresponds to the left-hand side of \eqref{eq:aaf}. The limit may be shown to exist by sub-additivity arguments. 

The two following lemmas will be used to prove Theorem~\ref{thm:RCM}. Unlike the rest of the paper, both lemmas below are based on probabilistic estimates specific to the random-cluster model. 
We refer the reader to \cite{Gri06} for a manuscript on the subject, and \cite{Dum13} for an account of recent progress. We will apply repeatedly classical facts about the random-cluster model, and give each time the precise reference in \cite{Gri06}.

\begin{lemma}\label{lem:no_loop_density}
	For all $q \geq 1$,
	\begin{align}\label{eq:no_loop_density}
		\lim_{N\to \infty}\lim_{M \to \infty} 
		\frac{1}{M}\log \phi_{\TFK_{N,M},p_c,q} \left[ \big(\tfrac2{\sqrt{q}} \big)^{\ell_0(\om)}q^{ - s(\om)} \right]
		= 0.
	\end{align}
\end{lemma}

\begin{lemma}\label{lem:cor_length}
	For all $q \geq 1$ and $r\ge1$, we have that
	\begin{align}
		&\liminf_{N\to\infty}\liminf_{M \to \infty}\frac{1}{M}\log\phi_{\TFK_{N,M},p_c,q} ( \Up(\om) = 1) \geq -\xi(q)^{-1}, \label{eq:cor_length1}\\
		&\limsup_{N\to\infty}\limsup_{M \to \infty}\frac{1}{M}\log\phi_{\TFK_{N,M},p_c,q} ( \Up(\om) \geq r) \leq -(r-1) \xi(q)^{-1}.
		\label{eq:cor_length2}
	\end{align}
\end{lemma}
\begin{remark}
    Inequality \eqref{eq:cor_length1} should actually be an equality. 
    Unfortunately, we did not manage to derive the reverse inequality using the random-cluster model only. 
    In order to circumvent this fact, in the proof of Theorem~\ref{thm:RCM} we will rely on \eqref{eq:cor_length2} (see Remark~\ref{rmk:rr}).
%
\end{remark}


In both proofs below, $q \geq 1$ and $p = p_c(q)$ are fixed, and we drop them from the notation of the random-cluster measure. 

\begin{proof}[Lemma~\ref{lem:no_loop_density}]
Fix $q \geq 1$. Since $q^{ - s(\om)}\ge q^{-1}$, it is sufficient to prove that 	
\begin{align*}
\lim_{N\rightarrow\infty}\lim_{M\rightarrow \infty}\tfrac1M\log \phi_{\TFK_{N,M}} \Big[\big(\tfrac2{\sqrt{q}} \big)^{\ell_0(\om)} \Big] =0.
\end{align*}
Fix $\delta>0$. To start, we will bound $\phi_{\TFK_{N,M}}(\ell_0(\om)\ge \de M)$. 

By closing all the edges intersecting $\bbR\times\{-\tfrac12\}$ and $\{-\tfrac12\}\times\bbR$, we transform the random-cluster model on $\TFK_{N,M}$ into the random-cluster model with free boundary conditions on the rectangle ${\rm R}_{N,M}^\diamond=[0,N-1]\times [0,M-1]$. The finite-energy property \cite[Eq.~(3.4)]{Gri06} implies the existence of a constant ${\bf c}>0$ independent of $N,M$ and $\de$ such that
\begin{equation}\label{eq:kkl}
\phi_{\TFK_{N,M}}(\ell_0(\om)\ge \de M)\le 
{\bf c}^{M+N}\phi_{{\rm R}_{N,M}^\diamond}^0(\exists n\text{ disjoint clusters crossing ${\rm R}_{N,M}$ horizontally}),
\end{equation}
where $n=\de M-N$. 
The appearance of $-N$ in the definition of $n$ is due to the fact that at most $N$ of the $\ell_0(\om)$ non-retractable loops intersect the horizontal line $\bbR\times\{-\frac12\}$.

For $x_1,\dots,x_r$ on the left side $\partial_{L}$ of ${\rm R}_{N,M}^\diamond$, let $H(x_1,\dots,x_r)$ be the event that $x_j$ is connected to the right side $\partial_{R}$ of ${\rm R}_{N,M}^\diamond$ for $j=1,\dots, r$ 
	and that the clusters of $x_1,\dots, x_r$ are all distinct.
If $\om$ is a configuration contributing to the right-hand side of~\eqref{eq:kkl}, 
then there exist $n$ points $x_1,\dots,x_n$ on $\partial_{L}$ such that $H(x_1,\dots,x_n)$ occurs. 

	Write $\mathsf C_{x_j}$ for the cluster of the point $x_j$. Then, for any $j \geq 1$ and any subset $C$ of vertices of ${\rm R}_{N,M}^\diamond$, we have that
	\begin{align*}
		\phi_{{\rm R}_{N,M}^\diamond}^{0} [H(x_1,\dots, x_{j+1})\big| H(x_1,\dots, x_j)\,,\,\bigcup_{i\le j}\mathsf C_{x_i}=C]
		&=\phi^{0}_{{\rm R}_{N,M}^\diamond\setminus C}(x_{j+1}\longleftrightarrow\partial_R) \\
		&\leq \phi^{0}_{\bbL}(0\longleftrightarrow \partial\Lambda_N),
	\end{align*}
where in the first equality, we used the domain Markov property\footnote{This argument is classical and involves the fact that the cluster of a point is measurable in terms of edges with one or two endpoints in that cluster (see Fig.~\ref{fig:explorations} for an illustration of this argument).} \cite[Lem.~4.13]{Gri06}  and in the second, the comparison between boundary conditions \cite[Lem.~4.14]{Gri06} and the invariance under translations of $\phi^0_{\bbL}$ \cite[Thm.~4.19]{Gri06}. By summing over possible values of $C$, we deduce that
\begin{align*}
		\phi_{{\rm R}_{N,M}^\diamond}^{0} [H(x_1,\dots, x_{j+1}) \big| H(x_1,\dots, x_j)]&\leq \phi^{0}_{\bbL}(0\longleftrightarrow \partial\Lambda_N).
	\end{align*}
Induction on $j<n$ implies that
	$$\phi_{{\rm R}_{N,M}^\diamond}^{0} [H(x_1,\dots, x_n)]\le \phi^{0}_{\bbL}(0\longleftrightarrow \partial\Lambda_N)^n.$$
After taking the union over all possible $x_1,\dots,x_n$ on $\partial_L$, we deduce from~\eqref{eq:kkl} that
\begin{align*}
	\phi_{\TFK_{N,M}}(\ell_0(\om)\ge \de M) & \le {\bf c}^{M+N}\times\binom{M}{n}\times \phi^0_{\bbL}(0\longleftrightarrow\partial\Lambda_N)^n \\ & \leq \big(2 {\bf c}^2 \left[\phi^0_{\bbL}(0\longleftrightarrow\partial\Lambda_N)\right]^{\delta/2}\big)^M,
\end{align*}
where we bound ${M \choose n}$ by $2^M$, and increase $M$ until $n > \delta M/2$.
Now, it is classical \cite[Thm.~6.17]{Gri06} that $\phi^0_{\bbL}(0\longleftrightarrow\partial\Lambda_N)$ tends to 0 as $N$ tends to infinity so that for $N$ large enough, 
\begin{equation*}
\phi_{\TFK_{N,M}}(\ell_0(\om)\ge \de M) \leq \left(\frac{1}{2}\right)^M.
\end{equation*}
This implies that for any $\de>0$, provided that $N$ is large enough,
\begin{equation}\label{eq:kkll}
    \limsup_{M\rightarrow\infty}\tfrac1M\Big|\log \phi_{\TFK_{N,M}}\Big[\big(\tfrac2{\sqrt{q}} \big)^{\ell_0(\om)} \Big]\Big|
    \le  \limsup_{M\rightarrow\infty}\tfrac1M\Big|\log \Big[\big(\tfrac2{\sqrt{q}} \big)^{\delta M} + \Big( \tfrac1{\sqrt{q}} \Big)^{M} \Big]\Big| \\ \leq \Big|\log \left(\tfrac{\sqrt{q}}{2}\right)\Big|\delta 
\end{equation}
which concludes the proof by letting $\de$ tend to 0.
\end{proof}

\begin{figure}
    \begin{center}
    \includegraphics[width = 0.3\textwidth]{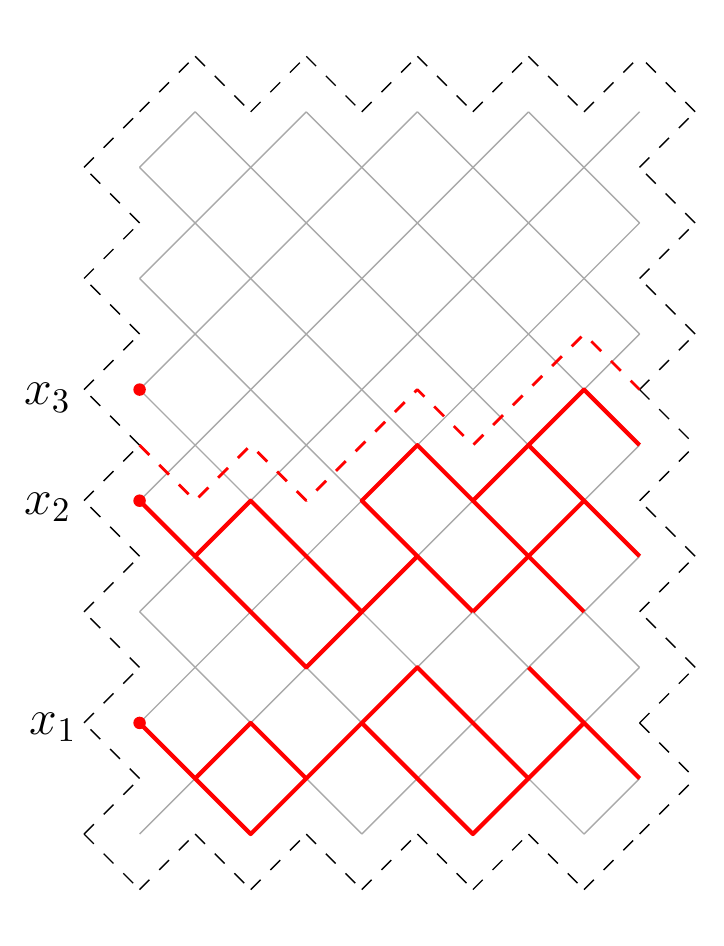}
    \includegraphics[width = 0.3\textwidth]{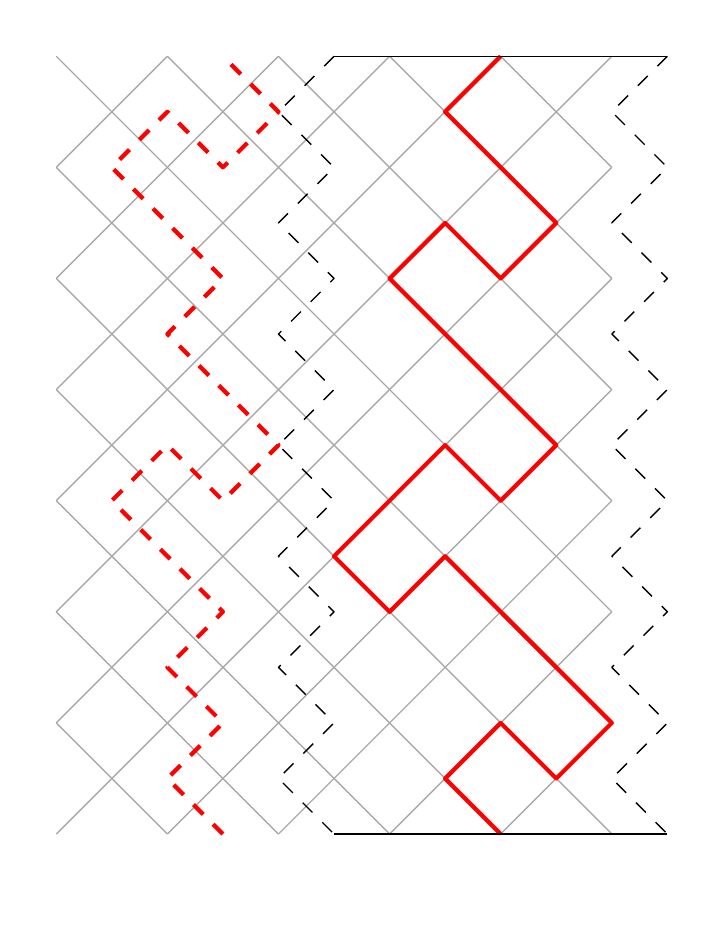}
    \includegraphics[width = 0.3\textwidth]{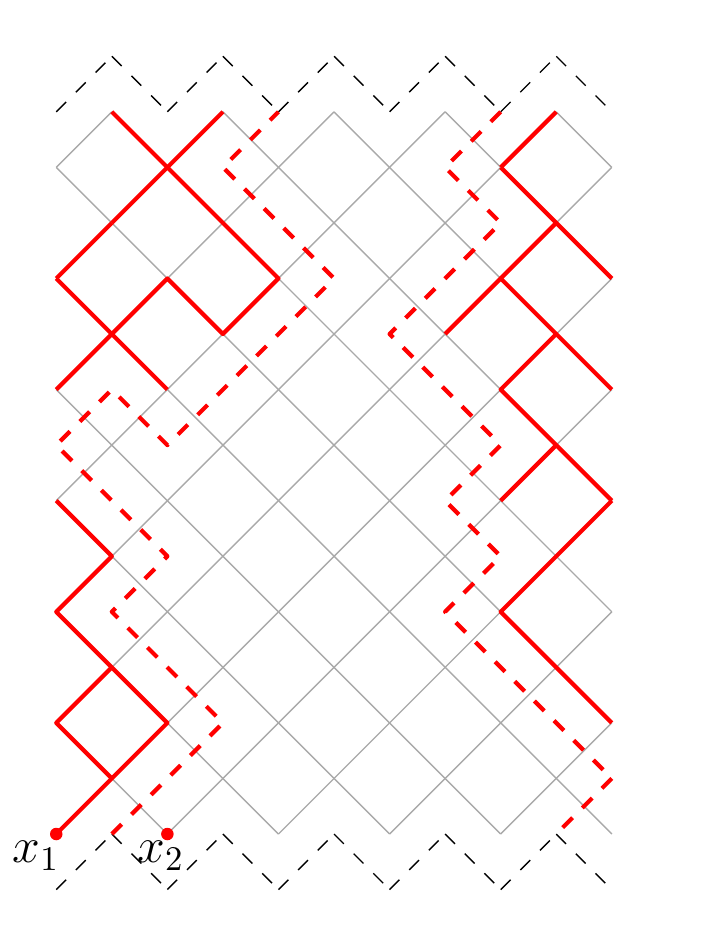}
    \caption{
    \emph{Left:} Exploring one by one the disjoint, horizontally crossing clusters contributing to~\eqref{eq:kkl}. Each new cluster (for instance the one of $x_3$) is surrounded by free boundary conditions. 
    \emph{Middle:} To create $\om$ with $U(\om)  = 1$, it is sufficient to ensure that $B$ occurs (dotted red line), and that, conditionally on $B$, $C$ also occurs.  The latter is more likely than the occurrence of a top-bottom crossing in the black rectangle with free boundary conditions on the lateral sides. 
    \emph{Right:} In exploring $H(x_1,\dots, x_r)$, every cluster crossing vertically the torus (except the first) is surrounded by free boundary conditions. }
    \label{fig:explorations}
    \end{center}
\end{figure}

Before starting the proof of Lemma~\ref{lem:cor_length}, we wish to highlight the fact that the random-cluster model enjoys a self-duality relation for planar graphs when $p=p_c$ \cite[Sec.~6.1]{Gri06}. On the torus, this self-duality can be restated as follows. Consider the measure 
$$
	\widetilde \phi_{\TFK_{N,M}}(\omega)
	=\frac{ p_c^{o(\om)}(1 - p_c)^{c(\om)}q^{k(\om)}q^{-s(\om)}}{\widetilde Z(N,M)},
$$
where $\widetilde Z(N,M)$ is the appropriate partition function. If $\omega$ is sampled according to $\widetilde \phi_{\TFK_{N,M}}(\omega)$, then $\omega^*$ is sampled according to the measure on $(\TFK_{N,M})^*$ obtained by translating $\widetilde \phi_{\TFK_{N,M}}(\omega)$ by $(1,0)$ (this claim follows directly from Lemma~\ref{lem:correspondence1}). 

Also note that $(\TFK_{N,M})^*$ can be obtained from $\TFK_{N,M}$ from reflections from either vertical or horizontal lines. We will use this observation several times in the next proof to transfer the probability of events defined in terms of $\omega$ to similar claims for $\omega^*$ (and vice versa).

\begin{proof}[\eqref{eq:cor_length1} of Lemma~\ref{lem:cor_length}]	
For a rectangle $R$, let $V_R$ (resp.~$H_R$) be the event that there exists a path in $\omega$ included in $R$ from the bottom to the top of $R$ (resp.~from the left to the right). We begin by proving\footnote{This claim was proved in the special case $N=M$ in \cite{BefDum12}. Here, some additional care must be taken since the torus has different vertical and horizontal size.} that there exists a constant $c>0$ such that for any $n,N,M$ with $3n\le \min\{N,M\}$, 
\begin{equation}\label{eq:crossing}
	\phi_{\TFK_{N,M}}(V_{[0,3n]\times[0,n]})\ge c.
\end{equation}
Indeed, if this is not the case, then the probability that some rectangle $[0,3n]\times[0,n]$ contains a path in $\omega^*$ from the left to the right  is larger than $1-c$. Therefore, the self-duality and the symmetry between $\TFK_{N,M}$ and its dual (mentioned above) imply that 
\begin{equation*}
	\phi_{\TFK_{N,M}}(H_{[0,3n]\times[0,n]})\ge \tfrac{1-c}{q^2}.
\end{equation*}
The FKG inequality \cite[Thm.~3.8]{Gri06} implies that 
$$\phi_{\TFK_{N,M}}\big(H_{[0,3n]\times[0,n]}\cap H_{[0,3n]\times[2n,3n]}\big)\ge \big(\tfrac{1-c}{q^2}\big)^2.$$
Now consider the bottom-most (resp.~top-most) path $\Gamma$ (resp.~$\Gamma'$) in $\omega$ crossing $[0,3n]\times[0,n]$ (resp.~$[0,3n]\times[2n,3n]$) from left to right. Fix two possible realizations $\gamma$ and $\gamma'$ of $\Gamma$ and $\Gamma'$. Conditioned on $\Gamma=\gamma$ and $\Gamma'=\gamma'$, the law of edges in $[0,3n]^2$ between $\gamma$ and $\gamma'$ is stochastically dominating the random-cluster measure with wired boundary conditions on the bottom and top of $[0,3n]^2$, and free on the left and right. Therefore, one may use self-duality in the square $[0,3n]^2$ to show that the probability that there is an open path connecting $\gamma$ to $\gamma'$ is larger or equal to $1/(1+q^2)$. This reasoning is classical, we refer for instance to \cite{BefDum12}. In particular, this path crosses $[0,3n]\times[2n,3n]$ from bottom to top. Overall, summing over all possible $\gamma$ and $\gamma'$ gives
\begin{align*}
    \phi_{\TFK_{N,M}}(V_{[0,3n]\times[2n,3n]})
    &\ge \tfrac1{1+q^2}\times \phi_{\TFK_{N,M}}\big(H_{[0,3n]\times[0,n]}\cap H_{[0,3n]\times[2n,3n]}\big)\\
    &\ge \tfrac1{1+q^2}\times\big(\tfrac{1-c}{q^2}\big)^2.
\end{align*}
Provided that $c=c(q)>0$ is chosen sufficiently small, this claim contradicts the assumption that \eqref{eq:crossing} was wrong. 
In conclusion, we  proved \eqref{eq:crossing} and we can proceed with the proof of~\eqref{eq:cor_length1}.
\medbreak

Fix $8n \leq \min\{M,N\}$.
As a consequence of \eqref{eq:crossing}, there exists  $x\in[0,3n]\times\{0\}$ and $y\in[0,3n]\times\{n\}$ such that 
$$\phi_{\TFK_{N,M}}\big(x \xleftrightarrow{[0,3n]\times[0,n]} y\big)\ge \frac{c}{9n^2}.$$
The FKG inequality \cite[Thm.~3.8]{Gri06} and the symmetry under reflections give that
\begin{align*}
\phi_{\TFK_{N,M}}\big(x \xleftrightarrow{[0,3n]\times[0,2n]} x+(0,2n)\big)&\ge \phi_{\TFK_{N,M}}\big(x \xleftrightarrow{[0,3n]\times[0,n]} y\big)\times\phi_{\TFK_{N,M}}\big(y \xleftrightarrow{[0,3n]\times[n,2n]} x+(0,2n)\big)\\
&\ge\Big(\frac{c}{9n^2}\Big)^2.
\end{align*}	
Write $M=2nk+r$ with $k\in\bbN$ and $0\le r<2n$. We can use the FKG inequality $k$ times to deduce that
\begin{align*}
	\phi_{\TFK_{N,M}}\big(x \xleftrightarrow{[0,3n]\times[0,2nk]} x+(0,2nk)\big)&\ge\Big(\frac{c}{9n^2}\Big)^{M/n}.
\end{align*}
Let $A$ be the event that $\omega$ contains a loop winding vertically around $\TFK_{N,M}$ and staying in $[0,3n]\times[0,M]$ (seen as a subgraph of $\TFK_{N,M}$), and that every edge of $\TFK_{N,M}$ intersecting $\bbR\times\{-\tfrac12\}$ but one is closed in $\omega$.

Since this event can be obtained from $\{x \xleftrightarrow{[0,3n]\times[0,2nk]} x+(0,2nk)\}$ by opening (in $\omega$) a self-avoiding path of length $2n-r$ zigzaging vertically between $x+(0,2nk)$ and $x$, and then closing all the remaining edges intersecting $\bbR\times\{-\tfrac12\}$, the finite-energy property \cite[Eq.~(3.4)]{Gri06} implies that 
$$\phi_{\TFK_{N,M}}(A)\ge {\bf c}^{2n+N}\times\Big(\frac{c}{9n^2}\Big)^{M/n},$$
for some constant ${\bf c} > 0$ only depending of $q$. 
Let $B$ be the event that $\omega$ does not contain any path from left to right in $[0,3n]\times[0,M]$, and that every edge of $\TFK_{N,M}$ intersecting $\bbR\times\{-\tfrac12\}$ but one is open in $\omega$. Using the self-duality and the symmetry between $\TFK_{N,M}$ and its dual, we deduce that
\begin{align}\label{eq:Hstar}
	\phi_{\TFK_{N,M}}(B)& \ge \tfrac1{q^2}\times{\bf c}^{2n+N}\times\Big(\frac{c}{9n^2}\Big)^{M/n}.
\end{align}

We are near the end: the event $B$ induces the existence of a path in $\omega^*$ winding vertically around the torus and contained in its left half. As \eqref{eq:Hstar} indicates, this comes at a (relatively) low cost. Next we also construct a vertically winding path contained in $\omega$, which will induce a vertically winding loop.

For each $j\in\bbN$, define $y_j:=(3N/4,2nj)$ and let $C$ be the event that $y_j$ is connected to $y_{j+1}$ (in $\omega$) for every $0\le j\le M/(2n)$\footnote{We define $y_j$ for every $j\in\bbN$, but we see $y_j$ as an element of $\TFK_{N,M}$, hence we think of $2nj$ as being taken modulo $M$.}.
Notice that the event $U(\omega)=1$ occurs if $B$ and $C$ occur together.  Therefore, 
$$\phi_{\TFK_{N,M}} ( \Up(\om) = 1) \geq \phi_{\TFK_{N,M}} (B\cap C)\ge \phi_{\TFK_{N,M}} (B)\times\phi_{\TFK_{N,M}} (C|B).$$
We now wish to bound the term $\phi_{\TFK_{N,M}} (C|B)$.
The comparison between boundary conditions \cite[Lem.~4.14]{Gri06} implies that the measure on $[N/2,N]\times[0,M]$ induced by $\phi_{\TFK_{N,M}}(\cdot|B)$ dominates the random-cluster measure $\phi^{\rm mix}_{[N/2,N]\times[0,M]}$ on $[N/2,N]\times[0,M]$ with free boundary conditions on the left and right sides, and wired on the top and bottom sides. Using the FKG inequality and the comparison between boundary conditions one more time, we find that
\begin{align*}
	\phi_{\TFK_{N,M}}(C|B)
	&\geq \prod_{j=0}^{\lfloor M/(2n)\rfloor} \phi^{\rm mix}_{[N/2,N]\times[0,M]}(y_j\longleftrightarrow y_{j+1})\\
	&\ge \phi_{\Lambda_{N/4}}^0(0\longleftrightarrow (0,2n))^{1+M/(2n)}.
\end{align*}
Overall, we deduce that 
$$
    \phi_{\TFK_{N,M}} ( \Up(\om) = 1) 
    \geq \tfrac1{q^2}\times{\bf c}^{2n+N}\times\Big(\frac{c}{9n^2}\Big)^{M/n}
    \times \phi_{\Lambda_{N/4}}^0(0\longleftrightarrow (0,2n))^{1+M/(2n)}.
$$
This in turn implies that 
$$
	\liminf_{M\rightarrow \infty}\frac1M\log\phi_{\TFK_{N,M}} ( \Up(\om) = 1) 
	\geq\tfrac1{n}\log(\tfrac{c}{9n^2})+\tfrac1{2n}\log\phi_{\Lambda_{N/4}}^0(0\longleftrightarrow (0,2n)).
$$
As $N$ tends to infinity (while $n$ is fixed), $\phi_{\Lambda_{N/4}}^0$ converges to $\phi_{\bbL}^0$ \cite[Thm.~4.19]{Gri06}. Thus
\begin{align*}
	\liminf_{N\rightarrow\infty}\liminf_{M\rightarrow \infty}\frac1M\log \phi_{\TFK_{N,M}} ( \Up(\om) = 1)  
	\geq\tfrac1{n}\log(\tfrac{c}{9n^2})+\tfrac1{2n}\log\phi_{\bbL}^0(0\longleftrightarrow (0,2n)).	
\end{align*}
Letting $n$ tend to infinity yields~\eqref{eq:cor_length1}.	
\end{proof}
	
\begin{proof}[\eqref{eq:cor_length2} of Lemma~\ref{lem:cor_length}]
	Fix $r \geq 1$ and consider $M,N \geq 2r$ even integers.
	Denote by $x_i = (2i,0)$ (for $i =1,\dots, N/2$) the points on the lower side of the torus $\TFK_{N,M}$
	and set $y_j:=x_j + (1,M-1)$.
	
	Let $\phi_{\bbH_{N,M}}^{0}$ be the measure on $\TFK_{N,M}$ conditioned on all edges intersecting 
	$\bbR\times\{-\tfrac12\}$ being closed; it may be viewed as a random-cluster measure on a cylinder $\bbH_{N,M}$ of height $M$ with free boundary conditions on the top and bottom. 
	
	Let $V(x_1,\dots,x_r)$ be the event that $x_j\longleftrightarrow y_j$ for $j=1,\dots, r$ 
	and that the clusters of $x_1,\dots, x_r$ are all distinct.
The finite-energy property \cite[Eq.~(3.4)]{Gri06} implies that
	\begin{equation} \label{eq:af}
	\limsup_{M \to \infty}\frac{1}{M}\log\phi_{\TFK_{N,M}}(\Up(\om) \geq r)  
	= \limsup_{M\to\infty}\frac1M\log \phi_{\bbH_{N,M}}^{0}[V(x_1,\dots,x_r)].
	\end{equation}
	Write $\mathsf C_{x_j}$ for the cluster of the point $x_j$. Then, for any $j \geq 1$, an exploration argument similar to that of Lemma~\ref{lem:no_loop_density} (and therefore omitted\footnote{It involves again the domain Markov property \cite[Lem.~4.13]{Gri06} and the comparison between boundary conditions \cite[Lem.~4.14]{Gri06}.}) implies
	\begin{align}\nonumber
		\phi_{\bbH_{N,M}}^{0} [V(x_1,\dots,x_{j+1})\big| V(x_1,\dots, x_j)] &=\phi_{\bbH_{N,M}}^{0} [y_{j+1} \in \mathsf C_{x_{j+1}} \text{ and } x_1,\dots, x_{j} \notin \mathsf C_{x_{j+1}} \big| V(x_1,\dots, x_j)] \\
		&\leq \phi^{0}_{\bbL}(y_{j+1}\longleftrightarrow x_{j+1})\nonumber \\
		&\leq \phi^{0}_{\bbL}(0\longleftrightarrow y_0),\label{eq:Vxxx}
	\end{align}
	where $y_0 = (1,M-1)$.
	Applying this $r-1$ times yields
	\begin{align*}
		\phi_{\bbH_{N,M}}^{0} [V(x_1,\dots,x_r)] &\leq \phi^{0}_{\bbL}[0\longleftrightarrow y_0]^{r-1}\le {\bf c}^{r-1}\times\phi^{0}_{\bbL}[0\longleftrightarrow (0,M)]^{r-1}.
	\end{align*}
	(In the second inequality, we used the finite-energy  one last time).
	The conclusion follows from \eqref{eq:af}, the previous inequality, and the definition of $\xi(q)$. 
\end{proof}

\begin{remark}
Note that in  order to obtain~\eqref{eq:Vxxx}, we need to explore the cluster $\mathsf C_{x_1}$, i.e.~we need $j\ge1$. Indeed, we used that conditioned on $V(x_1,\dots,x_j)$, the boundary conditions in $\TFK_{N,M}\setminus(\mathsf C_{x_1}\cup\dots \cup\mathsf C_{x_j})$ are dominated by free boundary conditions at infinity. The fact that we do not obtain a bound on $\phi_{\bbH_{N,M}}^{0} [V(x_1)]$ (in this case, the boundary conditions are cylindrical and cannot be easily compared to the free boundary conditions at infinity) is the reason why we obtain $r-1$ instead of $r$ in \eqref{eq:cor_length2}. 
\end{remark}
\subsubsection{Proof of Theorem~\ref{thm:RCM}}

	Fix $q >4$. 
	By \cite{DumSidTas16}, for points 1 and 2 it is sufficient to show that $\xi(q)<\infty$.
	We therefore focus on point 3, that is we compute $\xi(q)^{-1}$ explicitly and show that it is equal to 
	$$R(q):=\lambda + 2 \sum_{k=1 }^\infty \tfrac{(-1)^k}k \tanh(k\lambda)>0,$$
	where $\la>0$ satisfies $e^\la+e^{-\la}=\sqrt q$. We will show that this quantity is positive and analyse its asymptotics in Section~\ref{sec:Fourier}.

	We will refer to the associated six-vertex model, with $c = \sqrt{2 + \sqrt q}$. 
	Write $Z_{\rm RC}(N,M)$ for the partition function of the random-cluster model with parameters $p_c,q$ on $\TFK_{N,M}$, that is 
	\begin{align*}
		Z_{\rm RC}(N,M) := \sum_{\om \in \Om_{\rm RC}} w_{\rm RC}(\om).
	\end{align*}
	\paragraph{Lower bound on the inverse correlation length}	Equation~\eqref{eq:cor_length1} may be rewritten as 
	\begin{align*}
		\xi(q)^{-1} \geq -\liminf_{N\to\infty}\liminf_{M \to \infty}\frac{1}{M}\log \frac{\sum_{\om: \, \Up(\om) = 1} w_{\rm RC}(\om)}{Z_{\rm RC}(N,M)}.
	\end{align*}
	Since all configurations with $\Up(\om) =1$ have exactly two non-retractable loops and no net, 
	Corollary~\ref{cor:ZFK_vs_Z6v}~\eqref{eq:corii} implies that the numerator above is smaller than 
	\begin{align*}
		\frac{\sqrt{q}}2 q^{\frac{MN}4}(1+\sqrt q)^{-MN} Z^{(1)}_{6V}(N,M).
	\end{align*}
	Furthermore, in light of Corollary~\ref{cor:ZFK_vs_Z6v}~\eqref{eq:cori}, 
	Lemma~\ref{lem:no_loop_density} may be rewritten as 
	\begin{align}\label{eq:Z6v_ZFK}
		\lim_{N\to \infty}\lim_{M\to \infty} \frac{1}M \log \frac{q^{\frac{MN}4}(1+\sqrt q)^{-MN}\ Z_{6V}(N,M)}{Z_{\rm RC}(N,M)} = 1.
	\end{align}
%
Therefore, we may  write
	\begin{align}\label{eq:xi_lower}
		\xi(q)^{-1} \geq -\liminf_{N\to\infty} \liminf_{M \to \infty}\frac{1}{M}\log \frac{Z^{(1)}_{6V}(N,M)}{Z_{6V}(N,M)} 
		= -\liminf_{N\to\infty}\ \log \frac{\La_1(N)}{\La_0(N)}\stackrel{\eqref{eq:aggg}}=R(q).
	\end{align}

	\paragraph{Upper bound on the inverse correlation length.}	
	For all $r \geq 2$,~\eqref{eq:cor_length2} may be written as
	\begin{align*}
		(r-1)\xi(q)^{-1} \leq -\limsup_{N\to\infty}\limsup_{M \to \infty}\frac{1}{M}\log \frac{\sum_{\om: \, \Up(\om) \geq r} w_{\rm RC}(\om)}{Z_{\rm RC}(N,M)}.
	\end{align*}
	Using Corollary~\ref{cor:ZFK_vs_Z6v}~\eqref{eq:coriii} and~\eqref{eq:Z6v_ZFK} again, we find 
	\begin{align*}
		(r-1)\xi(q)^{-1} 
		\leq -\limsup_{N\to\infty}\limsup_{M \to \infty}\ \frac{1}{M}\log \frac{Z^{(r)}_{6V}(N,M)}{Z_{6V}(N,M)}
		= -\limsup_{N\to\infty}\ \log \frac{\La_r(N)}{\La_0(N)}\stackrel{\eqref{eq:aggg}}=rR(q).
	\end{align*}
	The bound above being valid for all $r \geq 2$, one may divide by $r-1$ and take $r$ to infinity. 
	The resulting upper bound on $\xi(q)^{-1}$ matches the lower bound of~\eqref{eq:xi_lower}, and the theorem is proved. 

\begin{remark}\label{rmk:rr}  
	As mentioned before, \eqref{eq:xi_lower} should, in fact, be an equality. This would allow us to compute $\xi(q)^{-1}$ using nothing but the asymptotics of $\Lambda_0(N)$ and $\Lambda_1(N)$, and require no control of $\Lambda_r(N)$ for $r \geq 2$.
	However, since we did not manage to derive the reversed inequality of \eqref{eq:cor_length1} (and hence of \eqref{eq:xi_lower}) using only the random-cluster model, we used \eqref{eq:cor_length2}  and our control of $\Lambda_r(N),\, r \geq 2$ as an indirect route to the desired bound. 

   	In retrospect, it may be deduced from the Theorem~\ref{thm:RCM} that \eqref{eq:cor_length1} and \eqref{eq:xi_lower} are actually equalities. 
    We believe that proving the equality in \eqref{eq:cor_length1} using only the random-cluster model is an interesting question. 
\end{remark}

\subsection{From the random-cluster to the Potts model: proof of Theorem~\ref{thm:Potts}} \label{sec:Potts}
Below, we consider the Potts and random-cluster models on the standard lattice $\bbZ^2$; contrarily to previous sections, no reference to the rotated lattice is used. In particular, $\phi^0_{\bbZ^2,p,q}$ and $\phi^1_{\bbZ^2,p,q}$ are infinite-volume measures on $\bbZ^2$ (like in the introduction, and unlike in the previous section).

 The results for the Potts model can be obtained from those for the random-cluster model via a classical coupling, see \cite{EdwSok88,Gri06}. We describe the consequences of this coupling in the theorem below; for a proof, see the references. 
 In the next statement, the operation of {\em attributing a spin $s\in\{1,\dots,q\}$ to a set $S$ of vertices} means that we fix $\sigma_x=s$ for every $x\in S$.
\begin{theorem}\label{thm:coupling}
Fix $\beta>0$ and an integer $q\ge2$. Set $p=1-e^{-\beta}$. 
\begin{itemize}[noitemsep,nolistsep]
\item Consider $\omega$ with law $\phi_{\bbZ^2,p,q}^0$. Then, the law of $\sigma\in\{1,\dots,q\}^{\mathbb Z^2}$ obtained by attributing independently and uniformly a spin in $\{1,\dots,q\}$ to each cluster of $\omega$ is $\mu_\beta^0$.
\item  Fix $i\in\{1,\dots,q\}$ and consider $\omega$ with law $\phi_{\bbZ^2,p,q}^1$. Then, the law of $\sigma\in\{1,\dots,q\}^{\mathbb Z^2}$ obtained by attributing independently and uniformly a spin in $\{1,\dots,q\}$ to each {\em finite} cluster of $\omega$, and spin $i$ to the {\em infinite} clusters\footnote{There is in fact a unique one almost surely, see \cite[Section 4.4.]{Gri06}.} of $\omega$ is $\mu_\beta^i$.
\end{itemize}
\end{theorem}

Theorem~\ref{thm:coupling} implies immediately the following facts.
\begin{enumerate}[noitemsep,nolistsep]
\item The critical inverse-temperature of the Potts model and the critical parameter of random-cluster model are related by the formula $p_c=1-e^{\beta_c}$.
\item For any $i\in\{1,\dots,q\}$,
\begin{equation*}
\mu_{\beta}^i[\sigma_0=i]=\tfrac1q+\phi_{\bbZ^2,p,q}^1[0\text{ is in an infinite cluster}].
\end{equation*}
\item For any $x,y\in \mathbb Z^2$,
\begin{equation*}
\mu_\beta^0[\sigma_x=\sigma_y]=\tfrac1q+\phi^0_{\bbZ^2,p,q}[x\text{ and }y\text{ are in the same cluster}].
\end{equation*}
\end{enumerate}

With these properties at hand, it is elementary to deduce Theorem~\ref{thm:Potts} from Theorem~\ref{thm:RCM}.
Theorem~\ref{thm:Potts} (2)  follows directly from items 1.~and 2.~above combined with (2) of Theorem~\ref{thm:RCM}. Theorem~\ref{thm:Potts} (3) follows from item 3.~above and the expression for $\xi(q)$ obtained in Theorem~\ref{thm:RCM}.

For Theorem~\ref{thm:Potts} (1),
it is well-known (see for instance results in \cite{Gri06}) that a Gibbs measure is extremal if and only if it is ergodic. 
Furthermore, the measures $\phi^0_{\bbZ^2,p,q}$ and $\phi^1_{\bbZ^2,p,q}$ are ergodic for any value of $p\in[0,1]$. 
Since there exists no infinite cluster $\phi^0_{\bbZ^2,p_c,q}$-almost surely (by (3) of Theorem~\ref{thm:RCM}), the construction of $\mu^0_{\beta_c}$ from $\phi^0_{\bbZ^2,p_c,q}$ described in Theorem~\ref{thm:coupling}  implies that $\mu^0_{\beta_c}$ is ergodic as well. 
In the same way, each measure $\mu^i_{\beta_c}$, $i=1,\dots, q$, may be shown to be ergodic (here the existence of an infinite cluster is not problematic, since it is given the fixed spin $i$). 
By Theorem~\ref{thm:Potts} (2), the measures $\mu^i_{\beta_c}$ induce different distributions for the spin of any given vertex, hence they are all distinct.

%
\section{Fourier computations}\label{sec:Fourier}

In this section, we gather the computations of certain Fourier-analytic identities used throughout the paper.

\paragraph{Evaluation of the Fourier coefficients of $\Xi_{\lambda}$ and $R$.} Let $m \geq 0$ and consider the contour integral
\[
\frac{1}{2\pi}\int_{C_N} \frac{ \sinh (\lambda) e^{-i m z} }{\cosh (\lambda) - \cos(z)}dz, 
\]
where $C_N$ is the boundary of $[-\pi, \pi] + i [-N, 0]$, oriented clockwise. 
As $N$ goes to infinity, this integral goes to $\hat \Xi_\lambda(m)$. 
Since the only residues of the integrand in the interior of $C_N$ occur at $- i \lambda$ , we conclude that 
\[
\hat \Xi_\lambda(m) = e^{- \lambda m} \quad \quad m \geq 0.
\]
If $m < 0$, we integrate around $C'_N$, the boundary of $[-\pi, \pi] + i [0,N]$, oriented counterclockwise. 
The residue will now be at $i \lambda$, and 
\[
\hat \Xi_\lambda(m) = e^{\lambda m} \quad \quad m < 0.
\]
Via \eqref{eq:rfourier}, this implies 
\begin{equation}\label{eq:Rhat}
\hat R(m) = \frac{e^{-\lambda |m|}}{1 + e^{-2 \lambda |m|}} = \frac{1}{2 \cosh(\lambda m)}. 
\end{equation}

\paragraph{Evaluation of the Fourier coefficients of $\Psi$ and $T$.}
To evaluate $\hat{\Psi}$, we first note that $k(\alpha)$ is an odd function, and $\Theta$ is anti-symmetric, meaning $\Psi$ is an odd function and $\hat{\Psi}(0) = 0$. As a consequence, \eqref{eq:rfourier} implies $\hat T(0) = 0$. 

For an integer $m \neq 0$, we first replace $\Theta(k(\alpha),\pi) + \Theta(k(\alpha),-\pi)$ with the equivalent expression $2 [\Theta(k(\alpha),\pi) - \pi]$ (using the fact that $\Theta(x,\pi) = \Theta(x,-\pi) + 2 \pi)$. Then, using integration by parts, we find
\begin{align*}
	\hat \Psi (m) 
	& = \frac{1}{2\pi} \int_{-\pi}^\pi [\Theta(k(\alpha),\pi) - \pi] e^{-im\alpha} d \alpha \\ 
	& = \frac{[\Theta(\pi, \pi) - \Theta(-\pi, \pi)] (-1)^m }{-2\pi i m}
		+\frac{1}{2\pi i m} \int_{-\pi}^{\pi} \frac{d}{d\alpha} \Theta(k(\alpha),\pi) e^{-im\alpha} d \alpha  \\ 
	& = \frac{(-1)^m }{\icomp m} 
		-\frac{1}{ 2\pi i m} \int_{-\pi}^{\pi} \Xi_{2\lambda}(\alpha - \pi) e^{-im\alpha} d \alpha  \\
	& = \frac{(-1)^{m}}{ i m} \left(1 - \hat{\Xi}_{2\lambda}(m)\right), 
\end{align*}
where we used $\Theta(\pi,\pi) - \Theta(-\pi,\pi) = -2\pi$, the change of variable $u = \alpha - \pi$ and the periodicity of $\Xi_{2\la}$ to show that the integral in the penultimate line is equal to $2 \pi (-1)^m \hat \Xi_{2\lambda}(m)$. Thus,
\[
\hat T(m) = \frac{(-1)^m \left(1 - e^{-2 \lambda |m|}\right)}{i m\left( 1 + e^{-2 \lambda |m|}\right)} 
= \frac{(-1)^m}{\icomp m} \tanh (\lambda |m|).
\]
\paragraph{Computations of $R$ and $T$.} We start with $T$. Pairing the terms for $\pm m$, we find%
\footnote{In the formula, the series is not absolutely convergent, however, $\sum_{m=1}^N (-1)^m \tanh (\lambda m) \sin (m\alpha)/m$ converges as $N\to \infty$, and we will consider this as the limit.}
\[
T(\alpha) 
= 2\sum_{m >0} \frac{(-1)^m}{m} \tanh (\lambda m) \left( \frac{e^{im \alpha} - e^{-im\alpha}}{2 i}\right) 
= 2\sum_{m >0} \frac{(-1)^m}{m} \tanh (\lambda m) \sin (m\alpha).
\]
We now turn to $R$. We will show that it is equal to the sum 
\[
\mathcal{R}(\alpha) := \frac{\pi}{2\lambda}\sum_{r \in \mathbb{Z}} \frac{1}{ \cosh[\pi (2 \pi r +\alpha)/(2 \lambda)]} 
\]
by showing that the two have the same Fourier coefficients. By direct computation and the Dominated Convergence Theorem, 
\begin{align*}
\hat{\mathcal{R}}(m) & = \frac{1}{4\la} \sum_{r \in \mathbb{Z}} \int_{-\pi}^{\pi}  \frac{e^{-i m \alpha} d \alpha}{ \cosh[\pi (2 \pi r +\alpha)/(2 \lambda)]}   \\ 
& = \frac{1}{4\la} \int_{-\infty}^{\infty } \frac{e^{-i m \alpha} d \alpha}{\cosh(\pi \alpha/2 \lambda)}
\end{align*}
using the $2\pi$ periodicity of the numerator. 
Observe that the hyperbolic secant function can be written as a {\em continuous} Fourier transform:
\[
\frac{1}{\cosh(\lambda m)} = \frac{1}{2\lambda} \displaystyle \int_{-\infty}^\infty \frac{e^{-i m \alpha} d\alpha}{ \cosh( \pi \alpha/2 \lambda)}.
\]
This concludes the proof since  $\hat R(m)=\tfrac1{2\cosh(\lambda m)}$ by \eqref{eq:Rhat}. 

\paragraph{Computation of the integral on the right-hand side of~\eqref{eq:ag}.}
The change of variable $x = k(\alpha)$ and some elementary algebraic manipulations give
\[
\int_{-\pi}^\pi \log \left| M \left(e^{i x} \right) \right| \rho(x) dx = \frac{1}{2\pi} \int_{-\pi}^\pi P(\alpha) R(\alpha) d \alpha,
\]
with
\[
P(\alpha) : = \log | M(e^{ik(\alpha)})| = \tfrac{1}{2} \log \left(\frac{\cosh(2 \lambda) - \cos(\alpha)}{1 - \cos(\alpha)} \right) = \int_0^\lambda \Xi_{2t}(\alpha) dt.
\]
The final equality may be checked by noticing that the two sides have equal derivatives and are both equal to $0$ when $\la = 0$. 
We note that, even though $P(\alpha)$ is not a bounded function, its singularity at $\alpha =0$ is logarithmic, and hence it is in $L^2([-\pi,\pi])$. Thus, we can use Fubini's Theorem to deduce that
\begin{equation}\label{eq:Phat}
	\hat{P}(m) = \int_{0}^\lambda e^{- 2 t |m|} d t
	=\begin{cases} \ \ \ \ \ \ \ \lambda \quad & \text{if } m = 0, \\ \frac{1 - \exp( - 2\lambda |m|)}{2|m|} & \text{if } m \neq 0. \end{cases} 
\end{equation}
Finally, Parseval's Theorem implies that 
\[
\frac{1}{2\pi} \int_{-\pi}^\pi P(\alpha) R(\alpha)d \alpha  
= \sum_{m \in \mathbb{Z}} \hat P(m) \hat R(-m) 
= \frac{\lambda}{2} + \sum_{m>0}\frac{ e^{- m \lambda} \tanh(\lambda m)}{m} 
\]
using~\eqref{eq:Rhat} in the final equality.

\paragraph{Computation of the integral on the right-hand side of~\eqref{eq:EigenvalueRatio} and~\eqref{eq:agg}.}

We begin our analysis of the second integral by recalling \eqref{eq:f_N_control}, which implies the existence of $C$ such that $|\tau (x)| < C |x|$ for all $x\in [-\pi,\pi]$. 
Thus, although $\ell'(x)$ grows as $1/|x|$ near the origin, the integrand is uniformly bounded.
Using the Dominated Convergence Theorem\footnote{In the formula below, the series in the right-hand side is not absolutely convergent. However, if terms are paired (each odd term with the succeeding even one) the resulting series becomes absolutely convergent. This observation is used here and below.}
 and the explicit computation of $\tau$ in Proposition~\ref{prop:rhoProp}, 
we find
\[
\int_{-\pi}^\pi  \ell'(x) \tau (x) dx 
= \int_{-\pi}^\pi P'(\alpha)\tau(k(\alpha))d\alpha
= \sum_{m >0} \frac{(-1)^m \tanh (\lambda m)}{m} \left[ \frac{1}{\pi}\int_{-\pi}^\pi P'(\alpha) \sin(m\alpha)  d\alpha \right].
\]
Calculating the integrals on the right-hand side is a simple case of integration by parts:
\begin{align*}
	\frac{1}{\pi}\int_{-\pi}^{\pi} P'(\alpha)\sin(m\alpha)d \alpha
	&= \left.\frac{P(\alpha)\sin(m \alpha)}{\pi}\right|_{-\pi}^{\pi} 
		- \frac{m}{\pi}\int_{-\pi}^{\pi}P(\alpha)\cos(m\alpha)d\alpha\\
	&=-m [\hat P(m)+\hat P(-m)]\\
	&=e^{-2\lambda m}-1,
\end{align*}
where we use our earlier computation~\eqref{eq:Phat} for the final line.  Substituting this in \eqref{eq:EigenvalueRatio} yields%
\begin{align*}
	\lim_{N\to \infty} \log \frac{\Lambda_r(N)}{\Lambda_0(N)}  
	= -r\cdot\Big[  \log|\De| - \sum_{m >0} \frac{(-1)^m}m \tanh (\lambda m) (e^{-2\lambda m}-1)\Big].
\end{align*}
By expanding $\log |\De| =\log \cosh(\lambda)$ in powers of $e^{-\lambda}$ and manipulating the result algebraically, we find that
\[
\log |\De| = \lambda  - \sum_{m >0} \frac{(-1)^m (e^{-2\lambda m} -1)}{m}.
\]
This directly implies
\begin{equation}\label{eq:hg}
	\log|\De| - \sum_{m >0} \frac{(-1)^m}m \tanh (\lambda m) (e^{-2\lambda m}-1) 
	= \lambda + 2 \sum_{m >0 } \frac{(-1)^m}m \tanh (m\lambda).
\end{equation}

\paragraph{Proof of~\eqref{eq:alternative}.}
We wish to show that 
\begin{equation}\label{eq:hgg}
\lambda + 2 \sum_{m \ge1 } \frac{(-1)^m}m \tanh (m\lambda)=\sum_{m \geq 0} \frac{4}{(2m + 1)\sinh \left[\pi^2 (2m +1)/(2 \lambda) \right]}.
\end{equation}
Let $C_N$ be the boundary of the rectangle $[-(2N +1)/2, (2N +1)/2]+ i[-\pi N/\lambda, \pi N/ \lambda]$, oriented counterclockwise, and consider 
\[
\mathcal{I}_N := \displaystyle \int_{C_N} \frac{\pi \tanh (\lambda z) dz}{z \sin (\pi z)}. 
\]
The integrand has a simple pole at every integer $m$ and at $i \pi (2r+1)/(2\lambda)$ for every integer $r$. A straightforward computation shows that the residues of the integrand at the natural numbers are:
\[
\text{Res} \left(\frac{\pi \tanh (\lambda z)}{z \sin (\pi z)}, m \right) 
= \begin{cases} \frac{ \tanh (\lambda m)}{\cos(\pi m) m} & m \neq 0, \\ \ \ \ \ \lambda & m=0. \end{cases}
\] 
Summing over $m \in [-N,N] \cap \bbZ$ gives the partial sums of the right-hand side of~\eqref{eq:hgg}. Meanwhile, 
\[
\text{Res} \left(\frac{\pi \tanh (\lambda z)}{z \sin (\pi z)}, i \pi (2m+1)/(2\lambda) \right) = \frac{-2}{(2 m+1) \sinh[\pi^2 (2m+1)/(2 \lambda)]} .
\] 
The hyperbolic tangent is bounded away for its poles (and therefore on $C_N$), so  we may deduce that, for some uniform constant $c_0$,
\[
|\mathcal{I}_N | \leq \frac{c_0}{N} \left[\int_{-\pi N/\lambda}^{\pi N/ \lambda} \frac{dt}{\cosh(\pi t)} + \int_{-(2N +1)/2}^{(2N +1)/2} \frac{dt}{|\sin (i\pi^2/\lambda + t)|}\right].
 \]
Both integrals are uniformly finite in $N$, hence $\mathcal{I}_N$ converges to zero. As a consequence, the sum of residues of the integrand converges to zero. 
Using the residues computed above, this implies\footnote{We obtain explicitly 
$\lambda + 2 \sum_{m  = 1}^N \frac{(-1)^m}m \tanh (m\lambda) - \sum_{m =0}^N \frac{4}{(2m + 1)\sinh \left[\pi^2 (2m +1)/(2 \lambda) \right]} \to 0$ as $N \to \infty$. 
}~\eqref{eq:hgg}. 

Upon inspection of the right-hand side of~\eqref{eq:hgg}, we observe that the quantity in the equation is strictly positive whenever $\lambda >0$. 
The asymptotic behaviour of~\eqref{eq:hgg} as $\De$ tends to $-1$ (corresponding to $2\lambda\sim \sqrt{q-4}$ tending to 0) is governed by the first term or the right-hand side, namely $\frac{4}{\sinh \left(\pi^2/(2 \lambda) \right)}\sim 8e^{-\pi^2/(2\lambda)}$.

\bibliographystyle{siam}
\bibliography{biblicomplete}

\begin{thebibliography}{10}

\bibitem{Bax78}
{\sc R.~J. Baxter}, {\em Solvable eight-vertex model on an arbitrary planar
  lattice}, Philos. Trans. Roy. Soc. London Ser. A, 289 (1978), pp.~315--346.

\bibitem{Bax89}
{\sc R.~J. Baxter}, {\em Exactly solved models in statistical mechanics},
  Academic Press Inc. [Harcourt Brace Jovanovich Publishers], London, 1989.
\newblock Reprint of the 1982 original.

\bibitem{BefDum12}
{\sc V.~Beffara and H.~Duminil-Copin}, {\em The self-dual point of the
  two-dimensional random-cluster model is critical for {$q\geq 1$}}, Probab.
  Theory Related Fields, 153 (2012), pp.~511--542.

\bibitem{Bethe31}
{\sc H.~Bethe}, {\em {Zur Theorie der Metalle I. Eigenwerte und Eigenfunktionen
  der Hnearen Atomkette}}, Zeitschrift f{\"u}r Physik, 71 (1931), pp.~205--226.

\bibitem{BorCorGor16}
{\sc A.~Borodin, I.~Corwin, and V.~Gorin}, {\em Stochastic six-vertex model},
  Duke Math. J., 165 (2016), pp.~563--624.

\bibitem{BufWal93}
{\sc E.~Buffenoir and S.~Wallon}, {\em The correlation length of the {P}otts
  model at the first-order transition point}, Journal of Physics A:
  Mathematical and General, 26 (1993), p.~3045.

\bibitem{Dum13}
{\sc H.~Duminil-Copin}, {\em Parafermionic observables and their applications
  to planar statistical physics models}, vol.~25 of Ensaios Matematicos,
  Brazilian Mathematical Society, 2013.

\bibitem{BetheAnsatz1}
{\sc H.~Duminil-Copin, M.~Gagnebin, M.~Harel, I.~Manolescu, and V.~Tassion},
  {\em The {B}ethe ansatz for the six-vertex and {XXZ} models: an exposition},
  (2016).
\newblock Preprint, arXiv:1611.09909.

\bibitem{DumMan14}
{\sc H.~Duminil-Copin and I.~Manolescu}, {\em The phase transitions of the
  planar random-cluster and {P}otts models with $q \geq 1$ are sharp},
  Probability Theory and Related Fields, 164 (2016), pp.~865--892.
\newblock arXiv:1409.3748.

\bibitem{DCRT16}
{\sc H.~Duminil-Copin, A.~Raoufi, and V.~Tassion}, {\em A new computation of
  the critical point for the planar random-cluster model with $ q \ge1$},
  Preprint, arXiv:1604.03702,  (2016).

\bibitem{DumRaoTas17}
{\sc H.~Duminil-Copin, A.~Raoufi, and V.~Tassion}, {\em Sharp phase transition
  for the random-cluster and {P}otts models via decision trees},  (2017).
\newblock Preprint, arXiv:1705.03104.

\bibitem{DumSidTas16}
{\sc H.~Duminil-Copin, V.~Sidoravicius, and V.~Tassion}, {\em Continuity of the
  phase transition for planar {P}otts models with $1\le q\le 4$},
  Communications in Mathematical Physics,  (2016), pp.~1--61.

\bibitem{EdwSok88}
{\sc R.~G. Edwards and A.~D. Sokal}, {\em Generalization of the
  {F}ortuin-{K}asteleyn-{S}wendsen-{W}ang representation and {M}onte {C}arlo
  algorithm}, Phys. Rev. D (3), 38 (1988), pp.~2009--2012.

\bibitem{For70}
{\sc C.~M. Fortuin}, {\em On the {R}andom-{C}luster model}, doctoral thesis,
  University of Leiden, 1971.

\bibitem{ForKas72}
{\sc C.~M. Fortuin and P.~W. Kasteleyn}, {\em On the random-cluster model. {I}.
  {I}ntroduction and relation to other models}, Physica, 57 (1972),
  pp.~536--564.

\bibitem{Fulmek}
{\sc M.~Fulmek}, {\em Nonintersecting lattice paths on the cylinder},
  S\'eminaire Lotharingien de Combinatoire, 52 (2004).
\newblock Article B52b.

\bibitem{Gol05}
{\sc P.~S. Goldbaum}, {\em Existence of solutions to the {B}ethe ansatz
  equations for the 1{D} {H}ubbard model: Finite lattice and thermodynamic
  limit}, Communications in Mathematical Physics, 258 (2005), pp.~317--337.

\bibitem{Gri06}
{\sc G.~Grimmett}, {\em The random-cluster model}, vol.~333 of Grundlehren der
  Mathematischen Wissenschaften, Springer-Verlag, Berlin, 2006.

\bibitem{KorZinn00}
{\sc V.~Korepin and P.~Zinn-Justin}, {\em Thermodynamic limit of the six-vertex
  model with domain wall boundary conditions}, Journal of Physics A:
  Mathematical and General, 33 (2000), p.~7053.

\bibitem{Koz15}
{\sc K.~Kozlowski}, {\em On condensation properties of {B}ethe roots associated
  with the {XXZ} chain}, arXiv:1508.05741,  (2015).

\bibitem{Pot52}
{\sc R.~B. Potts}, {\em Some generalized order-disorder transformations}, in
  Proceedings of the Cambridge Philosophical Society, vol.~48, Cambridge Univ
  Press, 1952, pp.~106--109.

\bibitem{Resh10}
{\sc N.~{Reshetikhin}}, {\em {Lectures on the integrability of the 6-vertex
  model}}, arXiv1010.5031,  (2010).

\bibitem{TempLieb71}
{\sc H.~N.~V. Temperley and E.~H. Lieb}, {\em Relations between the
  {\textquoteright}percolation{\textquoteright} and
  {\textquoteright}colouring{\textquoteright} problem and other
  graph-theoretical problems associated with regular planar lattices: Some
  exact results for the {\textquoteright}percolation{\textquoteright} problem},
  Proceedings of the Royal Society of London A: Mathematical, Physical and
  Engineering Sciences, 322 (1971), pp.~251--280.

\bibitem{Wu82}
{\sc F.~Y. Wu}, {\em The {P}otts model}, Rev. Modern Phys., 54 (1982),
  pp.~235--268.

\bibitem{YangYang66}
{\sc C.~N. Yang and C.~P. Yang}, {\em One-dimensional chain of anisotropic
  spin-spin interactions. {I}. {P}roof of {B}ethe's hypothesis for ground state
  in a finite system}, Phys. Rev., 150 (1966), pp.~321--327.

\bibitem{Zinn00}
{\sc P.~Zinn-Justin}, {\em Six-vertex model with domain wall boundary
  conditions and one-matrix model}, Phys. Rev. E, 62 (2000), pp.~3411--3418.

\end{thebibliography}

\end{document}